\documentclass[a4paper]{article}
\title{Double quasi-Poisson algebras are pre-Calabi-Yau}
\author{David Fern\'andez and Estanislao Herscovich}
\date{}

\usepackage{amsmath,amsthm,amsfonts,amssymb,stmaryrd,mathtools}
\mathtoolsset{showonlyrefs=true}
\usepackage{enumitem}
\usepackage{latexsym}
\usepackage{fullpage}
\usepackage{titlesec}
\usepackage{everypage}
\usepackage{mathrsfs}
\usepackage{mxedruli}
\usepackage{a4,xcolor,palatino,fancyhdr}
\colorlet{mycyan}{cyan!40!gray}
\colorlet{myblue}{blue!40!gray}
\colorlet{myred}{red!40!gray}
\colorlet{mygreen}{green!20!gray}
\definecolor{ultramarine}{RGB}{0,32,96}
\definecolor{light-gray}{gray}{0.8}
\colorlet{myultramarine}{ultramarine!20!gray}
\usepackage{graphicx}
\usepackage{float}  
\usepackage[small, bf, margin=90pt, tableposition=bottom]{caption}
\usepackage[utf8]{inputenc}
\usepackage[T1]{fontenc}
\usepackage[english,french]{babel}
\usepackage[colorlinks=true,citecolor=mygreen, linkcolor=mygreen, urlcolor=myultramarine,breaklinks, naturalnames]{hyperref}
\usepackage[
]{amsrefs}

\input{xy}
\xyoption{all}
\xyoption{poly}
\usepackage[all]{xy}

\newtheorem{theorem}{Theorem}[section]
\newtheorem{proposition}[theorem]{Proposition}
\newtheorem{definition}[theorem]{Definition}
\newtheorem{lemma}[theorem]{Lemma}

\newtheorem{remark}[theorem]{Remark}

\newtheorem{fact}[theorem]{Fact}

\DeclareFontFamily{U}{BOONDOX-calo}{\skewchar\font=45 }
\DeclareFontShape{U}{BOONDOX-calo}{m}{n}{
  <-> s*[1.05] BOONDOX-r-calo}{}
\DeclareFontShape{U}{BOONDOX-calo}{b}{n}{
  <-> s*[1.05] BOONDOX-b-calo}{}
\DeclareMathAlphabet{\mathcalboondox}{U}{BOONDOX-calo}{m}{n}
\SetMathAlphabet{\mathcalboondox}{bold}{U}{BOONDOX-calo}{b}{n}
\DeclareMathAlphabet{\mathbcalboondox}{U}{BOONDOX-calo}{b}{n}

\newcommand{\lr}[1]{
  \{\mkern-6mu\{#1\}\mkern-6mu\}}

\numberwithin{equation}{section}

\newcommand{\gan}{\text{{\mxedc g}}}

\def\para{{\tau}}

\newcommand\ZZ{{\mathbb{Z}}}
\newcommand\NN{{\mathbb{N}}}
\newcommand\RR{{\mathbb{R}}}
\newcommand{\kk}{{\Bbbk}} 
\newcommand{\D}{\operatorname{\mathbb{D}er}}
\newcommand{\Der}{\operatorname{Der}}

\newcommand{\out}{\operatorname{out}}

\newcommand{\E}{\operatorname{E}}
\newcommand{\GL}{\operatorname{GL}}
\newcommand{\Rep}{\operatorname{Rep}}

\def\Hom{{\mathrm {Hom}}}

\titlespacing*{\section}
{0pt}{1.25ex plus 1ex minus .2ex}{1.25ex plus .2ex}

\begin{document}

\maketitle
                                                     
\hrulefill
\selectlanguage{english}
\begin{abstract} 
In this article we prove that double quasi-Poisson algebras, which are noncommutative analogues of quasi-Poisson manifolds, naturally give rise to pre-Calabi-Yau algebras. 
This extends one of the main results in \cite{IK17}, where a correspondence between certain pre-Calabi-Yau algebras and double Poisson algebras was found (see also \cite{IKV19}, \cite{IK19} and \cite{FH19}). 
However, a major difference between the pre-Calabi-Yau algebra constructed in the mentioned articles and the one constructed in this work is that the higher multiplications indexed by even integers of the underlying $A_{\infty}$-algebra structure of the pre-Calabi-Yau algebra associated to a double quasi-Poisson algebra do not vanish, but are given by nice cyclic expressions multiplied by explicitly determined coefficients involving the Bernoulli numbers.
\end{abstract}



\hrulefill

\section{Introduction}

Quasi-Hamiltonian and quasi-Poisson geometry were introduced in \cite{AMM} and \cite{AKSM}, respectively, to give an alternative finite-dimensional construction of some symplectic moduli spaces originally obtained by symplectic reduction from infinite-dimensional symplectic manifolds; for instance, the moduli space of gauge equivalence classes of flat connections on a principal bundle over a compact Riemann surface. 
On the other hand, interested in the existence of a natural Poisson structure on multiplicative quiver varieties, M. Van den Bergh developed in \cite{VdB} a noncommutative version of quasi-Poisson geometry, later introducing in \cite{q-VdB} the notion of \textbf{\textcolor{myblue}{quasi-bisymplectic algebras}}. 
Roughly speaking, given an associative (but not necessarily commutative) algebra $A$, a \textbf{\textcolor{myblue}{double quasi-Poisson bracket}} is a linear map ${\lr{ \hskip 0.6mm , } : A^{\otimes 2} \rightarrow A^{\otimes 2}}$ satisfying skew-symmetry, the Leibniz identity and a modified double Jacobi identity (see Definition \ref{double-quasi-poisson-def}). 
This version of the Jacobi identity constitutes the main feature of double quasi-Poisson algebras. 
As Van den Bergh explained in \cite{VdB}, even though he was mostly interested in double quasi-Poisson algebras, 
he also needed to introduce and study the most fundamental notion of \textbf{\textcolor{myblue}{double Poisson algebra}} (see Definition \ref{def-double-Poisson}), which is the noncommutative analogue of a Poisson structure. 
In this sense, double Poisson algebras and double quasi-Poisson algebras are the mainstays of the approach to noncommutative Poisson geometry initiated in \cite{VdB}.
 
Since then, double quasi-Poisson algebras have been used in several different contexts.
For instance, in \cite{MT} G. Massuyeau and V. Turaev defined a canonical double quasi-Poisson bracket on the fundamental group $\pi$ of an oriented surface with base point on its boundary, which induces the well-known \textbf{\textcolor{myblue}{Fock-Rosly Poisson structure}} on $\Hom(\pi,G)/G$. 
On the other hand, motivated by the use of quasi-Hamiltonian reduction in the theory of integrable systems, O. Chalykh and M. Fairon used certain double quasi-Poisson structures on path algebras of quivers to give a new perspective on variants of the \textbf{\textcolor{myblue}{Ruijsenaars-Schneider model}} (see \cite{CF}). 
In \cite{Art}, S. Artomonov found an interesting link between double quasi-Poisson structures and (noncommutative) cluster algebras, introducing double quasi-Poisson brackets on categories. 
Finally, double quasi-Poisson brackets also appeared in the setting of the Kashiwara-Vergne problem in Lie theory (see \cite{AKKN}), and in P. Boalch's higher fission spaces to build spaces of Stokes representations (see \cite{Boalch}).

On the other hand, in the setting of Topological Quantum Field Theory, M. Kontsevich and Y. Vlassopoulos \cite{KV18} introduced the notion of \textbf{\textcolor{myblue}{$d$-pre-Calabi-Yau algebras}}.
These structures (or equivalent ones) have appeared in other works under different names, such as $V_{\infty}$-algebras in \cite{TZ16}, $A_{\infty}$-algebras with boundary in \cite{Sei12}, noncommutative divisors in \cite{Sei17}, Rmk. 2.11, or weak Calabi-Yau structures in \cites{Ko13,Ye18,KPS17}. 
Roughly speaking, a pre-Calabi-Yau algebra is a noncommutative analogue of a solution to the Maurer-Cartan equation for the Schouten bracket on polyvector fields; this idea has been formalized and successfully applied in \cite{IK19} and \cite{Ye18}.
Equivalently (see \cite{IK17}), pre-Calabi-Yau algebras are (quasi-)cyclic $A_{\infty}$-algebra structures on $A \oplus A^{\#}[d-1]$ for the natural bilinear form of degree $d-1$ such that $A$ is an $A_{\infty}$-subalgebra (see Definition \ref{definition-pre-d-CY}, for the case of $d=0$). 
Quite strikingly, N. Iyudu and M. Kontsevich showed in \cite{IK17} (see also \cite{IKV19}) that there exists an explicit one-to-one correspondence between a particular class of pre-Calabi-Yau algebras and that of non-graded double Poisson algebras. 
For a very interesting conceptual explanation of this correspondence in terms of the higher cyclic Hochschild cohomology and its generalized necklace bracket, see \cite{IK19}.
On the other hand, the bijection in \cite{IK17} was extended to the differential graded setting in \cite{FH19}, Thm. 5.1, 
and to double $P_{\infty}$-algebras in \cite{FH19}, Thm. 6.3.

In the light of the previous results, it is natural to ask whether there exists a link between pre-Calabi-Yau algebras and double quasi-Poisson algebras. 
The main goal of the article is to show that the answer is affirmative: double quasi-Poisson algebras do naturally give rise to pre-Calabi-Yau algebras. 
However, there is a major difference between the pre-Calabi-Yau algebra structure coming from a double quasi-Poisson algebra and the one constructed by N. Iyudu and M. Kontsevich, which is induced from a double Poisson algebra. 
Indeed, all the higher multiplications $m_{n}$ indexed by even integers $n$ of the underlying $A_{\infty}$-algebra structure of the pre-Calabi-Yau algebra associated with a double quasi-Poisson algebra do not vanish, but they are given by nice cyclic expressions multiplied by explicitly determined coefficients involving the Bernoulli numbers (see \eqref{eq:newcij}, \eqref{eq:mu2} and \eqref{eq:mun}). 
In fact, the multiplication $m_4$ of the pre-Calabi-Yau structure codifies the nonhomogeneous term in the modified double Jacobi identity (\textit{i.e.} the right-hand side of \eqref{eq:double-quasi-poisson-def}),  
in marked contrast with the situation considered in \cite{IKV19}, Thm. 4.2, or \cite{IK19}, Thm. 4.1, where the authors assume that $m_{4}$ vanishes.  
This has of course major consequences: a double quasi-Poisson algebra induces a quasi-Poisson algebra structure on the representation scheme, whereas a double Poisson algebra induces a plain Poisson algebra on the representation scheme.

The contents of the article are as follows. 
After fixing some notation and conventions, as well as recalling some sign and numeric conventions in Section \ref{section:preliminaries}, we briefly review in Section \ref{section:quasi-Poisson} the basic notions of quasi-Poisson manifolds, introduced by A. Alekseev, Y. Kosmann-Schwarzbach and E. Meinrenken in \cite{AKSM}, as well as their noncommutative versions, introduced by M. Van den Bergh in \cite{VdB}. 
In Section \ref{section:cyclic-pre-CY} we recall the basic definition of (quasi-cyclic) $A_{\infty}$-algebras, which are well-known in the literature, and of pre-Calabi-Yau algebras, that we mentioned previously. 
Finally, in Section \ref{section:core} we state and prove our main result: any double quasi-Poisson algebra gives rise to a pre-Calabi-Yau algebra (see Theorem \ref{theorem:Main}), extending the result proved by N. Iyudu and M. Kontsevich (see \cite{IK17}, Thm. 4.2, but also \cite{IKV19}, Thm. 4.2). 
The proof consists in verifying that some involved expressions for the higher multiplications do satisfy all the Stasheff identities.  
This is done by handling the cases given by the many different elements on which they can be evaluated. 

Some words concerning the main formulas in the paper are in order. 
The reader might ask how we found the somewhat curious expressions \eqref{eq:mu2}-\eqref{eq:mun} for the higher multiplications of the pre-Calabi-Yau algebra associated to a double quasi-Poisson algebra, which involve the unexpected constants \eqref{eq:newcij}. 
Unfortunately, we do not have any structural explanation so far: both expressions \eqref{eq:mu2}-\eqref{eq:mun} were found after computing a ridiculous amount of Stasheff identities and recursively proposing compatible higher multiplications to verify them. 
Once the main structure of the higher multiplications was determined, the expressions of the constants \eqref{eq:newcij} were obtained after recursively computing a large number of them by means of \eqref{eq:Cgen} using GAP and isolating a pattern, which involved solving recursive quadratic equations. 
The whole process is not explained in the article and it was rather tedious. 
It remains to know if a more conceptual approach can be achieved, as for instance in the case of the Lawrence-Sullivan dg Lie algebra (\textit{cf.} \cites{LS14, BCM17}), although it does not seem to be clear for us, even in the case of double Poisson algebras. 
In addition, note that the Bernoulli numbers also appeared in the exponentiation construction in quasi-Poisson geometry via solutions of the classical dynamical Yang-Baxter equation (see \cite{AKSM}, Section 7). 
We hope that our results would yield a better understanding of double quasi-Poisson algebras.
 
\paragraph*{Acknowledgments.} The first author is supported by the Alexander von Humboldt Stiftung in the framework of an Alexander von Humboldt professorship endowed by the German Federal Ministry of Education and Research. 
The second author was supported by the GDRI ``Representation Theory'' 2016-2020 and the BIREP group, 
and is deeply thankful to Henning Krause and William Crawley-Boevey for their hospitality at the University of Bielefeld. 
We would like to thank Natalia Iyudu for pointing out reference \cite{IK19}. 
Finally, we would like to deeply thank the referees for the careful reading of the manuscript, as well as the comments and suggestions 
that improved the quality of this article. 

\section{Notations and conventions}
\label{section:preliminaries}

We will use the same notations and conventions as in \cite{FH19}, Section 2, but for the reader's convenience we recall the most important ones.  
In what follows, $\kk$ will denote a field of characteristic zero. 
We recall that, if $V = \oplus_{n \in \ZZ} V^{n}$ is a (cohomological) graded vector space, $V[m]$ is the graded 
vector space over $\kk$ whose $n$-th homogeneous component $V[m]^{n}$ is given by $V^{n+m}$, for all $n, m \in \ZZ$. 
It is called the \textbf{\textcolor{myblue}{shift}} of $V$. 
Given a nonzero element $v \in V^{n}$, we will denote $|v| = n$ the \textbf{\textcolor{myblue}{degree}} of $v$. 
If we refer to the degree of an element, we will be implicitly assuming that it is nonzero and homogeneous. 
We recall that a \textbf{\textcolor{myblue}{morphism}} $f : V \rightarrow W$ of graded vector spaces of degree $d \in \ZZ$ is a homogeneous linear map of degree $d$, \textit{i.e.} $f(V^{n}) \subseteq W^{n+d}$ for all $n \in \ZZ$. We will denote by $\mathcal{H}om(V,W)$ the graded vector space whose 
component of degree $d$ is formed by all morphisms from $V$ to $W$ of degree $d$. 
If $W = \kk$, we will denote $\mathcal{H}om(V,\kk)$ by $V^{\#}$. 

Given any $d \in \ZZ$, we will denote by $s_{V,d} : V \rightarrow V[d]$ the \textbf{\textcolor{myblue}{suspension morphism}}, whose underlying map is the identity of $V$, and $s_{V,1}$ will be denoted simply by $s_{V}$.
To simplify notation, we write $sv$ instead of $s_V(v)$ for a homogeneous $v\in V$.
All morphisms between vector spaces will be $\kk$-linear (satisfying further requirements if the spaces are further decorated). 
All unadorned tensor products $\otimes$ would be over $\kk$. 
We also remark that $\mathbb{N}$ will denote the set of positive integers, whereas $\mathbb{N}_{0}$ will be the set of 
nonnegative integers. 
Given two integers $a \leq b$, we denote by $\llbracket a , b \rrbracket$ the interval $\{ n \in \ZZ : a \leq n \leq b \}$.

Given $n \in \NN$, we will denote by $\mathbb{S}_{n}$ the group of permutations of $n$ elements $\{ 1, \dots, n\}$, 
and given any $\sigma \in \mathbb{S}_{n}$, $\operatorname{sgn}(\sigma) \in \{ \pm 1\}$ will denote its sign. 
Given two graded vector spaces $V$ and $W$, we denote by $\tau_{V,W} : V \otimes W \rightarrow W \otimes V$ 
the morphism of degree zero determined by $v \otimes w \mapsto (-1)^{|v| |w|} w \otimes v$, for all homogeneous elements $v \in V$ and $w \in W$. 
More generally, for any permutation $\sigma \in \mathbb{S}_{n}$, we define the homogeneous morphism 
$\tau_{V,n}(\sigma) : V^{\otimes n} \rightarrow V^{\otimes n}$ of degree zero sending $\bar{v} = v_{1} \otimes \dots \otimes v_{n}$ to 
\begin{equation}
\label{eq:permeps1}
      (-1)^{\epsilon(\sigma,\bar{v})} v_{\sigma^{-1}(1)} \otimes \dots \otimes v_{\sigma^{-1}(n)},     
\end{equation}   
where 
\begin{equation}
\label{eq:permeps2}
\epsilon(\sigma,\bar{v}) = \underset{\text{\begin{tiny} $\begin{matrix}i<j,\\ \sigma^{-1}(i) > \sigma^{-1}(j)\end{matrix}$ \end{tiny}}}{\sum} |v_{\sigma^{-1}(i)}| |v_{\sigma^{-1}(j)}|.
\end{equation}          
To simplify notation, we will usually write $\sigma$ instead of $\tau_{V,n}(\sigma)$. 

For later use, we recall that, given homogeneous elements $v_{1}, \dots, v_{n} \in V$ of a graded vector space $V$, as well as 
$f_{1}, \dots, f_{n} \in V^{\#}$ homogeneous, then 
\begin{equation}
\label{eq:unipermvecfun}
   (f_{1} \otimes \dots \otimes f_{n})\big(\sigma (v_{1} \otimes \dots\otimes v_{n})\big) = \big(\sigma^{-1}(f_{1} \otimes \dots \otimes f_{n})\big)(v_{1} \otimes\dots\otimes v_{n}).
\end{equation}

For later use and to avoid ambiguities, we recall that, given any $x \in \RR$ and $d \in \ZZ$, the corresponding \textbf{\textcolor{myblue}{binomial coefficient}} is
\begin{equation}
\label{eq:bin}
C(x,d) = \begin{pmatrix} x \\ d \end{pmatrix} = \frac{\prod\limits_{i=0}^{d-1} (x - i) }{d!} \text{ if $d \geq 0$, and } C(x,d)=\begin{pmatrix} x \\ d \end{pmatrix} = 0 \text{ if $d < 0$. }
\end{equation}
As usual, if $d = 0$, the corresponding binomial coefficient is by definition $1$, since the product in the numerator is $1$. 
Recall the following direct identities
\begin{equation}
\label{eq:binid}
d \begin{pmatrix} x \\ d \end{pmatrix} = (x - d +1) \begin{pmatrix} x \\ d-1 \end{pmatrix} \text{ and } (x' - d') \begin{pmatrix} x' \\ d' \end{pmatrix} = x' \begin{pmatrix} x' - 1  \\ d' \end{pmatrix},  
\end{equation}
for all $x, x' \in \RR$ and all $d, d' \in \ZZ$. 
We also have the well-known \textbf{\textcolor{myblue}{Pascal identity}} 
\begin{equation}
\label{eq:binpas}
\begin{pmatrix} x \\ d \end{pmatrix} = \begin{pmatrix} x - 1 \\ d - 1 \end{pmatrix} + \begin{pmatrix} x -1 \\ d \end{pmatrix}, 
\end{equation}
for all $x \in \RR$ and $d \in \ZZ$. 
Furthermore, a telescopic argument on \eqref{eq:binpas} gives 
\begin{equation}
\label{eq:binpasgen}
\sum_{d=e'}^{e} (-1)^{d} \begin{pmatrix} x \\ d \end{pmatrix} = (-1)^{e'} \begin{pmatrix} x - 1 \\ e'-1 \end{pmatrix} + (-1)^{e} \begin{pmatrix} x -1 \\ e \end{pmatrix}, 
\end{equation}
for all $x \in \RR$ and all $e', e \in \ZZ$ such that $e' \leq e$. 
In this article, we will only consider the case where the numerators of binomial coefficients are integers. 
In that case, it is easy to see that the left member of \eqref{eq:binpasgen} coincides with the same sum where the upper limit is $\operatorname{min}(x,e)$ and the lower limit is $\operatorname{max}(0,e')$, provided $x \geq 0$. 
Moreover, we also recall the easy identity 
\begin{equation}
\label{eq:binids}
\begin{pmatrix} x \\ d \end{pmatrix} = \begin{pmatrix} x \\ x - d \end{pmatrix},  
\end{equation}
for all $x, d \in \ZZ$ with $x \geq 0$. 

\section{Double quasi-Poisson algebras}
\label{section:quasi-Poisson}
\subsection{Quasi-Poisson manifolds}

For the reader's convenience, in this subsection we will review the notion of a quasi-Poisson manifold, following \cite{AKSM}. 
The version for affine schemes is analogous. 
Let $G$ be a compact Lie group with Lie algebra $\mathfrak{g}$. 
We assume that $\mathfrak{g}$ is endowed with a $G$-invariant positive definite symmetric bilinear form $\langle \hskip 0.6mm , \rangle$, which we shall use to identify $\mathfrak{g}^*$ with $\mathfrak{g}$. 
For any $G$-manifold $M$ and any $\xi\in\mathfrak{g}$, the generating vector field of the induced infinitesimal action is defined by 
\[     \xi_M(m) =\frac{d}{d t}\Big(\exp(-t\xi).m\Big)\Big|_{t=0},     \] 
for all $m \in M$. 
It is known that the Lie algebra homomorphism $\mathfrak{g} \rightarrow \Gamma(M,TM)$ given by $\xi \mapsto \xi_M$, extends to an equivariant map, $\Lambda^{\bullet}\mathfrak{g} \rightarrow \Gamma(M,\Lambda^{\bullet}TM)$, preserving the (wedge) product and the Schouten-Nijenhuis bracket. 
In fact, for any $\alpha \in \Lambda^{\bullet}\mathfrak{g}$, 
we denote by $\alpha_M$ the corresponding multivector field $\alpha_M \in \Gamma(M,\Lambda^{\bullet}TM)$.

If $\{e_i\}_{i \in I}$ is an orthonormal basis of $\mathfrak{g}$ with respect to $\langle \hskip 0.6mm , \rangle$, the \textbf{\textcolor{myblue}{Cartan $3$-tensor}} $\phi\in\Lambda^3 \mathfrak{g}$ is given by
\[
 \phi= \frac{1}{12} \sum_{i,j,k \in I} \big\langle e_i, [e_j,e_k] \big\rangle \; e_i\wedge e_j\wedge e_k.
\]
Note that for any $G$-manifold $M$ the Cartan $3$-tensor $\phi$ corresponding to an invariant inner product on $\mathfrak{g}$ gives rise to an invariant trivector field $\phi_M$ on $M$.

Following \cite{AKSM}, a \textbf{\textcolor{myblue}{quasi-Poisson manifold}} is a $G$-manifold $M$ equipped with a $G$-equivariant skew-symmetric bilinear map $\{\hskip 0.6mm , \} : C^{\infty}(M) \times C^{\infty}(M) \rightarrow C^{\infty}(M)$ 
satisfying the Leibniz identity and such that 
\[
 \{\{f_1,f_2\},f_3\}+ \{\{f_2,f_3\},f_1\}+ \{\{f_3,f_1\},f_2\}=\phi_M(df_1,df_2,df_3),
\]
for all $f_{1}, f_{2}, f_{3} \in C^{\infty}(M)$. 
Equivalently, the Jacobi identity holds up to the invariant trivector field $\phi_M$, coming from the Cartan $3$-tensor. 
More succinctly, a quasi-Poisson manifold is a $G$-manifold $M$ equipped with a $G$-invariant bivector field 
$P \in \Gamma(M,\Lambda^2 TM)$ satisfying that $[P,P]=\phi_M$, where $[\hskip 0.6mm,]$ denotes the Schouten-Nijenhuis bracket.

\subsection{Multibrackets and double Poisson algebras}

\subsubsection{Double derivations}

One of the main features in noncommutative algebraic geometry, as developed in \cite{CBEG}, \cite{VdB} and \cite{q-VdB}, is that the role of vector fields is played by double derivations. 

Let $A$ be a nonunitary associative $\kk$-algebra. 
The $A$-bimodule of \textbf{\textcolor{myblue}{double derivations}} is given by $\D A = \Der(A,(A\otimes A)_{\out})$, 
where $(A\otimes A)_{\out}$ denotes the \textbf{\textcolor{myblue}{outer}} bimodule structure on $A\otimes A$ (\textit{i.e.} $a_1(a' \otimes a'')a_2 = a_1 a'\otimes a'' a_2$, for all $a_1,a_2,a',a'' \in A$). 
The surviving inner bimodule structure on $A\otimes A$ makes $\D A$ into an $A$-bimodule by means of 
$(a\Theta b)(c)=\Theta^{\prime}(c)b\otimes a\Theta^{\prime\prime}(c)$, for all $a,b,c\in A$, $\Theta\in\D A$, where we are using the usual Sweedler notation $\Theta(c) = \Theta'(c) \otimes \Theta''(c)$. 
If $A$ has a unit $1_{A}$, the bimodule of double derivations is endowed with a distinguished element $\E \in\D A$ (see \cite{VdB}, \S 3.3, and \cite{q-VdB}, \S 2.3), defined as
\begin{equation}
 \label{double-derivation-E}
\begin{split}
 \E : A &\longrightarrow A\otimes A,
 \\
 a &\longmapsto a\otimes 1_{A}-1_{A}\otimes a.
 \end{split}
\end{equation}
Note that this double derivation also appeared in \cite{CBEG}, \S 3.1, where it was denoted by $\Delta$.
Finally, we define the \textbf{\textcolor{myblue}{algebra of polyvector fields}} $DA = \oplus_{n\in \NN_{0}} D_{n}A$ on $A$ to be the tensor algebra $T_A(\D A)$, which is graded by the tensor power. 

\subsubsection{\texorpdfstring{Multibrackets}{Multibrackets}}

\begin{definition}[\cite{VdB}, Def. 2.2.1]
\label{def:n-bracket}
Given $n \in \NN$, an \textbf{\textcolor{myblue}{$n$-bracket}} on a nonunitary associative $\kk$-algebra $A$ is a linear map $\lr{ \dots } : A^{\otimes n}\to A^{\otimes n}$ which is a derivation from $A$ to $A^{\otimes n}$ in its last argument for the outer bimodule structure on $A^{\otimes n}$, \textit{i.e.}
\[
   \lr{a_1,a_2,\dots,a_{n-1},a_n a'_n}=a_n\lr{a_1,a_2,\dots,a_{n-1}, a'_n}+\lr{a_1,a_2,\dots,a_{n-1},a_n} a'_n,    
   \]
and is cyclically skew-symmetric, \textit{i.e.}
\[
\sigma \circ\lr{ \dots }\circ  \sigma^{-1}=(-1)^{n+1}\lr{ \dots },
\]
for $\sigma \in \mathbb{S}_{n}$ the unique cyclic permutation sending $1$ to $2$. 
It is easy to see that a $1$-bracket is simply a derivation on $A$. 
\end{definition}

Since we will mainly deal with $n$-brackets for $n=2$, which are called \textbf{\textcolor{myblue}{double brackets}}, let us present their definition more explicitly. 
A double bracket on an associative $\kk$-algebra $A$ is a $\kk$-linear map $\lr{\hskip 0.6mm, } : A^{\otimes 2} \rightarrow A^{\otimes 2}$  satisfying
 \begin{enumerate}[label={(DB.\arabic*)}]
\setcounter{enumi}{0} 
\item\label{item:double1} 
$\lr{a,b}=-\tau_{A,A}\lr{b,a}$,
\item\label{item:double2} 
$\lr{a,bc}=b\lr{a,c}+\lr{a,b}c$,
\end{enumerate}
for all $a,b,c\in A$.
We recall that $A$ is concentrated in degree zero, so $\tau_{A,A}$ is the usual flip. 
The identity \ref{item:double2} is called the \textbf{\textcolor{myblue}{Leibniz identity}} and it can be reformulated as saying that a double bracket is a double derivation in its second argument. 
A $3$-bracket will be usually called a \textbf{\textcolor{myblue}{triple}} bracket.  

Let $A$ be a nonunitary associative $\kk$-algebra. 
In \cite{VdB}, Prop. 4.1.1, it is showed that there exists a well-defined linear map
\begin{equation}                                                                                                                                                                                                                    
\label{map-mu}
   \mu : D_{n}A \longrightarrow \big\{\kk\text{-linear }n\text{-brackets on }A \big\},   
\end{equation}
for all $n \in \NN$, sending $Q=\delta_1 \otimes_{A} \dots \otimes_{A} \delta_n \in D_n A$, where $\delta_1, \dots, \delta_n \in \D A$, to 
\begin{equation}
\label{formula-mu}
 \lr{\dots }_Q=\sum^{n-1}_{i=0}(-1)^{(n-1)i}\,\sigma^{i} \circ\lr{ \dots }^{\sim}_Q\circ \sigma^{-i},
\end{equation}
where $\lr{a_1,\dots,a_n}^{\sim}_{Q}=\delta^{\prime}_n(a_n)\delta^{\prime\prime}_1(a_1)\otimes\delta^{\prime}_1(a_1)\delta^{\prime\prime}_2(a_2)\otimes \dots \otimes\delta^{\prime}_{n-1}(a_{n-1})\delta^{\prime\prime}_n(a_n)$, and $\sigma \in \mathbb{S}_{n}$ is the unique cyclic permutation sending $1$ to $2$.

\subsubsection{Double Poisson algebras}

Following \cite{VdB}, \S 2.3, if $\lr{ \hskip 0.6mm, } : A^{\otimes 2} \rightarrow A^{\otimes 2}$ is a double bracket on a nonunitary associative algebra $A$ and given elements $a, b_1, \dots, b_n \in A$, we define 
\begin{equation}
\label{eq:extension-left-double-bracket}
\lr{a,b_1\otimes b_2\otimes \dots\otimes b_n}_L=\lr{a,b_1}\otimes b_2\otimes \cdots \otimes b_n\in A^{\otimes( n+1)}.
\end{equation}
From a double bracket $\lr{ \hskip 0.6mm, }  : A^{\otimes 2} \rightarrow A^{\otimes 2}$ one defines the \textbf{\textcolor{myblue}{associated  triple bracket}} $\lr{ \hskip 0.6mm, \hskip 0.6mm, } : A^{\otimes 3} \rightarrow A^{\otimes 3}$ given by 
\begin{equation}
 \lr{c,b,a}=\lr{c,\lr{b,a}}_L + \sigma \lr{b,\lr{a,c}}_L + \sigma^{2} \lr{a,\lr{c,b}}_L,
\label{triple-bracket}
\end{equation}
for all $a, b, c \in A$, where $\sigma \in \mathbb{S}_{3}$ is the unique cyclic permutation sending $1$ to $2$. 
It is not hard to prove that it is indeed a triple bracket (see \cite{VdB}, Prop. 2.3.1). 

\begin{definition}[\cite{VdB}, Def. 2.3.2]
\label{definition:dP}
A double bracket $\lr{ \hskip 0.6mm, }$ on a nonunitary associative $\kk$-algebra $A$ is a \textbf{\textcolor{myblue}{double Poisson bracket}} if the associated triple bracket $\lr{ \hskip 0.6mm, \hskip 0.6mm,}$ vanishes. 
In this case, $(A, \lr{ \hskip 0.6mm, })$ is called a \textbf{\textcolor{myblue}{double Poisson algebra}}.
\label{def-double-Poisson}
\end{definition}

The identity $\lr{ \hskip 0.6mm, \hskip 0.6mm,}=0$ is called the \textbf{\textcolor{myblue}{double Jacobi identity}}. 
So, the triple bracket defined in \eqref{triple-bracket} can be regarded as a ``noncommutative Jacobiator'' that measures the failure to satisfy the double Jacobi identity.

\subsubsection{Double quasi-{P}oisson algebras}

In the definition of quasi-Poisson manifolds, the basic idea was that the failure of the Jacobi identity is controlled by the canonical trivector $\phi_M$ coming from the Cartan $3$-tensor. 
To adapt this idea to our noncommutative setting based on double derivations, Van den Bergh proposed to require that the ``noncommutative Jacobiator'' $\lr{ \hskip 0.6mm, \hskip 0.6mm,}$ defined in \eqref{triple-bracket} should be equal, up to constant, to the triple bracket $\lr{ \hskip 0.6mm, \hskip 0.6mm,}_{\E^3} = \mu(\E^3)$, where $\E$ is the distinguished double derivation defined in \eqref{double-derivation-E}, $\E^3 = E \otimes_{A} E \otimes_{A} E \in D_3 A$, and $\mu$ is the map introduced in \eqref{map-mu}. 

\begin{definition}[\textit{cf.} \cite{VdB}, Def. 5.1.1]
\label{double-quasi-poisson-def}
A \textbf{\textcolor{myblue}{double quasi-Poisson bracket}} with parameter $\para \in \kk$ 
on an associative $\kk$-algebra $A$ with unit $1_{A}$ is a $\kk$-linear double bracket $\lr{ \hskip 0.6mm, }$ such that its associated triple bracket satisfies 
\begin{equation}
 \label{eq:double-quasi-poisson-def}
  \lr{ \hskip 0.6mm, \hskip 0.6mm,}=\frac{\para}{12} \mu(\E^{3}) = \frac{\para}{12} \lr{ \hskip 0.6mm, \hskip 0.6mm,}_{\E^3}.
\end{equation}
\end{definition}

The previous definition is not exactly the one appearing in \cite{VdB}, Def. 5.1.1. 
However, we find more convenient to use this slightly more general notion, since it allows us to work simultaneously with double Poisson and double quasi-Poisson algebras. 
Indeed, note that Definition \ref{double-quasi-poisson-def} for $\para = 0$ reduces to Definition \ref{definition:dP}, whereas the case $\para = 1$ gives the usual definition of double quasi-Poisson algebra. 
More generally, if $\para \neq 0$ and $\kk$ is quadratically closed (\textit{i.e.} all square roots of elements of $\Bbbk$ are in $\Bbbk$), pick $\hat{\para} \in \kk$ such that $\hat{\para}^{2} = \para$. 
Then, \eqref{eq:double-quasi-poisson-def} is tantamount to $(A,\hat{\para}^{-1} \lr{ \hskip 0.6mm , })$ 
being a double quasi-Poisson algebra in the sense of Van den Bergh.

Using \eqref{formula-mu} we deduce that  
\begin{equation}
\label{RHS-quasi-Poisson-double}
 \begin{split}
   \frac{1}{12}\lr{c,b,a}_{\E^3} &=\frac{1}{4}\Big(ac\otimes b\otimes 1_{A}-ac\otimes 1_{A}\otimes b-a\otimes cb\otimes 1_{A}+a\otimes c\otimes b
 \\
 &\phantom{= \frac{1}{4}}+c\otimes 1_{A}\otimes ba-c\otimes b\otimes a+1_{A}\otimes cb\otimes a-1_{A}\otimes c\otimes ba\Big),
 \end{split}
\end{equation}
for all $a,b,c\in A$. 

In \cite{VdB}, Thm. 7.12.2, the author proves that a double quasi-Poisson algebra $(A,\lr{\hskip 0.6mm,})$ induces a quasi-Poisson bracket on the representation scheme $\Rep(A,V)$ (\textit{i.e.} the affine scheme parametrizing the $\kk$-linear representations of $A$ on a finite-dimensional vector space $V$), regarded as a $\GL(V)$-scheme. 
In other words, double quasi-Poisson algebras satisfy the Kontsevich-Rosenberg principle \cite{KoRo}. 

\section{\texorpdfstring{Quasi-cyclic $A_{\infty}$-algebras and pre-Calabi-Yau structures}{Quasi-cyclic A-infinity-algebras and pre-Calabi-Yau structures}}
\label{section:cyclic-pre-CY}

In this section we provide the basic definitions of (quasi-)cyclic $A_{\infty}$-algebras and pre-Calabi-Yau structures. 
Most of this material can be found in \cites{KoSo09, IK17, IKV19} (see also \cite{FH19}). 
We also recall and extend some technical terminology on cyclic $A_{\infty}$-algebras and pre-Calabi-Yau structures from \cite{FH19}, Section 4.   

\subsection{\texorpdfstring{$A_{\infty}$-algebras}{A-infinity-algebras}}
\label{subsection:Ainfty}

The following notion was introduced by J. Stasheff in \cite{Sta63}. 
We recall that a \textbf{\textcolor{myblue}{nonunitary $A_{\infty}$-algebra}} is a (cohomologically) graded vector space 
$A=\oplus_{n \in \ZZ} A^{n}$ together with a collection of maps $\{ m_{n} \}_{n \in \NN}$, 
where $m_{n} : A^{\otimes n} \rightarrow A$ is a homogeneous morphism of degree $2-n$, satisfying the \hypertarget{eq:ainftyalgebralink}{equation}
\begin{equation}
\tag{$\operatorname{SI}(N)$}
\label{eq:ainftyalgebra}
   \sum_{(r,s,t) \in \mathcal{I}_{N}} (-1)^{r + s t} m_{r + 1 + t} \circ (\mathrm{id}_{A}^{\otimes r} \otimes m_{s} \otimes \mathrm{id}_{A}^{\otimes t}) = 0, 
\end{equation} 
for all $N \in \NN$, where $\mathcal{I}_{N} = \{ (r,s,t) \in \NN_{0} \times \NN \times \NN_{0} : r + s + t = N \}$. 
We will denote by $\operatorname{SI}(N)$ the homogeneous morphism of degree $3-N$ from $A^{\otimes N}$ to $A$ given by the left hand side of \eqref{eq:ainftyalgebra}. 
Note that a nonunitary graded associative algebra is the same as a nonunitary $A_{\infty}$-algebra $(A,m_{\bullet})$ 
satisfying that $m_{n}$ vanishes for all $n \neq 2$. 

\begin{remark}
\label{remark:restmult}
Note that, if $A = A^{0} \oplus A^{1}$, by degree reasons, any collection $\{ m_{n} \}_{n \in \NN}$ of maps such that $m_{n} : A^{\otimes n} \rightarrow A$ has degree $2-n$ for all $n \in \NN$ satisfies that 
\begin{equation}
\label{eq:restmult}
m_{n}(A^{i_{1}} \otimes \dots \otimes A^{i_{n}}) \subseteq A^{2-n+ |\bar{i}|},
\end{equation}
for all $\bar{i} = (i_{1}, \dots, i_{n}) \in \{0,1\}^{n}$ and $|\bar{i}| = \sum_{j=1}^{n} i_{j} $.
Then, $m_{n}|_{A^{i_{1}} \otimes \dots \otimes A^{i_{n}}}$ vanishes unless that $|\bar{i}| \in \{ n-2, n-1 \}$.
\end{remark}

Let $A$ be a graded vector space with a distinguished element $1_{A} \in A^{0}$. 
A mapping $m_{n} : A^{\otimes n} \rightarrow A$, for $n \in \NN$, is called \textbf{\textcolor{myblue}{normalized}} (with respect to $1_{A}$), or \textbf{\textcolor{myblue}{$1_{A}$-normalized}}, if $m_{n}(a_{1},\dots,a_{n}) = 0$ whenever there is $i \in \llbracket 1 , n \rrbracket$ such that $a_{i} = 1_{A}$. 
We also recall that an $A_{\infty}$-algebra is \textbf{\textcolor{myblue}{strictly unitary}} if there exists an element $1_{A} \in A^{0}$ such that 
$1_{A}$ is a unit for the product $m_{2}$, and $m_{n}$ is $1_{A}$-normalized for all $n \in \NN\setminus \{2\}$. 

\subsection{\texorpdfstring{Quasi-cyclic structures on $A_{\infty}$-algebras}{Quasi-cyclic structures on A-infinity-algebras}}
\label{subsection:ultra-cyclic}

Given $d \in \ZZ$, a \textbf{\textcolor{myblue}{$d$-quasi-cyclic (nonunitary) $A_{\infty}$-algebra}} is a nonunitary $A_{\infty}$-algebra $(A,m_{\bullet})$ provided with a bilinear form $\gamma : A \otimes A \rightarrow \kk$ of degree $d$ satisfying that $\gamma \circ \tau_{A,A} = \gamma$, $\gamma$ has trivial kernel, \textit{i.e.} if $\gamma(a,b) = 0$ for some $a \in A$ and all $b \in A$ then $a = 0$, and 
\begin{equation}
\label{eq:ip1}
     \gamma\big(m_{n}(a_{1},\dots,a_{n-1}, a_{n}),a_{0}\big) = (-1)^{n + |a_{0}|(\sum_{i=1}^{n} |a_{i}|) } \gamma\big(m_{n}(a_{0},a_1,\dots,a_{n-1}),a_{n}\big),  
\end{equation}
for all homogeneous $a_{0}, \dots, a_{n} \in A$. 
If we further impose that $\gamma$ be nondegenerate, then $(A,m_{\bullet})$ is called a \textbf{\textcolor{myblue}{$d$-cyclic $A_{\infty}$-algebra}}.

Let $A$ be a graded vector space provided with a bilinear form $\gamma : A \otimes A \rightarrow \kk$ of degree $d$ satisfying that $\gamma \circ \tau_{A,A} = \gamma$ and having trivial kernel.
Let $\{ m_{n} : n \in \NN \}$ be a family of maps of the form $m_{n} : A^{\otimes n} \rightarrow A$ of degree $2-n$ satisfying the identity \eqref{eq:ip1}. 
Given $N \in \NN$, we will define the homogeneous linear map $\operatorname{SI}(N)_{\gamma} : A^{\otimes (N+1)} \rightarrow \kk$ 
of degree $3-N+d$ \hypertarget{eq:ainftyalgebragammalinkop}{by}
\begin{equation*}
\label{eq:ainftyalgebragammaop}
   \operatorname{SI}(N)_{\gamma} = \sum_{(r,s,t) \in \mathcal{I}_{N}} (-1)^{r + s t} \gamma \circ \big(m_{r + 1 + t} \circ (\mathrm{id}_{A}^{\otimes r} \otimes m_{s} \otimes \mathrm{id}_{A}^{\otimes t}) \otimes \mathrm{id}_{A}\big).
\end{equation*}
Then, the Stasheff identities \eqref{eq:ainftyalgebra} become equivalent to the vanishing of $\operatorname{SI}(N)_{\gamma}$, for all $N \in \NN$, \hypertarget{eq:ainftyalgebragammalink}{\textit{i.e.}}
\begin{equation}
\tag{$\operatorname{SI}(N)_{\gamma}$}
\label{eq:ainftyalgebragamma}
   \sum_{(r,s,t) \in \mathcal{I}_{N}} (-1)^{r + s t} \gamma \circ \big(m_{r + 1 + t} \circ (\mathrm{id}_{A}^{\otimes r} \otimes m_{s} \otimes \mathrm{id}_{A}^{\otimes t}) \otimes \mathrm{id}_{A}\big) = 0,
\end{equation}
for all $N \in \NN$.

\begin{lemma}
\label{lemma:cyc}
Let $A$ be a graded vector space provided with a bilinear form $\gamma : A \otimes A \rightarrow \kk$ of degree $d$ satisfying that $\gamma \circ \tau_{A,A} = \gamma$ and having trivial kernel. 
Let $\{ m_{n} : n \in \NN \}$ be a family of maps of the form $m_{n} : A^{\otimes n} \rightarrow A$ of degree $2-n$ satisfying identity \eqref{eq:ip1}. 
Then, 
\[     \operatorname{SI}(N)_{\gamma}(a_{1},\dots,a_{N},a_{0}) = (-1)^{N + |a_{0}| (|a_{1}|+\dots+|a_{N}|)} \operatorname{SI}(N)_{\gamma}(a_{0},a_{1},\dots,a_{N}),     \]
for all $N \in \NN$ and all homogeneous $a_{0}, \dots, a_{N} \in A$. 
\end{lemma}
\begin{proof}
Define the linear maps $\operatorname{SI}(N)^{0}_{\gamma}, {}^{0}\operatorname{SI}(N)_{\gamma}, \operatorname{SI}(N)^{+}_{\gamma}, {}^{+}\operatorname{SI}(N)_{\gamma}  : A^{\otimes (N+1)} \rightarrow \kk$ via
\begin{equation}
\label{eq:ainftyalgebragammapri}
   \operatorname{SI}(N)^{0}_{\gamma} = \sum_{(r,s,0) \in \mathcal{I}_{N}} (-1)^{r} \gamma \circ \big(m_{r + 1} \circ (\mathrm{id}_{A}^{\otimes r} \otimes m_{s}) \otimes \mathrm{id}_{A}\big)
\end{equation}
and 
\begin{equation}
\label{eq:ainftyalgebragammapripri}
   {}^{0}\operatorname{SI}(N)_{\gamma} = \sum_{(0,s,t) \in \mathcal{I}_{N}} (-1)^{s t} \gamma \circ \big(m_{1 + t} \circ (m_{s} \otimes \mathrm{id}_{A}^{\otimes t}) \otimes \mathrm{id}_{A}\big),
\end{equation}
together with $\operatorname{SI}(N)^{+}_{\gamma} = \operatorname{SI}(N)_{\gamma} - \operatorname{SI}(N)^{0}_{\gamma}$
and ${}^{+}\operatorname{SI}(N)_{\gamma} = \operatorname{SI}(N)_{\gamma} - {}^{0}\operatorname{SI}(N)_{\gamma}$. 
Using \eqref{eq:ip1}, we get 
\[     \operatorname{SI}(N)^{+}_{\gamma}(a_{1},\dots,a_{N},a_{0}) = (-1)^{N + |a_{0}|( |a_{1}|+\dots+|a_{N}|)} \hskip 2mm {}^{+}\operatorname{SI}(N)_{\gamma}(a_{0},\dots,a_{N}).     \]
Indeed, 
\begin{small}
	\begin{align*}     
	&\operatorname{SI}(N)^{+}_{\gamma}(a_{1},\dots,a_{N},a_{0}) 
	\\
	&= \hskip -4mm \sum_{\text{\begin{tiny}$\begin{matrix}(r,s,t) \in \mathcal{I}_{N} \\ t > 0 \end{matrix}$\end{tiny}}} \hskip -4mm (-1)^{r+st+s (\sum\limits_{j=1}^{r} |a_{j}|)} \gamma \Big(m_{r + 1+t}\big(a_{1}, \dots, a_{r}, m_{s}(a_{r+1}, \dots, a_{r+s}),a_{r+s+1},\dots,a_{N}\big),a_{0}\Big) 
	\\
	&= \hskip -4mm \sum_{\text{\begin{tiny}$\begin{matrix}(r,s,t) \in \mathcal{I}_{N} \\ t > 0 \end{matrix}$\end{tiny}}} \hskip -4mm (-1)^{r+s t+s (\sum\limits_{j=1}^{r} |a_{j}|) + r+1+t + |a_{0}| (s+\sum\limits_{j=1}^{N} |a_{j}|)} 
	\\
	&\phantom{\hskip 1.4cm \sum_{\text{\begin{tiny}$\begin{matrix}(r,s,t) \in \mathcal{I}_{N} \\ t > 0 \end{matrix}$\end{tiny}}}} 
	\gamma \Big(m_{r + 1+t}\big(a_{0}, \dots, a_{r}, m_{s}(a_{r+1}, \dots, a_{r+s}),a_{r+s+1},\dots,a_{N-1}\big),a_{N}\Big)
    \\
	&= \hskip -4mm \sum_{\text{\begin{tiny}$\begin{matrix}(r',s,t') \in \mathcal{I}_{N} \\ r' > 0 \end{matrix}$\end{tiny}}} \hskip -4mm (-1)^{r'+st'+s (\sum\limits_{j=0}^{r'-1} |a_{j}|) + N + |a_{0}| (\sum\limits_{j=1}^{N} |a_{j}|)} 
	\\
	&\phantom{\hskip 1.4cm \sum_{\text{\begin{tiny}$\begin{matrix}(r,s,t) \in \mathcal{I}_{N} \\ t > 0 \end{matrix}$\end{tiny}}}}
	\gamma \big(m_{r' + 1+t'}\Big(a_{0}, \dots, a_{r'-1}, m_{s}(a_{r'}, \dots, a_{r'+s-1}),a_{r'+s},\dots,a_{N-1}\big),a_{N}\Big)
	\\
	&= (-1)^{N + |a_{0}|( |a_{1}|+\dots+|a_{N}|)} \hskip 2mm {}^{+}\operatorname{SI}(N)_{\gamma}(a_{0},\dots,a_{N}),
	\end{align*}
\end{small}
\hskip -0.8mm where $r' = r+1$ and $t'= t-1$. 
On the other hand, 
\begin{small}
\begin{align*}     
     &\operatorname{SI}(N)^{0}_{\gamma}(a_{1},\dots,a_{N},a_{0}) 
     = \hskip -4mm \sum_{(r,s,0) \in \mathcal{I}_{N}} \hskip -4mm (-1)^{r+s (\sum\limits_{j=1}^{r} |a_{j}|)} \gamma \Big(m_{r + 1}\big(a_{1}, \dots, a_{r}, m_{s}(a_{r+1}, \dots, a_{N})\big),a_{0}\Big) 
     \\
     &= \hskip -4mm \sum_{(r,s,0) \in \mathcal{I}_{N}} \hskip -4mm (-1)^{r+s (\sum\limits_{j=1}^{r} |a_{j}|) + r+1+ |a_{0}| (s+\sum\limits_{j=1}^{N} |a_{j}|)} 
     \gamma \big(m_{r + 1}(a_{0}, \dots, a_{r}), m_{s}(a_{r+1}, \dots,a_{N})\big)
     \\
     &= - \hskip -4mm \sum_{(r,s,0) \in \mathcal{I}_{N}} \hskip -4mm (-1)^{|a_{0}| (\sum\limits_{j=1}^{N} |a_{j}|) + s (\sum\limits_{j=0}^{r} |a_{j}|)+
     	(s +\sum\limits_{j=r+1}^{N} |a_{j}|)(r+1+\sum\limits_{j=0}^{r} |a_{j}|)}
     \\
     &\phantom{= - \hskip -4mm \sum_{(r,s,0) \in \mathcal{I}_{N}} \hskip -4mm (-1)^{|a_{0}| (\sum\limits_{j=1}^{N} |a_{j}|) + s (\sum\limits_{j=0}^{r} |a_{j}|)}}
     \gamma \big(m_{s}(a_{r+1}, \dots,a_{N}), m_{r+1}(a_{0}, \dots, a_{r})\big)
     \\
     &= (-1)^{N +|a_{0}| (\sum\limits_{j=1}^{N} |a_{j}|)} \hskip -4mm \sum_{(r,s,0) \in \mathcal{I}_{N}} \hskip -4mm (-1)^{r s + N-1}  \gamma \Big(m_{s}\big(
     m_{r + 1}(a_{0}, \dots, a_{r}), a_{r+1}, \dots, a_{N-1}\big), a_{N} \Big)
     \\
     &= (-1)^{N +|a_{0}| (\sum\limits_{j=1}^{N} |a_{j}|)} \hskip -4mm \sum_{(0,s',t') \in \mathcal{I}_{N}} \hskip -4mm (-1)^{s' t'}  \gamma \Big(m_{t'+1}\big(
     m_{s'}(a_{0}, \dots, a_{s'-1}), a_{s'}, \dots, a_{N-1}\big), a_{N} \Big)
     \\
     &= (-1)^{N +|a_{0}| (\sum\limits_{j=1}^{N} |a_{j}|)} \hskip 2mm {}^{0}\operatorname{SI}(N)_{\gamma}(a_{0},\dots,a_{N}),
\end{align*}
\end{small}
\hskip -0.8mm where $s' = r+1$ and $t'= s-1$. 
Combining the previous equations we obtain the desired result. 
\end{proof}

The following result is an immediate consequence of property \eqref{eq:ip1}.
\begin{fact}
\label{fact:cyc2}
Let $A$ be a graded vector space with a distinguished element $1_{A}$, provided with a bilinear form $\gamma : A \otimes A \rightarrow \kk$ of degree $d$ satisfying that $\gamma \circ \tau_{A,A} = \gamma$ and having trivial kernel. 
Let $n \in \NN$ and $m_{n} : A^{\otimes n} \rightarrow A$ be a $1_{A}$-normalized map of degree $2-n$ satisfying identity \eqref{eq:ip1}. 
Then, $\gamma(m_{n}(a_{1},\dots,a_{n}),1_{A}) = 0$, for all $a_{1}, \dots, a_{n} \in A$.
\end{fact}

Note that $(A,m_{\bullet})$ is an $A_{\infty}$-algebra if and only if $\operatorname{SI}(N)_{\gamma}$ vanishes for all $N \in \NN$, 
and, if $(A,m_{\bullet})$ is also provided with a bilinear form $\gamma : A \otimes A \rightarrow \kk$ of degree $d$ satisfying that $\gamma \circ \tau_{A,A} = \gamma$ and having trivial kernel, it is $d$-quasi-cyclic if and only if 
\eqref{eq:ip1} holds for all $n \in \NN$. 

\subsection{Natural bilinear forms and pre-Calabi-Yau structures}
\label{subsection:nat-bilinear}

We recall the definition of the \textbf{\textcolor{myblue}{natural bilinear form}} (of degree $-1$) associated with any vector space $A$, which is considered as a (cohomologically) graded vector space concentrated in degree zero. 
First, we set $\partial A = A \oplus A^{\#}[-1]$. 
For clarity, we will denote the suspension map $s_{A^{\#},-1} : A^{\#} \rightarrow A^{\#}[-1]$ simply by $t$, 
and any element of $A^{\#}[-1]$ will be thus denoted by $tf$, for $f \in A^{\#}$.
Define now the bilinear form 
\[
 \gan_{A} : \partial A \otimes \partial A \longrightarrow \kk
\]
 by 
\begin{equation}
\label{eq:natform} 
     \gan_{A}(tf,a) = \gan_{A}(a,tf) = f(a)
     \text{ $\phantom{x}$ and $\phantom{x}$ }
     \gan_{A}(a,b)  = \gan_{A}(tf,tg) = 0,
\end{equation}
for all $a,b \in A$ and $f, g \in A^{\#}$. 
Note that $\gan_{A}$ has degree $-1$, $\gan_{A} \circ \tau_{\partial A, \partial A} = \gan_{A}$ and its kernel is trivial.
If there is no risk of confusion, we shall denote $\gan_{A}$ simply by $\gan$. 

We now present the crucial notion of pre-Calabi-Yau algebra, introduced in \cite{KV18} (see also \cites{IK17, IKV19}, Def. 2.5). 
The definition below is an equivalent version for the case of an algebra $A$ concentrated in degree zero and it follows from unraveling 
the generalized necklace bracket.
\begin{definition}
\label{definition-pre-d-CY}
A \textbf{\textcolor{myblue}{pre-Calabi-Yau (algebra) structure}} on a vector space $A$, considered as a graded vector space concentrated in degree zero, is the datum of a $(-1)$-quasi-cyclic $A_{\infty}$-algebra structure on the graded vector space 
$\partial A = A \oplus A^{\#}[-1]$ for the natural bilinear form $\gan_{A} : \partial A \otimes \partial A \rightarrow \kk$ of degree $-1$ defined in \eqref{eq:natform} such that the corresponding multiplications $\{ m_{n} \}_{n \in \NN}$ of $\partial A$ 
satisfy that 
\begin{enumerate}[label=(pCY.\arabic*)]
	\item\label{item:pCY1} $m_{n}(A^{\otimes n}) \subseteq A$, for all $n \in \NN$; 
    \item\label{item:pCY2} given $n \geq 3$, $\ell \in \llbracket 0 , n-2 \rrbracket$ and $\bar{j} = (j_{1},\dots,j_{n-1}) \in \NN_{0}^{n-1}$ with $|\bar{j}| = \sum_{k=1}^{n-1} j_{k} = 2$ and $j_{k} = 0$ if $k \notin \{ 1, \ell+1 \}$, there exists a linear map $b_{\bar{j}} : A[1]^{\otimes 2} \rightarrow A^{\otimes (n-1)}$ such that 
    \begin{equation}
    \label{eq:CY}
    \begin{split}
    (f_{1} \otimes \dots \otimes f_{n-1})\big(&b_{\bar{j}}(sa \otimes sb)\big) 
    \\
    &= 
    \gan_{A}\big(m_{n}(a,tf_{1},\dots,tf_{\ell},b,tf_{\ell+1},\dots,tf_{n-2}),tf_{n-1}\big), 
    \end{split}
    \end{equation}
    for all $a, b \in A$ and $f_{1}, \dots, f_{n-1} \in A^{\#}$. 
    \end{enumerate}
\end{definition}

\begin{remark}
\label{remark:findim}
The previous definition implies in particular that $m_{1}$ and $m_{n}|_{A^{\otimes n}}$ for $n \neq 2$ vanish, and $(A,m_{2}|_{A^{\otimes 2}})$ is an associative algebra such that its canonical inclusion into $\partial A$ is a strict morphism of $A_{\infty}$-algebras. 
Furthermore, by the cyclicity property \eqref{eq:ip1}, $m_{2}|_{A \otimes A^{\#}[-1]}$ and $m_{2}|_{A^{\#}[-1] \otimes A}$ are the usual left and right actions of $A$ on $A^{\#}[-1]$, respectively (see the first paragraph of Subsection \ref{subsection:definition}).

We regard condition \ref{item:pCY2} as a finiteness assumption. 
Indeed, as explained in \cite{KV18}, one can construct (using \eqref{eq:CY} and adding the proper signs) an injective linear map $\mathscr{R}$ from the higher Hochschild complex of the underlying vector space of $A$, defined in \cite{KV18} (see also \cites{IK17, IKV19,IK19}), to the usual Hochschild complex of the underlying graded vector space of $\partial A$. 
Condition \ref{item:pCY2} only means that $m_{n}$ is in the image of such a map. 
The vanishing of the generalized necklace bracket of an element $b_{\bullet}$ in the higher Hochschild complex is tantamount to the Stasheff identities for 
$\mathscr{R}(b_{\bullet})$.
Notice that \ref{item:pCY2} is immediately satisfied if $A$ is finite dimensional, so this condition is redundant in that case. 
This also follows from the fact that the map $\mathscr{R}$ is an isomorphism if $A$ is finite dimensional. 
Moreover, since the natural bilinear form $\gan_{A}$ is nondegenerate if $A$ is finite dimensional, a pre-Calabi-Yau structure on a finite dimensional vector space $A$ gives an honest cyclic $A_{\infty}$-algebra structure on $\partial A$.
\end{remark}

\subsection{\texorpdfstring{Good $A_{\infty}$-algebras and pre-Calabi-Yau algebras}{Good A-infinity-algebras and pre-Calabi-Yau algebras}}
\label{subsection:good-algebras}

We will now recall the following terminology that will be useful in the sequel (see \cite{FH19}, Subsection 4.4). 
Let us first fix some notation. 
Assume that there is a decomposition $B_{0} \oplus B_{1}$ of a graded vector space $B$. 
In the following section, $B_{0}$ will be a graded vector space $A$ and $B_{1}$ will be $A^{\#}[-1]$. 
We will write $B_{\bar{i}} = B_{i_{1}} \otimes \dots \otimes B_{i_{n}}$, for all $\bar{i} = (i_{1},\dots,i_{n}) \in \{ 0 , 1 \}^{n}$, 
and we recall that $| \bar{i} | = \sum_{j = 1}^{n} i_{j}$. 
Then, for any integer $n \in \NN$, the decomposition $B = B_{0} \oplus B_{1}$ induces a canonical decomposition
\[     B^{\otimes n} = T_{n,g} \oplus T_{n,b},     \]
where 
\begin{equation}
   T_{n,g} = \bigoplus_{\bar{i} \in \mathscr{I}_{n}} B_{\bar{i}}, \hskip 5mm 
 T_{n,b} = \bigoplus_{\bar{i} \in \{0,1\}^{n} \setminus \mathscr{I}_{n}} B_{\bar{i}},  
\label{etiqueta-T-n-g}
\end{equation}  
and
\[     \mathscr{I}_{n} = \big\{ \bar{i} = (i_{1},\dots,i_{n})\in \{0,1\}^{n} : \text{ $i_{j} \neq i_{j+1}$ for all $j \in \{ 1, \dots, n - 1 \}$ } \big\}.     \]
Note that $T_{1,b} = 0$. 

Given $n \in \NN$, we say that a map $m_{n} : B^{\otimes n} \rightarrow B$ is \textbf{\textcolor{myblue}{acceptable}} if $m_{n}|_{B_{\bar{i}}}$ vanishes in the following two cases: there is $j \in \llbracket 1, n-1\rrbracket$ such that $i_{j} = i_{j+1} = 0$, or we have that $|\bar{i} | = n - 1$ and $i_{1} . i_{n} = 0$. 
This condition is void if $n=1$. 
Given $n \in \NN$ odd, a map $m_{n} : B^{\otimes n} \rightarrow B$ will be called \textbf{\textcolor{myblue}{good}} if $m_{n}|_{T_{n,b}}$ vanishes and $m_{n}(B_{i_{1}} \otimes \dots \otimes B_{i_{n}}) \subseteq B_{i_{1}}$, 
for all $(i_{1},\dots,i_{n}) \in \mathscr{I}_{n}$. 
Note that any good map $m_{n} : B^{\otimes n} \rightarrow B$ is acceptable. 
Moreover, if we consider the new grading on $B$ where $B_{0}$ is concentrated in degree zero and $B_{1}$ is concentrated in degree $1$, a map $m_{3} : B^{\otimes 3} \rightarrow B$ of degree $-1$ is good if and only if it is acceptable. 

\begin{definition}
\label{definition:terminology2}
Let $B$ be an $A_{\infty}$-algebra with a decomposition $B = B_{0} \oplus B_{1}$ of the underlying graded vector space. 
We say that $B$ is \textbf{\textcolor{myblue}{acceptable}} if for every integer $n \in \NN$ the multiplication map $m_{n}$ is acceptable. 
We say that $B$ is \textbf{\textcolor{myblue}{good}} if we further have that for every odd integer $n \in \NN$ the multiplication map $m_{n}$ is good.
\end{definition}

Note that all these definitions apply in particular to pre-Calabi-Yau structures on $A$, where we take $B_{0} = A$ and $B_{1} = A^{\#}[-1]$. 

\section{Main result: double quasi-Poisson algebras are pre-Calabi-Yau}
\label{section:core}

\subsection{\texorpdfstring{The definition of the $A_{\infty}$-algebra structure}{The definition of the A-infinity-algebra structure}}
\label{subsection:definition}

Let $A$ be an associative $\kk$-algebra with product $\mu_{A}$ and unit $1_{A}$. 
Assume that $A$ is provided with a double bracket $\lr{ \hskip 0.6mm , } : A^{\otimes 2} \rightarrow A^{\otimes 2}$. 
Note that $A$ is considered to be concentrated in degree zero.
The shifted dual $A^{\#}[-1]$ is a bimodule over $A$ via $ \gan(a \cdot t f \cdot b, c) = f(b c a)$, for all $a,b,c \in A$ and $f \in A^{\#}$, where we denote $s_{A^{\#},-1}$ simply by $t$.  
Then, $\partial A = A \oplus A^{\#}[-1]$ is a graded (associative) unitary $\kk$-algebra with 
product $m_{2}$, also denoted by a dot and given by $(a,tf) \cdot (a',tf') = (a a' , tf \cdot a' + a \cdot t f')$, for $a,a' \in A$ and $f, f' \in A^{\#}$, and with unit $(1_{A},0)$. 
Moreover, the bracket $\lr{ \hskip 0.6mm,}$ on $A$, induces the unique good map 
$m_{3} : \partial A^{\otimes 3} \rightarrow \partial A$ of cohomological degree $-1$ satisfying that 
\begin{equation}
\label{eq:m3IK}
     (f \otimes g)\big(\lr{a,b}\big) =  \gan\big(m_{3}(b,tg,a),tf\big) = \gan\big(m_{3}(tf,b,tg),a\big),
\end{equation}
for all $a, b \in A$ and $f, g \in A^{\#}$. 
Note that $m_{3}$ is $1_{A}$-normalized, \textit{i.e.} it vanishes if one of the arguments is the unit $1_{A}$ of $A$. 
The fact that $\lr{ \hskip 0.6mm,}$ is skew-symmetric (\textit{i.e.} \ref{item:double1}) tells us that $m_{3}$ satisfies \eqref{eq:ip1}. 

As a side note, we leave to the reader to verify the easy assertion that this graded algebra together with the natural bilinear form of degree $-1$ defined in \eqref{eq:natform} is in fact a pre-Calabi-Yau structure, by taking $m_{2}$ to be the product of $\partial A = A \oplus A^{\#}[-1]$, and $m_{n} = 0$, for all $n \neq 2$. 
Furthermore, combining this with the comments in the first paragraph of Remark \ref{remark:findim}, we see that a pre-Calabi-Yau structure (on a vector space concentrated in degree zero) with $m_{n} = 0$ for $n \neq 2$ is exactly the same thing as an associative algebra.

Recall that $B_{0} = A$, $B_{1} = A^{\#}[-1]$, and we set $B_{\bar{i}} = \otimes_{j=1}^{k} B_{i_{j}}$, for $k \in \NN$ and $\bar{i} =(i_{1},\dots,i_{k}) \in \{ 0, 1 \}^{k}$. 
To reduce space, an element $x_{1} \otimes \dots \otimes x_{k} \in B_{\bar{i}}$ will be denoted simply by 
$[x_{1}, \dots, x_{k}]$. 
There is an obvious action of the cyclic group $C_{k}$ on $\{ 0, 1 \}^{k}$ by cyclic permutations, \textit{i.e.} $\sigma \cdot (i_{1},\dots,i_{k}) = (i_{\sigma^{-1}(1)},\dots,i_{\sigma^{-1}(k)})$, for all $\sigma \in C_{k}$. 
Let $C_{k}.\bar{i}$ be an orbit of this action and set
\[     \bar{B}_{\bar{i}} = \bigoplus_{\bar{j} \in C_{k}.\bar{i}} B_{\bar{j}}.     \]
Then $\bar{B}_{\bar{i}}$ has an action of $C_{k}$, given by $\sigma \cdot [x_{1},\dots,x_{k}] = (-1)^{K} [x_{\sigma^{-1}(1)},\dots,x_{\sigma^{-1}(k)}]$, for all $\sigma \in C_{k}$, where $K$ is the Koszul sign. 
Define $B_{\bar{i}}^{\circ} = \bar{B}_{\bar{i}}/C_{k}$ the space of coinvariants under this action and $\pi_{\bar{i}} : B_{\bar{i}} \rightarrow B_{\bar{i}}^{\circ}$ as the composition of the inclusion $B_{\bar{i}} \subseteq \bar{B}_{\bar{i}}$ and the canonical projection $\bar{B}_{\bar{i}} \rightarrow B_{\bar{i}}^{\circ}$. 
The class $\pi_{\bar{i}}([x_{1}, \dots, x_{k}])$ of $[x_{1}, \dots, x_{k}] \in B_{\bar{i}} \subseteq \bar{B}_{\bar{i}}$ will be denoted by $\langle x_{1}, \dots, x_{k} \rangle$. 

For $i, j \in \NN$ such that $i+j$ is odd, we define 
\begin{equation}
\label{eq:newcij}
     C_{i,j} = \begin{pmatrix} i+j - 2 \\ i - 1 \end{pmatrix} (-1)^{1+ \frac{i+j-1}{2}} \frac{\mathcal{B}_{i+j-1}}{(i+j-1)!} \para^{\frac{i+j-1}{2}},
\end{equation}
where $\mathcal{B}_{\ell}$ denotes the $\ell$-th Bernoulli number. 
We refer the reader to \cite{AIK14}, for a nice introduction on Bernoulli numbers. 
It is clear that $C_{1,2} = \para/12$, and they satisfy $C_{i,j}= C_{j,i}$, for all $i, j \in \NN$ such that $i+j$ is odd. 

For $k \in \ZZ_{\geq 3}$ odd, let $\bar{e}_{k} \in \{ 0, 1 \}^{k}$ be the element all of whose components are $1$. 
The linear map 
\begin{equation}
\label{eq:ev}
   \operatorname{ev}_{k} : B_{\bar{e}_{k}}^{\circ} \rightarrow \kk
\end{equation}
sending $\langle tf_{1}, \dots , tf_{k} \rangle$ to $f_{1}(1_{A}) \dots f_{k}(1_{A})$ is clearly well-defined and is homogeneous of degree $-k$. 
Moreover, for $k \in \ZZ_{\geq 3}$ odd and $\bar{i} \in \{ 0, 1 \}^{k+2}$ such that $|\bar{i}| = i_{1} + \dots + i_{k+2} = k$, the linear map 
\begin{equation}
\label{eq:mu}
   \mathscr{M}_{k} : B_{\bar{i}}^{\circ} \rightarrow B_{\bar{e}_{k}}^{\circ}
\end{equation}
given by sending $\langle tf_{1}, \dots, tf_{i},a,tg_{1},\dots,tg_{j},b \rangle$, where $j \geq i$ and $i+j=k$, to $C_{i,j}$ times
\begin{equation}
\label{eq:mu2}
   \begin{cases}
   0, &\text{if $i = 0$,}
   \\
   &
   \\
   \begin{aligned}
       \langle &tf_{1} \cdot a,tg_{1},\dots,tg_{j-1},tg_{j} \cdot b \rangle 
        + \langle b \cdot tf_{1}, a \cdot tg_{1},tg_{2},\dots,tg_{j} \rangle 
        \\
       - &\langle tf_{1}, a \cdot tg_{1},tg_{2},\dots,tg_{j-1},tg_{j} \cdot b \rangle 
        - \langle b \cdot tf_{1} \cdot a, tg_{1},\dots,tg_{j} \rangle, 
    \end{aligned} 
    &\text{if $i =1$,}
   \\
   &
   \\
   \begin{aligned}
       \langle &tf_{1}, \dots, tf_{i-1},tf_{i} \cdot a,tg_{1},\dots,tg_{j-1},tg_{j} \cdot b \rangle 
       \\
       + &\langle b \cdot tf_{1},  tf_{2}, \dots,tf_{i}, a \cdot tg_{1},tg_{2},\dots,tg_{j} \rangle 
       \\
       - &\langle tf_{1}, \dots, tf_{i}, a \cdot tg_{1},tg_{2},\dots,tg_{j-1},tg_{j} \cdot b \rangle 
       \\
       - &\langle b \cdot tf_{1}, tf_{2}, \dots, tf_{i-1},tf_{i} \cdot a, tg_{1},\dots,tg_{j} \rangle, 
    \end{aligned} 
    &\text{if $i \geq 2$,}
   \end{cases}
\end{equation} 
is well-defined for all $f_{1}, \dots, f_{i}, g_{1}, \dots, g_{j} \in A^{\#}$ and $a, b \in A$. 
It is clear that \eqref{eq:mu} is homogeneous of degree zero. 
Note that the parity of $k$ implies that the Koszul sign in the action of the cyclic group $C_{k}$ does not produce any negative sign. 

For any even integer $n \geq 4$, define $m_{n} : \partial A^{\otimes n} \rightarrow \partial A$ 
as the unique map satisfying that 
\begin{equation}
\label{eq:mun}
   \gan \circ (m_{n} \otimes \mathrm{id}_{\partial A})|_{B_{\bar{i}}} = 
  \operatorname{ev}_{n-1} \circ \mathscr{M}_{n-1} \circ \pi_{\bar{i}},
\end{equation}
for all $\bar{i} \in \{ 0 , 1 \}^{n+1}$ such that $| \bar{i} | = n - 1$, and $\gan \circ (m_{n} \otimes \mathrm{id}_{\partial A})|_{B_{\bar{i}}} = 0$ if $| \bar{i} | \neq  n - 1$. 
Note that the degree of $m_{n}$ is precisely $2-n$ for all even integers $n \geq 4$, it is normalized with respect to $1_{A}$
(\textit{i.e.} it vanishes if one of its arguments is $1_{A}$), it is acceptable and it satisfies \eqref{eq:ip1}. 
Note that $m_{n}$ vanishes for all $n \geq 4$ if $\para = 0$, so in that case the full structure on $\partial A$ given by $\{ m_{n} \}_{n \in \NN}$ reduces to the one introduced in \cite{IK17}.

\subsection{The statement of the main result}

The next result establishes a remarkable link between double quasi-Poisson algebras and pre-Calabi-Yau algebras, which can be regarded as an extension of \cite{IK17}, Thm. 4.2 (see also \cite{IKV19}, Thm. 4.2). 
The correspondence between double Poisson algebras and certain pre-Calabi-Yau algebras was worked out in the differential graded setting in \cite{FH19}, Thm. 5.1, and for double $P_{\infty}$-algebras in \cite{FH19}, Thm. 6.3. 
For a very interesting conceptual explanation of the relation between pre-Calabi-Yau structures and noncommutative Poisson structures 
in terms of the higher cyclic Hochschild cohomology, see \cite{IK19}.

\begin{theorem}
\label{theorem:Main}
Let $A$ be a(n associative) $\kk$-algebra with unit $1_{A}$ and let $\lr{ \hskip 0.6mm , } : A^{\otimes 2} \rightarrow A^{\otimes 2}$ be a double bracket. 
Assume that $A$ is double quasi-Poisson. 
Then, $\partial A$ provided with the usual multiplication $m_{2}$, as well as the maps $m_{3}$ and $\{ m_{n} : n \in 2.\NN_{\geq 2} \}$ defined in \eqref{eq:m3IK} and \eqref{eq:mun}, respectively, is a strictly unitary $A_{\infty}$-algebra and it defines a structure of pre-Calabi-Yau algebra on $A$. 
\end{theorem}

Note that, if $\para = 0$, the previous result is exactly \cite{IK17}, Thm. 4.2. 
Moreover, if $\para = 0$, our proof below reduces to that in \cite{IK17}, where all the Stasheff identities (\hyperlink{eq:ainftyalgebralink}{$\operatorname{SI}(N)$}) with $N \geq 6$ immediately vanish, as $m_{n} = 0$ for $n \geq 4$. 
This is in stark contrast with the case $\para \neq 0$, for which an infinite number of Stasheff identities needs to be inspected. 
Our proof consists in carefully studying all the Stasheff identities and showing that our choice of higher multiplications $m_{n}$ for $n \geq 4$ 
makes all of the Stasheff identities vanish, regardless of the value of $\para$.

\subsection{The proof of the main result} 

	We first note the following simple result.
	\begin{fact}
		\label{fact:pCY2}
		The maps $m_{n}$ defined in \eqref{eq:m3IK} for $n= 3$, and \eqref{eq:mun} for $n \geq 4$, satisfy conditions \ref{item:pCY1} and \ref{item:pCY2}.  
	\end{fact}
	\begin{proof}
		The verification of condition \ref{item:pCY1} for $n \geq 3$ is a direct consequence of definitions \eqref{eq:m3IK} and \eqref{eq:mun}. 
		We now proceed to prove \ref{item:pCY2}. 	
		The case of $m_{3}$ is clear, since in that case $b_{(2,0)} = 0$ and $b_{(1,1)} = - \lr{\hskip 0.6mm , } \circ (s_{A}^{-1} \otimes s_{A}^{-1})$.
		For $n \geq 4$, recall that we have $\ell \in \llbracket 0 , n-2 \rrbracket$ and $\bar{j} = (j_{1},\dots,j_{n-1}) \in \NN_{0}^{n-1}$ with $|\bar{j}| = \sum_{k=1}^{n-1} j_{k} = 2$ and $j_{k} = 0$ if $k \notin \{ 1, \ell + 1 \}$. 
		Then, using \eqref{eq:CY} and \eqref{eq:mun} we get that $b_{(2,0,\dots,0)} = 0$, 
		\begin{small}
			\begin{equation}
			\begin{split}
			&b_{\text{\begin{tiny}$(1,1,\underset{(n-3) \text{ $0$'s}}{\underbrace{\text{\begin{tiny}$0,\dots,0$\end{tiny}}}})$\end{tiny}}}(sa , sb) 
			\\
			&= 
			C_{1,n-2} \bigg( a \otimes b \otimes 1_{A}^{\otimes (n-3)} -  b a \otimes 1_{A}^{\otimes (n-2)} +  b \otimes  1_{A}^{\otimes (n-3)} \otimes a -  1_{A} \otimes b \otimes 1_{A}^{\otimes (n-4)} \otimes a \bigg),     
			\end{split}
			\end{equation}
		\end{small}
		and
		\begin{small}
			\begin{equation}
			\begin{split}
			&b_{\text{\begin{tiny}$(1,\underset{(n-3) \text{ $0$'s}}{\underbrace{\text{\begin{tiny}$0,\dots,0$\end{tiny}}}},1)$\end{tiny}}}(sa , sb) 
			\\
			&= 
			C_{1,n-2} \bigg( a \otimes 1_{A}^{\otimes (n-3)}  \otimes b -  a \otimes 1_{A}^{\otimes (n-4)} \otimes b  \otimes 1_{A} +  1_{A}^{\otimes (n-3)}  \otimes b \otimes a -  1_{A}^{\otimes (n-2)} \otimes a b \bigg),     
			\end{split}
			\end{equation}
		\end{small}
		as well as 
		\begin{small}
			\begin{equation}
			\begin{split}
			b_{\bar{j}}(sa , sb) = 
			C_{\ell,n-\ell-1} \bigg(& a \otimes 1_{A}^{\otimes (\ell-2)} \otimes \Big(1_{A} \otimes b - b \otimes 1_{A}\Big) \otimes 1^{\otimes (n-\ell-2)} 
			\\
			&-  1_{A}^{\otimes (\ell-1)} \otimes \Big(1_{A} \otimes b - b \otimes 1_{A}\Big) \otimes 1^{\otimes (n-\ell-3)} \otimes a  \bigg),     
			\end{split}
			\end{equation}
		\end{small}
		if $\ell \in \llbracket 2 , n-3 \rrbracket$.
		The claim is thus proved. 
	\end{proof}

Fact \ref{fact:pCY2} tells us that, in order to prove the main theorem, it suffices to show that the structure given on $\partial A$ 
by means of \eqref{eq:m3IK} and \eqref{eq:mun} is an $A_{\infty}$-algebra, since we have already observed in Subsection \ref{subsection:definition} that \eqref{eq:ip1} is verified with respect to the natural bilinear form.  
By hypothesis, we assume that the identity in Definition \ref{double-quasi-poisson-def} is verified (see also Definition \ref{def:n-bracket}). 
We are going to prove that the Stasheff identities (\hyperlink{eq:ainftyalgebralink}{$\operatorname{SI}(N)$}) for $\partial A = A \oplus A^{\#}[-1]$ hold. 
It is clear that (\hyperlink{eq:ainftyalgebralink}{$\operatorname{SI}(N)$}) holds for $N=1,2,3$, because $m_{1}$ vanishes and $m_{2}$ is associative. 
Thus, it remains to prove that (\hyperlink{eq:ainftyalgebralink}{$\operatorname{SI}(N)$}) is satisfied for all integers $N \geq 4$. 
Instead of working with the Stasheff identities (\hyperlink{eq:ainftyalgebralink}{$\operatorname{SI}(N)$}), we will work with 
the equivalent identity (\hyperlink{eq:ainftyalgebragammalink}{$\operatorname{SI}(N)_{\gan}$}). 
This is done in Lemmas \ref{lemma:si4}, \ref{lemma:si5} and \ref{lemma:sieven}, and Proposition \ref{proposition:siodd}. 

It is clear that it suffices to prove that $\operatorname{SI}(N)_{\gan}|_{B_{\bar{i}}}$ vanishes for all $N \in \NN_{\geq 4}$ and $\bar{i} \in \{ 0, 1\}^{N+1}$, where $B_{\bar{i}} = B_{i_{1}} \otimes \dots \otimes B_{i_{N+1}}$, $B_{0} = A$ and $B_{1} = A^{\#}[-1]$.
Since \hyperlink{eq:ainftyalgebragammalinkop}{$\operatorname{SI}(N)_{\gan}$} has degree $2-N$, $\operatorname{SI}(N)_{\gan}|_{B_{\bar{i}}}$ vanishes for all $N \in \NN_{\geq 4}$ and $\bar{i} \in \{ 0, 1\}^{N+1}$ such that $| \bar{i} | = \sum_{j=1}^{N+1} i_{j} \neq N-2$, we shall prove that $\operatorname{SI}(N)_{\gan}|_{B_{\bar{i}}}$ vanishes under the assumption that $| \bar{i} | = N-2$. 
 
\subsubsection{\texorpdfstring{The Stasheff identity $\operatorname{SI}(4)$}{The Stasheff identity SI(4)}} 

The following lemma was proved in the main result of \cite{IK17} (see also \cite{IKV19}, Thm. 4.2 and \cite{FH19}, Thm. 5.2), but we recall its proof for the reader's convenience. 
 
\begin{lemma}
\label{lemma:si4}
Assume the same hypotheses as in Theorem \ref{theorem:Main}. 
We use the previously introduced notation. 
Then, the Stasheff identity (\hyperlink{eq:ainftyalgebragammalink}{$\operatorname{SI}(4)_{\gan}$}) for $\partial A = A \oplus A^{\#}[-1]$ holds. 
\end{lemma}
\begin{proof} 
As noted at the beginning of the subsection, it suffices to prove that $\operatorname{SI}(4)_{\gan}|_{B_{\bar{i}}}$ vanishes for all $\bar{i} \in \{ 0, 1\}^{5}$ such that $| \bar{i} | = \sum_{j=1}^{5} i_{j} = 2$, so we will assume from now on that $| \bar{i} | = 2$. 
Furthermore, using that $m_{3}$ is good, it is easy to show that every term of the identity (\hyperlink{eq:ainftyalgebragammalink}{$\operatorname{SI}(4)_{\gan}$}) restricted to $B_{\bar{i}}$ vanishes if $\bar{i} \in \{ 0,1\}^{5} \setminus \mathscr{B}_{5}$, where $\mathscr{B}_{5} = \{ (0,0,1,0, 1), (0,1,0,1,0), (1,0,1,0,0), (0,1,0,0,1), (1,0,0,1,0) \}$. 
It is easy to see that the cyclic group $C_{5}$ acts transitively on $\mathscr{B}_{5}$, so by Lemma \ref{lemma:cyc}, given $\bar{i}, \bar{i}' \in \mathscr{B}_{5}$, $\operatorname{SI}(4)_{\gan}|_{B_{\bar{i}}}$ vanishes if and only if $\operatorname{SI}(4)_{\gan}|_{B_{\bar{i}'}}$ does so. 
It thus suffices to prove that $\operatorname{SI}(4)_{\gan}|_{B_{\bar{i}}} = 0$ for one element $\bar{i}$ of $\mathscr{B}_{5}$. 
Let us choose $\bar{i} = (0,0,1,0,1)$. 
Then, $B_{\bar{i}} = A\otimes A\otimes A^{\#}[-1]\otimes A \otimes A^{\#}[-1]$ and 
consider $a \otimes b \otimes tf \otimes c \otimes tg \in B_{\bar{i}}$, where $a, b, c \in A$ and $f, g \in A^{\#}$. 
We will show that $\operatorname{SI}(4)_{\gan}(a \otimes b \otimes tf \otimes c \otimes tg) = 0$. 
Since $m_{3}$ is good, $m_3(a,b,(tf) \cdot c)=0=m_3(a,b,tf).c$, so (\hyperlink{eq:ainftyalgebragammalink}{$\operatorname{SI}(4)_{\gan}$}) evaluated at $a \otimes b \otimes tf \otimes c \otimes tg$ is tantamount to 
\begin{equation}
\label{eq:leib2-gan}
\gan\big( m_{3}(a b, t f, c) , tg\big)-\gan\big( m_{3}(a, b \cdot tf, c) , tg\big) -\gan\big(  a. m_{3}(b,tf,c) , tg\big) =0.
\end{equation} 
To prove that \eqref{eq:leib2-gan} holds, we will use the Leibniz property \ref{item:double2} in the definition of double bracket, \textit{i.e.} 
\begin{equation}
\label{eq:leib}
     \lr{c,ab} = \lr{c,a} b +  a\lr{c,b},
\end{equation}     
for all $a, b, c \in A$.
By applying $g\otimes f$, for arbitrary $f,g \in A^{\#}$, we see that \eqref{eq:leib} is tantamount to 
\begin{equation}
\label{eq:leib3}
   (g\otimes f)\big(\lr{c,ab}\big)-(g\otimes f)\big(\lr{c,a}b\big)-(g\otimes f)\big(a\lr{c,b}\big)=0.
\end{equation}     
By \eqref{eq:m3IK}, we have
\begin{equation}
\label{eq:leib31}
 (g\otimes f)\big(\lr{c,ab}\big)= \gan\big(m_{3}(a b, t f, c),tg\big),
\end{equation} 
which is the first term in \eqref{eq:leib2-gan}. 
Using that $(g \otimes f)(v \otimes (w.b))=(g \otimes (b \cdot f))(v \otimes w) $, for all $b, v , w \in A$ and $f, g \in A^{\#}$, we have that 
\begin{equation}
\label{eq:leib32bis}
-(g\otimes f)\big(\lr{c,a}b\big) =-\big(g \otimes (b \cdot f)\big)\big(\lr{c,a}\big) =-\gan\big(m_{3}(a, b \cdot tf, c),tg\big),
\end{equation}
which is the second term in \eqref{eq:leib2-gan}. 
In the last identity we used that $t(b \cdot f)=b \cdot tf $, for all $b \in A$ and $f \in A^{\#}$, and we applied \eqref{eq:m3IK}.

Finally, if we apply the identity $(g \otimes f)(a.v \otimes w)=((g \cdot a) \otimes f)(v \otimes w) $, for all $a, v , w \in A$ and $f, g \in A^{\#}$ to the third term of \eqref{eq:leib3}, together with \eqref{eq:m3IK}, we get
\begin{equation}
\label{eq:leib33}
\begin{split}
-(g\otimes f)\big(a\lr{c,b}\big) &=-\big((g \cdot a) \otimes f\big)\big(\lr{c,b}\big) = -\gan\big(m_3(b,tf,c), (tg) \cdot a\big) 
     \\
     &= - \gan\big((tg) \cdot a,m_{3}(b,tf,c)\big) =-\gan\big(a.m_3(b,tf,c),tg\big), 
\end{split}
\end{equation}
where we used the identity $(tg) \cdot a = t(g \cdot a)$ in the second equality, the supersymmetry of $\gan$ in the third equality and 
the cyclicity of $\gan$ with respect to $m_{2}$ in the last equality. 
Combining \eqref{eq:leib3} with \eqref{eq:leib31}, \eqref{eq:leib32bis} and \eqref{eq:leib33}, we conclude that 
the sum of the last three equations is zero, which proves \eqref{eq:leib2-gan}, as we claimed. 
The lemma is thus proved. 
\end{proof}

\subsubsection{\texorpdfstring{General setting for the Stasheff identity $\operatorname{SI}(N)$ with $N>4$}{General setting for the Stasheff identity SI(4) with N > 4}} 
\label{subsubsec:generalsetting}

We will now prove Stasheff's identities (\hyperlink{eq:ainftyalgebragammalink}{$\operatorname{SI}(N)_{\gan}$}) for all integers 
$N \geq 5$. 
Since the only nonzero higher multiplication maps are $m_{3}$ and $m_{n}$ with $n \in 2. \NN_{\geq 2}$, the only nontrivial Stasheff identities (\hyperlink{eq:ainftyalgebragammalink}{$\operatorname{SI}(N)_{\gan}$}) with $N \geq 5$ are 
\begin{enumerate}[label=(C\arabic*)]
\item\label{item:sieven} (\hyperlink{eq:ainftyalgebragammalink}{$\operatorname{SI}(n+2)_{\gan}$}) for all even $n \in \NN_{\geq 4}$, which involves only $m_{3}$ and $m_{n}$;

\item\label{item:siodd} (\hyperlink{eq:ainftyalgebragammalink}{$\operatorname{SI}(n+1)_{\gan}$}) for all even $n \in \NN_{\geq 4}$, which involves only $m_{2 i}$, for $i \in \llbracket 1, n/2 \rrbracket$, and also $m_{3}$ if $n=4$.
\end{enumerate}

\begin{remark}
\label{remark:tau1}
For $n \in \NN$ even, define $\bar{m}_{n}$ as the multiplication $m_{n}$ given in Subsection \ref{subsection:definition} when $\para = 1$. 
Then, \eqref{eq:newcij} tells us that $m_{n} = \para^{(n-2)/2} \bar{m}_{n}$, for all $n \in \NN$ even, where we use the convention $0^{0} = 1$ for $\para = 0$ and $n=2$. 
As explained in \ref{item:siodd}, \hyperlink{eq:ainftyalgebralink}{$\operatorname{SI}(N)$} for the higher multiplications $\{ m_{n} \}_{n \in \NN}$ and $N \geq 7$ odd is of the form 
\begin{equation}
\begin{split}
     \sum_{i=1}^{(N-1)/2} &\sum_{r=0}^{2i-1} m_{2 i} \circ (\mathrm{id}_{\partial A}^{\otimes r} \otimes m_{N+1-2 i} \otimes \mathrm{id}_{\partial A}^{\otimes (i-r+(N+1)/2)}) 
     \\
     &= \para^{(N-3)/2} \sum_{i=1}^{(N-1)/2} \sum_{r=0}^{2i-1} \bar{m}_{2 i} \circ (\mathrm{id}_{\partial A}^{\otimes r} \otimes \bar{m}_{N+1-2 i} \otimes \mathrm{id}_{\partial A}^{\otimes (i-r+(N+1)/2)}).  
     \end{split}   
\end{equation}
In consequence, if $N \geq 7$ is odd, (\hyperlink{eq:ainftyalgebralink}{$\operatorname{SI}(N)$}) for the higher multiplications $\{ m_{n} \}_{n \in \NN}$ (and any value of $\para$) holds if and only if (\hyperlink{eq:ainftyalgebralink}{$\operatorname{SI}(N)$}) for the higher multiplications $\{ \bar{m}_{n} \}_{n \in \NN}$ holds. 
In other words, it suffices to prove (\hyperlink{eq:ainftyalgebralink}{$\operatorname{SI}(N)$}) for $\para = 1$, if the integer $N \geq 7$ is odd. 
\end{remark}

As noted at the beginning of this subsection, it suffices to show that $\operatorname{SI}(N)_{\gan}(\omega)$ vanishes for any $\omega \in B_{\bar{i}}$, with $\bar{i} = (i_{1}, \dots, i_{N+1}) \in \{ 0 , 1 \}^{N+1}$ and $| \bar{i} | = N-2$. 
By applying a cyclic permutation to $\omega$ and using Lemma \ref{lemma:cyc}, we may assume without loss of generality that 
$i_{1} = i_{\ell+2} = i_{\ell+\ell'+3} = 0$, with $\ell \leq \min(\ell',\ell'')$ and if $\ell < \ell'$ then $\ell < \ell''$, where $\ell, \ell', \ell'' \in \NN_{0}$ satisfy that $\ell + \ell' + \ell'' = N-2$.
More precisely, we consider an element $\Omega \in A \times A^{\#}[-1]^{\ell} \times A \times A^{\#}[-1]^{\ell'} \times A \times A^{\#}[-1]^{\ell''}$ of the form 
\begin{equation}
\label{eq:Omega}
     \Omega = (a, tf_{1},\dots, tf_{\ell}, b, tg_{1}, \dots, tg_{\ell'}, c, th_{1}, \dots, th_{\ell''}), 
\end{equation}
for $f_{1},\dots, f_{\ell}, g_{1}, \dots, g_{\ell'}, h_{1}, \dots, h_{\ell''} \in A^{\#}$ and $a,b,c \in A$, such that $\ell \leq \min(\ell',\ell'')$ and if $\ell < \ell'$ then $\ell < \ell''$, where $\ell, \ell', \ell'' \in \NN_{0}$ satisfy that $\ell + \ell' + \ell'' = N-2$. 
The class of $\Omega$ in $A \otimes A^{\#}[-1]^{\otimes \ell} \otimes A \otimes A^{\#}[-1]^{\otimes \ell'} \otimes A \otimes A^{\#}[-1]^{\otimes \ell''}$ is precisely 
\begin{equation}
\label{eq:omega}
     \omega = \omega' \otimes th_{\ell''}, \text{ with } \omega' = [a, tf_{1},\dots, tf_{\ell}, b, tg_{1}, \dots, tg_{\ell'}, c, th_{1}, \dots, th_{\ell''-1}].
\end{equation}

\subsubsection{\texorpdfstring{The Stasheff identity $\operatorname{SI}(5)$}{The Stasheff identity SI(5)}} 

The proof of the next result follows the same pattern as the proof of the main result in \cite{IK17} (see also \cite{FH19}, Thm. 5.2), but 
the presence of $m_{4}$ makes it somehow subtler. 

\begin{lemma}
\label{lemma:si5}
Assume the same hypotheses as in Theorem \ref{theorem:Main}. 
We use the notation introduced in Subsubsection \ref{subsubsec:generalsetting}. 
Let $\omega \in B_{\bar{i}}$ be an element of the form \eqref{eq:omega}, with $\bar{i} = (i_{1}, \dots, i_{6}) \in \{ 0 , 1 \}^{6}$ and $| \bar{i} | = 3$. 
Then, $\operatorname{SI}(5)_{\gan}(\omega)$ vanishes. 
\end{lemma}
\begin{proof} 
By the previous comments it suffices to prove that $\operatorname{SI}(5)_{\gan}(\omega)$ vanishes for $\omega$ of the form \eqref{eq:omega} and one of the following cases
\begin{enumerate}[label=(\textrm{SI}(5).\alph*)]
\item\label{item:caso1} $\ell = \ell' = 0$, and $\ell''= 3$;
\item\label{item:caso2} $\ell = 0$, $\ell' = 1$, and $\ell''= 2$;
\item\label{item:caso3} $\ell = 0$, $\ell' = 2$, and $\ell''= 1$;
\item\label{item:caso4} $\ell = \ell' = \ell''= 1$.
\end{enumerate}
First, we will consider \ref{item:caso4}.
We will show that $\operatorname{SI}(5)_{\gan}(\omega)$ coincides with 
evaluating $h \otimes g \otimes f$ at the left member of identity \eqref{eq:double-quasi-poisson-def} evaluated at $c \otimes b \otimes a$, where $\omega = a\otimes tf\otimes b\otimes tg\otimes c \otimes th$, $a,b,c\in A$ and $f,g, h \in A^{\#}$. 
The identity $\operatorname{SI}(5)_{\gan}(\omega) = 0$ is exactly
 \begin{small}
 \begin{equation}
  \label{eq:SI-5-caso-1-gan}
 \begin{aligned}
&\gan\big(m_3\big(m_3(a,tf,b),tg,c\big),th\big)
 +\gan\big(m_3\big(a,m_3(tf,b,tg),c\big),th\big) 
 - \gan\big(m_3\big(a,tf,m_3(b,tg,c)\big),th\big)
 \\
&= \gan\big(a.m_4(tf,b,tg,c),th\big) - \gan\big( m_4(a,tf,b,tg).c,th\big)
- \gan\big(m_4 (a \cdot tf,b,tg,c ) ,th\big)
 \\
 &\phantom{=} + \gan\big(m_4(a,tf \cdot b,tg,c),th\big)
 - \gan\big(m_4(a,tf,b \cdot tg,c),th\big)
 + \gan\big(m_4(a,tf,b,tg \cdot c),th\big).
\end{aligned}
\end{equation}
\end{small}
\hskip -0.8mm We will prove that this equation is precisely the value of $h \otimes g \otimes f$ at the identity \eqref{eq:double-quasi-poisson-def} evaluated at $c \otimes b \otimes a$. 
Since $(A,\lr{ \hskip 0.6mm ,})$ is a double quasi-Poisson algebra, using \eqref{triple-bracket} and \eqref{RHS-quasi-Poisson-double} in \eqref{eq:double-quasi-poisson-def}, we obtain the identity
\begin{equation}
\label{eq:quasi-Poisson-caso-1}
 \begin{aligned}
\lr{c,\lr{b,a}}_L&+ \sigma \lr{b,\lr{a,c}}_L+ \sigma^{2} \lr{a,\lr{c,b}}_L
\\
&=\frac{\para}{4}\Big(ac\otimes b\otimes 1_{A}-ac\otimes 1_{A}\otimes b-a\otimes cb\otimes 1_{A}+a\otimes c\otimes b
 \\
 &\phantom{= \frac{\para}{4}} +c\otimes 1_{A}\otimes ba-c\otimes b\otimes a+1_{A}\otimes cb\otimes a-1_{A}\otimes c\otimes ba\Big),
 \end{aligned}
\end{equation}
where $\sigma \in \mathbb{S}_{3}$ is the unique cyclic permutation sending $1$ to $2$. 
If we apply $h\otimes g\otimes f$ to \eqref{eq:quasi-Poisson-caso-1} we get
\begin{equation}
\label{eq:quasi-Poisson-caso-1-gan}
 \begin{aligned}
(h\otimes &g\otimes f)\Big(\lr{c,\lr{b,a}}_L+ \sigma \lr{b,\lr{a,c}}_L+ \sigma^{2} \lr{a,\lr{c,b}}_L\Big)
 \\
 &=\frac{\para}{4}\Big(f(1_{A})g(b)h(ac)-f(b)g(1_{A})h(ac)-f(1_{A})g(cb)h(a)+f(b)g(c)h(a)
\\
&\phantom{= \frac{\para}{4}} +f(ba)g(1_{A})h(c)-f(a)g(b)h(c)+f(a)g(cb)h(1_{A})-f(ba)g(c)h(1_{A})\Big).
 \end{aligned}
 \end{equation}
 
Next, it is straightforward to see that in the particular case when $n=4$, \eqref{eq:mu2} gives
\begin{small}
\begin{equation}
\label{eq:m4-caso-particular}
\begin{split}
m_4(tf,b,tg,c)&=\frac{\para}{12}\big(f(b)g(c)1_A+f(1_A)g(b)c-f(b)g(1_A)c-f(1_A)g(cb)1_A\big),
\\
m_4(a,tf,b,tg)&=\frac{\para}{12}\big(f(b)g(1_A)a+f(a)g(b)1_A-f(1_A)g(b)a-f(ba)g(1_A)1_A\big),
\\
m_4(a,tf,tg,c)&=\frac{\para}{12}\big(f(1_A)g(c)a+f(a)g(1_A)c-f(a)g(c)1_A-f(1_A)g(1_A)ac\big),
\end{split}
\end{equation}
\end{small}
\hskip -0.8mm where $a,b,c\in A$, $f,g,h\in A^{\#}$. 
Now, if we plug \eqref{eq:m4-caso-particular} into the right-hand side of \eqref{eq:SI-5-caso-1-gan}, the latter gives 
\begin{align*}
&\frac{\para}{12}\gan\Big(a\big(f(b)g(c)1_A-f(b)g(1_A)c+f(1_{A})g(b)c-f(1_A)g(cb)1_A\big),th\Big)
\\
&-\frac{\para}{12}\gan\Big(\big(f(b)g(1_A)a-f(1_A)g(b)a+f(a)g(b)1_A-f(ba) g(1_A) 1_A\big)c,th\Big)
\\
&-\frac{\para}{12}\gan\Big(\big(f(ba)g(c)1_A-f(ba)g(1_A)c+f(a)g(b)c-f(a)g(cb)1_A\big),th\Big)
\\
&+\frac{\para}{12}\gan\Big( \big(f(b)g(c)a-f(ba)g(c)1_A+f(ba)g(1_A)c-f(b)g(1_A)ac\big),th\Big)
\\
&-\frac{\para}{12}\gan\Big(\big(f(1_A)g(cb)a-f(a)g(cb)1_A+f(a)g(b)c-f(1)g(b)ac\big),th\Big)
\\
&+\frac{\para}{12}\gan\Big(\big(f(b)g(c)a-f(1_A)g(cb)a+ f(a) g(cb) 1_A-f(ba) g(c) 1_{A} \big),th\Big).
\end{align*}
Applying the definition of $\gan$ given in \eqref{eq:natform}, the previous expression is equivalent to 
 \begin{equation}
  \label{eq:RHS-quasi}
 \begin{split}
 &\frac{\para}{12}\Big(3f(1_{A})g(b)h(ac)-3f(b)g(1_{A})h(ac)-3f(1_{A})g(cb)h(a)+3f(b)g(c)h(a)
\\
&\phantom{\frac{\para}{12}} + 3f(ba)g(1_{A})h(c)-3f(a)g(b)h(c)+3f(a)g(cb)h(1_{A})-3f(ba)g(c)h(1_{A})\Big),
 \end{split}
 \end{equation}
which is clearly equal to the right-hand side of \eqref{eq:quasi-Poisson-caso-1-gan}.
We will now deal with the left-hand side of \eqref{eq:quasi-Poisson-caso-1-gan}, and show it is equal to the left-hand side of \eqref{eq:SI-5-caso-1-gan}. 
To do that, first we recall the following result (see \cite{IK17}, Lemma 4.4, or \cite{FH19}, Fact 5.4, for the general case).
\begin{fact}
\label{lemma:lema-tecnico-Jacobi-statement}
Assume the same hypotheses as in Theorem \ref{theorem:Main}. 
Then,
\begin{equation}
\label{eq:lema-tecnico-Jacobi}
  (f\otimes g\otimes h)\big(\lr{a,\lr{b,c}}_L\big)=\gan\Big(m_3\big(m_3(c,th,b),tg,a\big),tf\Big),
\end{equation}
for all $a,b,c \in A$ and $f,g,h \in A^{\#}$. 
\end{fact}

We claim that the left-hand side of \eqref{eq:SI-5-caso-1-gan} coincides with the left-hand side of \eqref{eq:quasi-Poisson-caso-1-gan}, \textit{i.e.}
\begin{equation}
  \label{eq:igualdad-m3-triple-bracket}
 \begin{split}
&\gan\Big(m_3\big(m_3(a,tf,b),tg,c\big),th\Big)
 + \gan\Big(m_3\big(a,m_3(tf,b,tg),c\big),th\Big) 
 \\
 &-  \gan\Big(m_3\big(a,tf,m_3(b,tg,c)\big),th\Big)
 \\
 &=(h\otimes g\otimes f)\Big(\lr{c,\lr{b,a}}_L+ \sigma \lr{b,\lr{a,c}}_L+ \sigma^{2} \lr{a,\lr{c,b}}_L\Big),
 \end{split}
\end{equation}
where $\sigma \in \mathbb{S}_{3}$ is the unique cyclic permutation sending $1$ to $2$. 
This was already proved in \cite{IK17}, and more generally in \cite{FH19}, but we provide the proof for completeness. 
Indeed, by Fact \ref{lemma:lema-tecnico-Jacobi-statement}, the first term of the left member of \eqref{eq:igualdad-m3-triple-bracket} is 
 \begin{equation}
 \label{eq:first-term-Jacobi}
  \gan\Big(m_3\big(m_3(a,tf,b),tg,c\big),th\Big)= (h\otimes g\otimes f)\big(\lr{c,\lr{b,a}}_L\big).
 \end{equation}
On the other hand, 
\begin{equation}
\label{eq:second-term-Jacobi}
\begin{split}
&\gan\Big(m_3\big(a,m_3(tf,b,tg),c\big),th\Big) =-\gan\Big(m_3(c,th,a),m_3(tf,b,tg)\Big) 
\\
&=-\gan\big(m_3(tf,b,tg),m_3(c,th,a)\big) =\gan\Big(m_3\big(m_3(c,th,a), tf,b\big),tg\Big) 
\\
&=(g\otimes f\otimes h)\big(\lr{b,\lr{a,c}}_L\big) =(h\otimes g\otimes f)\big(\sigma\lr{b,\lr{a,c}}_L\big), 
\end{split}
\end{equation}
where we have used the cyclicity of $\gan$ in the first and third equalities, the supersymmetry of $\gan$ in the second identity, 
Fact \ref{lemma:lema-tecnico-Jacobi-statement} in the fourth equality and \eqref{eq:unipermvecfun} in the last identity. 
Similarly, using the cyclicity of $\gan$, Fact \ref{lemma:lema-tecnico-Jacobi-statement}, and \eqref{eq:unipermvecfun}, we get
\begin{equation}
\label{eq:third-term-Jacobi}
\begin{split}
-&\gan\Big(m_3\big(a,tf,m_3(b,tg,c)\big),th\Big)=\gan\Big(m_3\big(m_3(b,tg,c),th,a\big),tf\Big) 
\\
&=(f\otimes h\otimes g)\big(\lr{a,\lr{c,b}}_L\big)=(h\otimes g\otimes f)\big(\sigma^{2}\lr{a,\lr{c,b}}_L\big). 
\end{split}
\end{equation}
Hence, \eqref{eq:first-term-Jacobi}, \eqref{eq:second-term-Jacobi} and \eqref{eq:third-term-Jacobi} tell us that the left hand side of \eqref{eq:SI-5-caso-1-gan} coincides with the left hand of \eqref{eq:quasi-Poisson-caso-1-gan}, as was to be shown. 
As a consequence, we proved that $\operatorname{SI}(5)_{\gan}(\omega)$ vanishes for $\omega$ of the form \eqref{eq:omega} 
satisfying \ref{item:caso4}. 

We shall now prove that $\operatorname{SI}(5)_{\gan}(\omega) = 0$, for $\omega$ satisfying \ref{item:caso1}-\ref{item:caso3}.  
We first note that in these $3$ cases the identity 
\[     \big(m_3\circ (m_3\otimes \mathrm{id}^{\otimes 2})\big)(\omega')-\big(m_3\circ (\mathrm{id}\otimes m_3\otimes \mathrm{id})\big)(\omega')+\big(m_3\circ(\mathrm{id}^{\otimes 2}\otimes m_3)\big)(\omega')=0     \] 
holds trivially. 
Indeed, this follows from the fact that $m_{3}$ is acceptable, \textit{i.e.} $m_3=0$ unless $m_3|_{A\otimes A^{\#}[-1]\otimes A}$ (with image in $A$) or $m_3|_{A^{\#}[-1]\otimes A\otimes A^{\#}[-1]}$ (with image in $A^{\#}[-1]$). 
Hence, for any of the $3$ remaining cases, (\hyperlink{eq:ainftyalgebralink}{$\operatorname{SI}(5)$}) evaluated at $\omega'$ reduces to 
\begin{small}
\begin{equation}
 \label{eq:SI-5-simplificada}
 \begin{split}
 0&= \big(m_2\circ (m_4\otimes \mathrm{id})\big)(\omega')
 - \big(m_2\circ (\mathrm{id}\otimes m_4)\big)(\omega')
 + \big(m_4\circ (m_2\otimes \mathrm{id}^{\otimes 3})\big) (\omega')
 \\
 &- \big(m_4\circ (\mathrm{id} \otimes m_2\otimes \mathrm{id}^{\otimes 2})\big)(\omega')
 + \big(m_4 \circ (\mathrm{id}^{\otimes 2}\otimes m_2\otimes \mathrm{id})\big)(\omega')
 - \big(m_4\circ (\mathrm{id}^{\otimes 3}\otimes m_2)\big)(\omega'),
 \end{split}
\end{equation}
\end{small}
\hskip -0.8mm which we are going to verify. 
It is clear that when we evaluate \eqref{eq:SI-5-simplificada} at $\omega'$ satisfying \ref{item:caso1}, every term vanishes by the definition \eqref{eq:mun} of $m_4$. 
Thus (\hyperlink{eq:ainftyalgebragammalink}{$\operatorname{SI}(5)_{\gan}$}) is trivially satisfied in that case.

Let us prove (\hyperlink{eq:ainftyalgebragammalink}{$\operatorname{SI}(5)_{\gan}$}) for \ref{item:caso2}.
Let $\omega = a\otimes b\otimes tf\otimes c \otimes tg \otimes th \in B_{(0,0,1,0,1,1)}$, for  $a,b,c\in A$ and $f,g,h\in A^{\#}$. 
We shall prove that \eqref{eq:SI-5-simplificada} holds when we evaluate it at $\omega'$. 
Applying \eqref{eq:m4-caso-particular}, \hyperlink{eq:ainftyalgebragammalinkop}{$\operatorname{SI}(5)_{\gan}$} evaluated at $\omega$ gives 
\begin{align*}
&-\gan\big(a.m_4(b,tf,c,tg),th\big)+\gan\big(m_4(a b,tf,c,tg),th\big)-\gan\big(m_4(a,b \cdot tf,c,tg),th\big)
\\
&=-\frac{\para}{12}\big(f(c)g(1_A)h(ab)+f(b)g(c)h(a)-f(1_A)g(c)h(ab)-f(cb)g(1_A)h(a)\big)
\\
&\quad +\frac{\para}{12}\big(f(c)g(1_A)h(ab)+f(ab)g(c)h(1_A)-f(1_A)g(c)h(ab)-f(cab)g(1_A)h(1_A)\big)
\\
&\quad -\frac{\para}{12}\big(f(cb)g(1_A)h(a)+f(ab)g(c)h(1_A)-f(b)g(c)h(a)-f(cab)g(1_A)h(1_A)\big)
\\
&=0.
\end{align*}

Finally, let us prove (\hyperlink{eq:ainftyalgebragammalink}{$\operatorname{SI}(5)_{\gan}$}) for \ref{item:caso3}.
Let $\omega = a\otimes b\otimes tf\otimes tg \otimes c \otimes th$ be an element of $B_{(0,0,1,1,0,1)}$, for $a,b,c\in A$ and $f,g,h\in A^{\#}$. 
We shall prove that \eqref{eq:SI-5-simplificada} holds when we evaluate it at $\omega'$. 
Then, \hyperlink{eq:ainftyalgebragammalinkop}{$\operatorname{SI}(5)_{\gan}$} evaluated at $\omega$ gives 
\begin{align*}
-&\gan\big(a.m_4(b,tf,tg,c),th\big)+\gan\big(m_4(ab,tf,tg,c),th\big)-\gan(m_4(a,b \cdot tf,tg,c),th\big)
\\
&=-\frac{\para}{12}\big(f(1_A)g(c)h(ab)+f(b)g(1_A)h(ac)-f(b)g(c)h(a)-f(1_A)g(1_A)h(abc)\big)
\\
&+\frac{\para}{12}\big(f(1_A)g(c)h(ab)+f(ab)g(1_A)h(c)-f(ab)g(c)h(1_A)-f(1_A)g(1_A)h(abc)\big)
\\
&-\frac{\para}{12}\big(f(b)g(c)h(a)+f(ab)g(1_A)h(c)-f(ab)g(c)h(1_A)-f(b)g(1_A)h(ac)\big)
\\
&=0.
\end{align*}
The lemma is thus proved. 
\end{proof}

\subsubsection{\texorpdfstring{The Stasheff identities $\operatorname{SI}(N)$ with even parameter $N>5$}{The Stasheff identities SI(N) with even parameter N > 5}} 
\label{subsubsec:stasheffeven}

We will now proceed to prove the Stasheff identities (\hyperlink{eq:ainftyalgebragammalink}{$\operatorname{SI}(N)_{\gan}$}) for $N \geq 6$. 
The following result shows that this is the case if $N \geq 6$ is even. 
\begin{lemma}
\label{lemma:sieven}
Assume the same hypotheses as in Theorem \ref{theorem:Main}. 
We use the notation introduced in Subsubsection \ref{subsubsec:generalsetting}. 
Let $n \in \NN_{\geq 4}$ be even and let $\omega \in B_{\bar{i}}$ be an element of the form \eqref{eq:omega}, with $\bar{i} = (i_{1}, \dots, i_{n+3}) \in \{ 0 , 1 \}^{n+3}$ and $| \bar{i} | = n$. 
Then, $\operatorname{SI}(n+2)_{\gan}(\omega)$ vanishes. 
\end{lemma}
\begin{proof}
As noted in item \ref{item:sieven}, $\operatorname{SI}(n+2)_{\gan}(\omega)$ only involves $m_{3}$ and $m_{n}$. 
The following cases exhaust all the possibilities. 
\begin{enumerate}[label=(\alph*)]
\item\label{item:si2n4-a} Suppose that $\ell > 1$, so $\ell', \ell'' > 1$. 
The fact that $m_{3}$ is acceptable implies that the terms $\gan \circ ((m_{n} \circ (\mathrm{id}^{\otimes r} \otimes m_{3} \otimes \mathrm{id}^{\otimes (n-r-1)})) \otimes \mathrm{id})(\omega)$ in $\operatorname{SI}(n+2)_{\gan}(\omega)$ vanish, with the possible exception of $r = \ell$ and $r = \ell + \ell' + 1$. 
However, the expression of $m_{n}$ tells us that the corresponding arguments $m_{3}(tf_{\ell}, b, tg_{1})$ and $m_{3}(tg_{\ell'}, c, th_{1})$ have to be then evaluated at $1_{A}$, so they vanish by Fact \ref{fact:cyc2}. 
The definition of $m_{n}$ and Fact \ref{fact:cyc2} imply that $\gan \circ ((m_{3} \circ (\mathrm{id}^{\otimes 2} \otimes m_{n})) \otimes \mathrm{id})(\omega)$ vanishes, since in that case $m_{n}(tf_{2},\dots, tf_{\ell}, b, tg_{1}, \dots, tg_{\ell'}, c, th_{1}, \dots, th_{\ell''-1})$ is a scalar multiple of $1_{A}$. 
Moreover, using that $(\mathrm{id} \otimes m_{n} \otimes \mathrm{id})(\omega') \in A$ and the fact that $m_{3}$ is acceptable, 
$\gan \circ ((m_{3} \circ (\mathrm{id} \otimes m_{n} \otimes \mathrm{id})) \otimes \mathrm{id})(\omega) = 0$. 
The definition of $m_{n}$ also shows that $\gan \circ ((m_{3} \circ (m_{n} \otimes \mathrm{id}^{\otimes 2})) \otimes \mathrm{id})(\omega)$ vanishes if $\ell''>2$, since in that case $(m_{n} \otimes \mathrm{id}^{\otimes 2})(\omega') = 0$. 
If $\ell'' = 2$, then $\ell = \ell' = 2$ as well and $n = 6$, so $(m_{n} \otimes \mathrm{id}^{\otimes 2})(\omega') \in A$, which implies that 
$\gan \circ ((m_{3} \circ (m_{n} \otimes \mathrm{id}^{\otimes 2})) \otimes \mathrm{id})(\omega)$ vanishes, since $m_{3}$ is acceptable. 

\item\label{item:si2n4-b} Suppose that $\ell = \ell' = 0$. 
The fact that $m_{3}$ is acceptable implies that all the terms in $\operatorname{SI}(n+2)_{\gan}(\omega)$ of the form $\gan \circ ((m_{n} \circ (\mathrm{id}^{\otimes r} \otimes m_{3} \otimes \mathrm{id}^{\otimes (n-r-1)})) \otimes \mathrm{id})(\omega)$ vanish. 
Moreover, we have that $\gan \circ ((m_{3} \circ (m_{n} \otimes \mathrm{id}^{\otimes 2})) \otimes \mathrm{id})(\omega) = 0$ and $\gan \circ ((m_{3} \circ (\mathrm{id} \otimes m_{n} \otimes \mathrm{id})) \otimes \mathrm{id})(\omega) = 0$, 
since $m_{n}$ is acceptable and its first two arguments are elements of $A$. 
Finally, $m_{3}$ being acceptable also implies that $\gan \circ ((m_{3} \circ (\mathrm{id}^{\otimes 2} \otimes m_{n})) \otimes \mathrm{id})(\omega) = 0$. 

\item\label{item:si2n4-c} Suppose that $\ell = 0$ and $\ell' = 1$, so $\ell'' = n-1 \geq 3$. 
The fact that $m_{3}$ is acceptable implies that the terms of the form $\gan \circ ((m_{n} \circ (\mathrm{id}^{\otimes r} \otimes m_{3} \otimes \mathrm{id}^{\otimes (n-r-1)})) \otimes \mathrm{id})(\omega)$ in $\operatorname{SI}(n+2)_{\gan}(\omega)$ vanish, with the possible exception of $r = \ell+1$ as well as $r = \ell' + 1$. 
However, since in these last two cases the first two arguments of $m_{n}$ in $\gan \circ ((m_{n} \circ (\mathrm{id}^{\otimes r} \otimes m_{3} \otimes \mathrm{id}^{\otimes (n-r-1)})) \otimes \mathrm{id})(\omega)$ are elements of $A$, the latter vanishes, for $m_{n}$ is acceptable. 
The same reason tells us that the term $\gan \circ ((m_{3} \circ (m_{n} \otimes \mathrm{id}^{\otimes 2})) \otimes \mathrm{id})(\omega)$ vanishes, and $\gan \circ ((m_{3} \circ (\mathrm{id}^{\otimes 2} \otimes m_{n})) \otimes \mathrm{id})(\omega)$ also vanishes because $m_{3}$ is acceptable. 
Finally, $m_{3}$ being acceptable also implies that $\gan \circ ((m_{3} \circ (\mathrm{id} \otimes m_{n} \otimes \mathrm{id})) \otimes \mathrm{id})(\omega) = 0$, since the definition of $m_{n}$ tells us that the first two tensor-factors of $(\mathrm{id} \otimes m_{n} \otimes \mathrm{id})(\omega')$ are elements of $A$. 

\item \label{item:si2n4-d} Suppose that $\ell = 0$ and $\ell' \geq 2$. 
The fact that $m_{3}$ is acceptable implies that the terms in $\operatorname{SI}(n+2)_{\gan}(\omega)$ of the form $\gan \circ ((m_{n} \circ (\mathrm{id}^{\otimes r} \otimes m_{3} \otimes \mathrm{id}^{\otimes (n-r-1)})) \otimes \mathrm{id})(\omega)$ vanish with the possible exception of $r = \ell'+1$ if $\ell''> 1$. 
However, the expression of $m_{n}$ tells us that $\gan \circ ((m_{n} \circ (\mathrm{id}^{\otimes (\ell'+1)} \otimes m_{3} \otimes \mathrm{id}^{\otimes (n-\ell'-2)})) \otimes \mathrm{id})(\omega)$ vanishes, since $\ell'>1$, so the first two arguments of $m_{n}$ are elements of $A$, and $m_{n}$ is acceptable. 
Since $m_{3}$ is acceptable, $\gan \circ ((m_{3} \circ (\mathrm{id}^{\otimes 2} \otimes m_{n})) \otimes \mathrm{id})(\omega)$ vanishes, whereas $\gan \circ ((m_{3} \circ (m_{n} \otimes \mathrm{id}^{\otimes 2})) \otimes \mathrm{id})(\omega) = 0$ follows from the definition of $m_{n}$ because $m_{n}$ is acceptable. 
Finally, $m_{3}$ being acceptable also implies that $\gan \circ ((m_{3} \circ (\mathrm{id} \otimes m_{n} \otimes \mathrm{id})) \otimes \mathrm{id})(\omega) = 0$ if $\ell''>1$, whereas the fact that $m_{n}$ is acceptable tells us that $(\mathrm{id} \otimes m_{n} \otimes \mathrm{id})(\omega')=0$ if $\ell'' = 1$, so $\gan \circ ((m_{3} \circ (\mathrm{id} \otimes m_{n} \otimes \mathrm{id})) \otimes \mathrm{id})(\omega)$ also vanishes. 

\item \label{item:si2n4-e} Suppose that $\ell = 1$ and $\ell', \ell'' \geq 2$. 
The fact that $m_{3}$ is acceptable implies that the terms $\gan \circ ((m_{n} \circ (\mathrm{id}^{\otimes r} \otimes m_{3} \otimes \mathrm{id}^{\otimes (n-r-1)})) \otimes \mathrm{id})(\omega)$ in $\operatorname{SI}(n+2)_{\gan}(\omega)$ vanish, with the possible exception of $r = 0, 1, \ell'+2$. 
However, using the expression of $m_{n}$, we see that its argument $m_{3}(tg_{\ell'},c, th_{1})$ appearing in the case $r = \ell'+2$ has then to be evaluated at $1_{A}$, which in turn implies that $\gan \circ ((m_{n} \circ (\mathrm{id}^{\otimes (\ell'+2)} \otimes m_{3} \otimes \mathrm{id}^{\otimes (n-\ell'-3)})) \otimes \mathrm{id})(\omega)$ vanishes, by Fact \ref{fact:cyc2}. 
Note that $m_{3}$ being acceptable implies that $\gan \circ ((m_{3} \circ (m_{n} \otimes \mathrm{id}^{\otimes 2})) \otimes \mathrm{id})(\omega)$ and $\gan \circ ((m_{3} \circ (\mathrm{id} \otimes m_{n} \otimes \mathrm{id})) \otimes \mathrm{id})(\omega)$ vanish. 
Then, $\operatorname{SI}(n+2)_{\gan}(\omega)$ reduces to 
\begin{equation*}
\begin{split}
-&\gan\Big(m_{n}\big(m_{3}(a,tf_{1},b),tg_{1},\dots,tg_{\ell'},c, th_{1}, \dots, th_{\ell''-1}\big), th_{\ell''}\Big) 
\\
- &\gan\Big(m_{n}\big(a,m_{3}(tf_{1},b,tg_{1}),tg_{2},\dots,tg_{\ell'},c, th_{1}, \dots, th_{\ell''-1}\big), th_{\ell''}\Big) 
\\
+ &\gan\Big(m_{3}\big(a,tf_{1},m_{n}(b,tg_{1},\dots,tg_{\ell'},c, th_{1}, \dots, th_{\ell''-1})\big), th_{\ell''}\Big)
\\
&= \alpha \Big[ \gan\big( m_{3}(a,tf_{1},b),tg_{1}\big) h_{\ell''}(1_{A}) \big( g_{\ell'}(c) h_{1}(1_{A}) - g_{\ell'}(1_{A}) h_{1}(c)\big) 
\\
&\phantom{= \alpha} - \gan\big(m_{3}(a,tf_{1},b),th_{\ell''} \big) g_{1}(1_{A}) \big( g_{\ell'}(c) h_{1}(1_{A}) - g_{\ell'}(1_{A}) h_{1}(c)\big) 
\\
&\phantom{= \alpha} + \gan\big(m_{3}(tf_{1},b,tg_{1}),a\big) h_{\ell''}(1_{A}) \big( g_{\ell'}(c) h_{1}(1_{A}) - g_{\ell'}(1_{A}) h_{1}(c)\big) 
\\
&\phantom{= \alpha} +  \gan\big(m_{3}(a,tf_{1},b),th_{\ell''}\big) g_{1}(1_{A}) \big( g_{\ell'}(c) h_{1}(1_{A}) - g_{\ell'}(1_{A}) h_{1}(c)\big) \Big],
\end{split}
\end{equation*}          
where we have used that the map $m_{3}$ is normalized (with respect to $1_{A}$), and 
$\alpha = C_{\ell',\ell''} \prod_{i = 2}^{\ell'-1} g_{i}(1_{A}) \prod_{i = 2}^{\ell''-1} h_{i}(1_{A})$. 
The second and fourth terms trivially cancel each other, and the cyclicity of $m_{3}$ tells us that the first and third terms also cancel, 
which implies that $\operatorname{SI}(n+2)_{\gan}(\omega)$ vanishes.

\item \label{item:si2n4-f} Suppose that $\ell = \ell' = 1$, so $\ell'' = n-2 \geq 2$. 
The fact that $m_{3}$ is acceptable implies that the terms of the form $\gan \circ ((m_{n} \circ (\mathrm{id}^{\otimes r} \otimes m_{3} \otimes \mathrm{id}^{\otimes (n-r-1)})) \otimes \mathrm{id})(\omega)$ in $\operatorname{SI}(n+2)_{\gan}(\omega)$ vanish, with the possible exception of $r = 0, 1, 2, 3$. 
Note that $m_{3}$ being acceptable implies that $\gan \circ ((m_{3} \circ (m_{n} \otimes \mathrm{id}^{\otimes 2})) \otimes \mathrm{id})(\omega)$ and $\gan \circ ((m_{3} \circ (\mathrm{id} \otimes m_{n} \otimes \mathrm{id})) \otimes \mathrm{id})(\omega)$ vanish. 
Then, $\operatorname{SI}(n+2)_{\gan}(\omega)$ reduces to 
\begin{equation}
\label{eq:terms}
\begin{split}
-&\gan\Big(m_{n}\big(m_{3}(a,tf_{1},b),tg_{1},c, h_{1}, \dots, th_{\ell''-1}\big), th_{\ell''}\Big) 
\\
- &\gan\Big(m_{n}\big(a,m_{3}(tf_{1},b,tg_{1}),c, th_{1}, \dots, th_{\ell''-1}\big), th_{\ell''}\Big) 
\\
+ &\gan\Big(m_{n}\big(a,tf_{1},m_{3}(b,tg_{1},c), th_{1},\dots, th_{\ell''-1}\big), th_{\ell''}\Big) 
\\
+ &\gan\Big(m_{n}\big(a,tf_{1},b,m_{3}(tg_{1},c, th_{1}), th_{2},\dots, th_{\ell''-1}\big), th_{\ell''}\Big) 
\\
+ &\gan\Big(m_{3}\big(a,tf_{1},m_{n}(b,tg_{1},c, th_{1}, \dots, th_{\ell''-1})\big), th_{\ell''}\Big).
\end{split}
\end{equation}  
Applying the definition \eqref{eq:mun} of $m_{n}$ to each term of the previous expression and using the fact that $m_{3}$ and $m_{n}$ are normalized, we see that the first and third terms in the previous equation are each one equal to the sum of four nonzero summands, whereas the second term is the sum of three nonzero summands, and the fourth and fifth terms are each one the sum of only two nonzero summands. 
Moreover, two of the four summands coming from the first term of \eqref{eq:terms} directly cancel the two summands of the fifth term of \eqref{eq:terms}. 
The cyclicity property \eqref{eq:ip1} tells us that one remaining summand of the first term of \eqref{eq:terms} cancels one of the three summands of the second term of \eqref{eq:terms}, that a second summand of the second term of \eqref{eq:terms} cancels one of the four summands of the third term of \eqref{eq:terms}, and that two other summands of the third term of \eqref{eq:terms} cancel the two summands of the fourth term of \eqref{eq:terms}. 
The three remaining summands (each one coming from the first three terms of \eqref{eq:terms}) are 
\begin{equation}
\label{eq:terms2}
     \alpha \Big[ \gan\big(c.m_{3}(a,tf_{1},b) , tg_{1}\big) + \gan\big(m_{3}(tf_{1},b,tg_{1}),c a \big) - \gan\big(m_{3}(b,tg_{1},c).a, tf_{1}\big) \Big],    
\end{equation}  
where $\alpha = C_{1,\ell''} \prod_{q = 1}^{\ell''} h_{q}(1_{A})$. 
Using the cyclicity property of $m_{3}$, \eqref{eq:terms2} becomes 
\begin{align*}
      &\alpha \Big[ - \gan\big(m_{3}(b,tg_{1} \cdot c,a), tf_{1}\big) + \gan\big(m_{3}(b,tg_{1},c a),tf_{1}\big) - \gan\big(m_{3}(b,tg_{1},c).a, t f_{1}\big) \Big] 
      \\
      &= \alpha \operatorname{SI}(4)_{\gan}(b,tg_{1},c,a,tf_{1}) = 0,     
\end{align*}
where we have used Lemma \ref{lemma:si4}.  
\end{enumerate}
\end{proof}

\subsubsection{\texorpdfstring{The Stasheff identities $\operatorname{SI}(N)$ with odd parameter $N>5$}{The Stasheff identities SI(N) with odd parameter N > 5}} 
\label{subsubsec:stasheffodd}

The aim of this final subsubsection is to prove the next result.
\begin{proposition}
\label{proposition:siodd}
Assume the same hypotheses as in Theorem \ref{theorem:Main}. 
We use the notation introduced in Subsubsection \ref{subsubsec:generalsetting}. 
Let $n \geq 6$ be an even integer and let $\omega \in B_{\bar{i}}$ be an element of the form \eqref{eq:omega}, with $\bar{i} = (i_{1}, \dots, i_{n+2}) \in \{ 0 , 1 \}^{n+2}$ and $| \bar{i} | = n-1$. 
Then, $\operatorname{SI}(n+1)_{\gan}(\omega)$ vanishes. 
\end{proposition}

Since the handling of the Stasheff identities \hypertarget{eq:ainftyalgebragammalink}{$\operatorname{SI}(N)_{\gan}$} with $N \geq 6$ odd is quite involved, 
we will further separate into two cases, Lemmas \ref{lemma:sioddl0} and \ref{lemma:sioddl1}. 
Moreover, as explained in Remark \ref{remark:tau1}, it suffices to deal with the case $\para=1$.

\paragraph{\texorpdfstring{The Stasheff identities $\operatorname{SI}(N)$ with odd parameter $N>5$ and $\ell = 0$.}{The Stasheff identities SI(N) with odd parameter N > 5 and l = 0.}} 
\label{paragraph:stasheffodd>5ell=0}

\begin{lemma}
\label{lemma:sioddl0}
Assume the same hypotheses as in Theorem \ref{theorem:Main}. 
We use the notation introduced in Subsubsection \ref{subsubsec:generalsetting}. 
Let $n \geq 6$ be an even integer and let $\omega \in B_{\bar{i}}$ be an element of the form \eqref{eq:omega}, with $\bar{i} = (i_{1}, \dots, i_{n+2}) \in \{ 0 , 1 \}^{n+2}$ and $| \bar{i} | = n-1$. 
Suppose further that $\ell = 0$. 
Then, $\operatorname{SI}(n+1)_{\gan}(\omega)$ vanishes. 
\end{lemma}
\begin{proof}
The fact that $m_{2j}$ is acceptable for every integer $j \geq 2$ tells us that the term 
$\gan \circ ((m_{n+2-2k} \circ (\mathrm{id}^{\otimes r} \otimes m_{2k} \otimes \mathrm{id}^{\otimes (n+1-r-2k)})) \otimes \mathrm{id})(\omega)$ of $\operatorname{SI}(n+1)_{\gan}(\omega)$ vanishes for all $r \in \llbracket 0 , n+1-2k \rrbracket \setminus \{ 1 \}$ and all $k \in \llbracket 2, (n-2)/2 \rrbracket$. 
Analogously, the acceptability of $m_{n}$ tells us that $\gan \circ ((m_{n} \circ (\mathrm{id}^{\otimes r} \otimes m_{2} \otimes \mathrm{id}^{\otimes (n-r-1)})) \otimes \mathrm{id})(\omega)$ vanishes for all $r \in \llbracket 2 , n-1 \rrbracket$, as well as $\gan \circ ((m_{2} \circ  (m_{n} \otimes \mathrm{id})) \otimes \mathrm{id})(\omega) = 0$. 
With respect to the term $\gan \circ ((m_{n+2-2k} \circ (\mathrm{id} \otimes m_{2k} \otimes \mathrm{id}^{\otimes (n-2k)})) \otimes \mathrm{id})(\omega)$ for $k \in \llbracket 2, (n-2)/2 \rrbracket$, we see that $m_{2k}$ has either two arguments in $A$ or only one argument. 
In the first case, the term $\gan \circ ((m_{n+2-2k} \circ (\mathrm{id} \otimes m_{2k} \otimes \mathrm{id}^{\otimes (n-2k)})) \otimes \mathrm{id})(\omega)$ vanishes because $m_{n+2-2k}$ is acceptable, whereas in the latter the evaluation of $m_{2k}$ already vanishes, because $m_{2k}$ is acceptable. 
If $\ell' = 0$, the acceptability of $m_{n}$ implies that the remaining terms of $\operatorname{SI}(n+1)_{\gan}(\omega)$ vanish as well. 
We will thus assume that $\ell' \geq 1$.
In this case, $\operatorname{SI}(n+1)_{\gan}(\omega)$ reduces to 
\begin{equation}
\label{eq:siganomegan+1l0}
\begin{split}
     &-\gan\big(a . m_{n}(b,tg_{1},\dots,tg_{\ell'},c,th_{1}, \dots, th_{\ell''-1}),th_{\ell''}\big) 
     \\
     &+ \gan\big(m_{n}(a b,tg_{1},\dots,tg_{\ell'},c,th_{1}, \dots, th_{\ell''-1}),th_{\ell''}\big) 
     \\
     &- \gan\big(m_{n}(a, b \cdot tg_{1},\dots,tg_{\ell'},c,th_{1}, \dots, th_{\ell''-1}),th_{\ell''}\big).
\end{split}
\end{equation}

Let 
\[     \alpha = C_{\ell', \ell''} 
\prod_{i=2}^{\ell'-1} g_{i}(1_{A}) \prod_{i=2}^{\ell''-1} h_{i}(1_{A}).     \]
If $\ell' = 1$, then $\ell'' = n - 2 \geq 4$, and \eqref{eq:siganomegan+1l0} gives    
\begin{small}
\begin{equation*}
\begin{split}
     &\alpha \Big[ -\big( g_{1}(c) h_{1}(1_{A}) h_{\ell''}(ab) + g_{1}(b) h_{1}(c) h_{\ell''}(a) - g_{1}(cb) h_{1}(1_{A}) h_{\ell''}(a) 
     \\ 
     &\phantom{ \alpha } - g_{1}(1_{A}) h_{1}(c) h_{\ell''}(ab) \big) + \big( g_{1}(c) h_{1}(1_{A}) h_{\ell''}(ab) + g_{1}(a b) h_{1}(c) h_{\ell''}(1_{A}) 
     \\ 
     &\phantom{ \alpha} - g_{1}(cab) h_{1}(1_{A}) h_{\ell''}(1_{A}) - g_{1}(1_{A}) h_{1}(c) h_{\ell''}(ab) \big) - \big( g_{1}(ab) h_{1}(c) h_{\ell''}(1_{A}) 
     \\
     &\phantom{ \alpha} + g_{1}(cb) h_{1}(1_{A}) h_{\ell''}(a) - g_{1}(cab) h_{1}(1_{A}) h_{\ell''}(1_{A}) - g_{1}(b) h_{1}(c) h_{\ell''}(a) \big) \Big]
     \\
     &= 0.
\end{split}
\end{equation*}
\end{small}
\hskip -0.8mm If $\ell' > 1$ but $\ell'' = 1$, \eqref{eq:siganomegan+1l0} gives    
\begin{small}
\begin{equation*}
\begin{split}
     &\alpha \Big[ - \big( g_{1}(1_{A}) g_{\ell'}(c) h_{1}(a b) + g_{1}(b) g_{\ell'}(1_{A}) h_{1}(a c) - g_{1}(b) g_{\ell'}(c) h_{1}(a) 
     \\ 
     &- g_{1}(1_{A}) g_{\ell'}(1_{A}) h_{1}(a b c) \big) + \big( g_{1}(1_{A}) g_{\ell'}(c) h_{1}(ab) + g_{1}(a b) g_{\ell'}(1_{A}) h_{1}(c) 
     \\ 
     &- g_{1}(ab) g_{\ell'}(c) h_{1}(1_{A}) - g_{1}(1_{A}) g_{\ell'}(1_{A}) h_{1}(a b c) \big) - \big( g_{1}(ab) g_{\ell'}(1_{A}) h_{1}(c) 
     \\
     &+ g_{1}(b) g_{\ell'}(c) h_{1}(a) - g_{1}(ab) g_{\ell'}(c) h_{1}(1_{A}) - g_{1}(b) g_{\ell'}(1_{A}) h_{1}(a c) \big) \Big]
     \\
     &= 0.
\end{split}
\end{equation*}
\end{small}
\hskip -0.8mm Finally, if $\ell', \ell'' > 1$, then \eqref{eq:siganomegan+1l0} gives
\begin{small}
\begin{equation*}
\begin{split}
     &\alpha \Big[ - \big( g_{1}(1_{A}) g_{\ell'}(c) h_{1}(1_{A}) h_{\ell''}(ab) + g_{1}(b) g_{\ell'}(1_{A}) h_{1}(c) h_{\ell''}(a) - g_{1}(b) g_{\ell'}(c) h_{1}(1_{A}) h_{\ell''}(a) 
     \\ 
     &- g_{1}(1_{A}) g_{\ell'}(1_{A}) h_{1}(c) h_{\ell''}(ab) \big) + \big( g_{1}(1_{A}) g_{\ell'}(c) h_{1}(1_{A}) h_{\ell''}(ab) + g_{1}(a b) g_{\ell'}(1_{A}) h_{1}(c) h_{\ell''}(1_{A}) 
     \\ 
     &- g_{1}(ab) g_{\ell'}(c) h_{1}(1_{A}) h_{\ell''}(1_{A}) - g_{1}(1_{A}) g_{\ell'}(1_{A}) h_{1}(c) h_{\ell''}(ab) \big) - \big( g_{1}(ab) g_{\ell'}(1_{A}) h_{1}(c) h_{\ell''}(1_{A}) 
     \\
     &+ g_{1}(b) g_{\ell'}(c) h_{1}(1_{A}) h_{\ell''}(a) - g_{1}(ab) g_{\ell'}(c) h_{1}(1_{A}) h_{\ell''}(1_{A}) - g_{1}(b) g_{\ell'}(1_{A}) h_{1}(c) h_{\ell''}(a) \big) \Big]
     \\
     &= 0.
\end{split}
\end{equation*}
\end{small} 
\hskip -0.8mm The lemma is thus proved.
\end{proof}

\paragraph{\texorpdfstring{The Stasheff identities $\operatorname{SI}(N)$ with odd parameter $N>5$ and $\ell > 0$.}{The Stasheff identities SI(N) with odd parameter N > 5 and l > 0.}} 
\label{paragraph:stasheffodd>5ell>1}

The aim of this last paragraph is to ascertain the next result, which completes our proof of the Stasheff identities for $\partial A$. 
\begin{lemma}
\label{lemma:sioddl1}
Assume the same hypotheses as in Theorem \ref{theorem:Main}. 
We use the notation introduced in Subsubsection \ref{subsubsec:generalsetting}. 
Let $n \geq 6$ be an even integer and let $\omega \in B_{\bar{i}}$ be an element of the form \eqref{eq:omega}, with $\bar{i} = (i_{1}, \dots, i_{n+2}) \in \{ 0 , 1 \}^{n+2}$ and $| \bar{i} | = n-1$. 
Suppose that $\ell \geq 1$. 
Then, $\operatorname{SI}(n+1)_{\gan}(\omega)$ vanishes. 
\end{lemma}

The proof will require some preparations, namely Facts \ref{fact:sif}, \ref{fact:sil}, \ref{fact:sim}, \ref{fact:codd}, \ref{fact:ceven} and \ref{fact:cgen}, that we will now provide. 
Note that $\ell'' > 1$, since $\ell'' = 1$ implies $\ell = \ell'=1$ but $\ell+\ell' + \ell'' = n - 1\geq 5$, and note as well that $\ell + \ell' +2 = n + 1 - \ell'' \leq n- 1$. 
To simplify some expressions appearing in the sequel, set 
\begin{equation}
\label{eq:kappa}
     \kappa(\Omega) = \prod_{i=2}^{\ell-1} f_{i}(1_{A}) \prod_{i=2}^{\ell'-1} g_{i}(1_{A}) \prod_{i=2}^{\ell''-1} h_{i}(1_{A}).
\end{equation}

Define the element $\operatorname{SI}(n+1)^{2,n}_{\gan}(\omega)$ by 
\begin{equation}
\label{eq:sin+12n}
\begin{split}
   \sum_{r = 0}^{n-1} (-1)^{r} &\gan\Big( \big(m_{n} \circ (\mathrm{id}^{\otimes r} \otimes m_{2} \otimes \mathrm{id}^{\otimes (n-r-1)})\big) \otimes \mathrm{id} \Big)(\omega)
   \\
   + &\gan\Big(\big(m_{2} \circ (m_{n} \otimes \mathrm{id})\big) \otimes \mathrm{id}\Big)(\omega) - \gan\Big(\big(m_{2} \circ (\mathrm{id} \otimes m_{n})\big) \otimes \mathrm{id}\Big)(\omega). 
\end{split}
\end{equation}
Notice that the terms in \eqref{eq:sin+12n} with $r \in \llbracket 0, n-1 \rrbracket \setminus \{ 0, \ell, \ell+1, \ell + \ell' +1, \ell + \ell' +2 \}$ vanish, since the product in $\partial A$ of two elements of $A^{\#}[-1]$ vanishes. 
Moreover, by a degree argument the first term of the second line of \eqref{eq:sin+12n} also vanishes. 
Hence, $\operatorname{SI}(n+1)^{2,n}_{\gan}(\omega)$ reduces to 
\begin{equation}
\label{eq:sin+12nbis}
\begin{split}
   &- \gan\big(m_{n}(tf_{1},\dots,tf_{\ell},b, tg_{1}, \dots, tg_{\ell'}, c, th_{1}, \dots, th_{\ell''-1}), th_{\ell''} \cdot a\big)
   \\
   &+ \gan\big(m_{n}(a \cdot tf_{1},tf_{2},\dots,tf_{\ell},b, tg_{1}, \dots, tg_{\ell'}, c, th_{1}, \dots, th_{\ell''-1}), th_{\ell''}\big) 
   \\
&+(-1)^{\ell} \gan\big(m_{n}(a,tf_{1},\dots,tf_{\ell-1},tf_{\ell} \cdot b, tg_{1}, \dots, tg_{\ell'}, c, th_{1}, \dots, th_{\ell''-1}), th_{\ell''}\big) 
\\
&-(-1)^{\ell} \gan\big(m_{n}(a,tf_{1},\dots,tf_{\ell},b \cdot tg_{1}, tg_{2}, \dots, tg_{\ell'}, c, th_{1}, \dots, th_{\ell''-1}), th_{\ell''}\big) 
\\
&-(-1)^{\ell + \ell'} \gan\big(m_{n}(a,tf_{1},\dots,tf_{\ell},b,tg_{1}, \dots, tg_{\ell'-1}, tg_{\ell'} \cdot c, th_{1}, \dots, th_{\ell''-1}), th_{\ell''}\big) 
\\
&+(-1)^{\ell + \ell'} \gan\big(m_{n}(a,tf_{1},\dots,tf_{\ell}, b, tg_{1}, \dots, tg_{\ell'}, c \cdot th_{1}, th_{2}, \dots, th_{\ell''-1}), th_{\ell''}\big). 
\end{split}
\end{equation}
Set 
\begin{equation}
\label{eq:betagen}
      \beta^{\ell,\ell',\ell''} = \big(C_{\ell',\ell''+\ell}-(-1)^{\ell} C_{\ell'',\ell+\ell'}+(-1)^{\ell+\ell'} C_{\ell,\ell'+\ell''}\big) \kappa(\Omega),     
\end{equation}
where $\kappa(\Omega)$ was defined in \eqref{eq:kappa}. 
Using \eqref{eq:mun} in \eqref{eq:sin+12nbis}, a long but easy computation tells us that $\operatorname{SI}(n+1)^{2,n}_{\gan}(\omega)$ 
is given by $\beta^{\ell,\ell',\ell''} E^{\ell,\ell',\ell''}$, where $E^{\ell,\ell',\ell''}$ is given by 
\begin{equation}
\label{eq:sin+12nbisl=l'=1}
\begin{split}
   &-\big(f_{1}(b) g_{1}(c) h_{1}(1_{A}) + f_{1}(1_{A}) g_{1}(b) h_{1}(c) 
   \\
   &\phantom{-\Big(}- f_{1}(b) g_{1}(1_{A}) h_{1}(c) - f_{1}(1_{A}) g_{1}(c b) h_{1}(1_{A}) \big) h_{\ell''}(a) 
   \\
   &+\big(f_{1}(ba) g_{1}(c) h_{1}(1_{A})  + f_{1}(a) g_{1}(b) h_{1}(c) 
   \\
   &\phantom{-\Big(}- f_{1}(ba) g_{1}(1_{A}) h_{1}(c) - f_{1}(a) g_{1}(c b) h_{1}(1_{A}) \big)h_{\ell''}(1_{A}) 
\end{split}
\end{equation}
if $\ell = \ell' = 1$, by
\begin{equation}
\label{eq:sin+12nbisl=1,l'>1}
\begin{split}
   -&\Big( \big( f_{1}(b) g_{1}(1_{A}) - f_{1}(1_{A}) g_{1}(b) \big) h_{\ell''}(a) - \big( f_{1}(b a) g_{1}(1_{A}) - f_{1}(a) g_{1}(b) \big) h_{\ell''}(1_{A}) \Big)  
   \\
   &\phantom{\Big(}\big( g_{\ell'}(c) h_{1}(1_{A}) - g_{\ell'}(1_{A}) h_{1}(c)\big)
\end{split}
\end{equation}
if $\ell = 1$, $\ell' > 1$, and by 
\begin{equation}
\label{eq:sin+12nbisl>1}
\begin{split}
&\big(f_{1}(a) h_{\ell''}(1_{A}) - f_{1}(1_{A}) h_{\ell''}(a)\big) 
\big( f_{\ell}(b) g_{1}(1_{A}) - f_{\ell}(1_{A}) g_{1}(b) \big)
\\
&\big( g_{\ell'}(c) h_{1}(1_{A}) - g_{\ell'}(1_{A}) h_{1}(c) \big)    
\end{split}
\end{equation}
if $\ell > 1$. 

Define the element $\operatorname{SI}(n+1)^{\neq 2,n}_{\gan}(\omega)$ by 
\begin{equation}
\label{eq:sin+1not2n}
\begin{split}
   \sum_{k=2}^{\frac{n-2}{2}} \sum_{r = 0}^{n-2k+1} (-1)^{r} &\gan\Big(\big(m_{n-2k+2} \circ (\mathrm{id}^{\otimes r} \otimes m_{2k} \otimes \mathrm{id}^{\otimes (n-r-2k+1)})\big) \otimes \mathrm{id} \Big)(\omega). 
\end{split}
\end{equation}
It is clear that $\operatorname{SI}(n+1)_{\gan}(\omega) = \operatorname{SI}(n+1)^{2,n}_{\gan}(\omega) + \operatorname{SI}(n+1)^{\neq 2,n}_{\gan}(\omega)$. 
Moreover, by the acceptability of $m_{2k}$ for $k \in \llbracket 2, (n-2)/2 \rrbracket$ and the fact that is $1_{A}$-normalized, we see that the terms of $\operatorname{SI}(n+1)^{\neq 2,n}_{\gan}(\omega)$ corresponding to indices not appearing in the following list vanish:
\begin{enumerate}[label=(T\arabic*)]
\item\label{item:T1} $r = 0$ or $r =1$, and $2k \in \llbracket \ell + 2, \ell + \ell' + 1 \rrbracket$; 
\item\label{item:T2} $r \in  \llbracket 2, \ell \rrbracket$ and $2k \in \llbracket \ell + \ell' + 2 - r, \ell + \ell' + 3 - r \rrbracket$ (note that there is only one value of $k$ for every value of $r$, and this item is nonvoid if $\ell>1$); 
\item\label{item:T3} $r = \ell + 1$ and $2k \in \llbracket \ell' + 2, \ell' + \ell'' + 1 \rrbracket$; 
\item\label{item:T4} $r = \ell + 2$ and $2k \in \llbracket \ell' + 2, \ell' + \ell'' \rrbracket$. 
\end{enumerate}
We let the reader check that the lower and upper limits for $2k$ in the previous list always imply that $2k \in \llbracket 4, n-2 \rrbracket$. 

To facilitate the handling of our expressions we will consider the functions $\mathcalboondox{p}, \mathcalboondox{i} : \ZZ \rightarrow \{ 0, 1\}$ defined by 
$\mathcalboondox{p}(j) + \mathcalboondox{i}(j) = 1$ for all $j \in \ZZ$, $\mathcalboondox{p}(j) = 0$ if $j$ is even and $\mathcalboondox{p}(j) = 1$ if $j$ is odd. 
We shall also write $\mathcalboondox{p}_{j}$ and $\mathcalboondox{i}_{j}$ instead of $\mathcalboondox{p}(j)$ and $\mathcalboondox{i}(j)$, respectively. 
We also recall that, given $x \in \RR$, $\lfloor x \rfloor = \operatorname{sup} \{ n \in \ZZ : n \leq x\}$ is the \textbf{\textcolor{myblue}{floor}} function, 
and $\lceil x \rceil = - \lfloor - x \rfloor$ is the \textbf{\textcolor{myblue}{ceiling}} function.

To simplify our notation, set $\gamma_{j,0}^{\ell,\ell', \ell''}$ as 
\begin{equation}
\label{eq:gammaj0}
\begin{split}
\gan\Big(m_{n-\ell-2j+2-\mathcalboondox{p}_{\ell}}\big(m_{\ell+2j+\mathcalboondox{p}_{\ell}}(a,&tf_{1},\dots,tf_{\ell},b, tg_{1}, \dots, tg_{2j-2+\mathcalboondox{p}_{\ell}}),
\\
&tg_{2j-1+\mathcalboondox{p}_{\ell}}, \dots, tg_{\ell'}, c, th_{1}, \dots, th_{\ell''-1}\big), th_{\ell''} \Big)
\end{split}
\end{equation}
for $j \in \llbracket 1, \lfloor \ell'/2 \rfloor + \mathcalboondox{i}_{\ell} \mathcalboondox{p}_{\ell'} \rrbracket$, $\gamma_{j,1}^{\ell,\ell',\ell''}$ as 
\begin{equation}
\label{eq:gammaj1}
\begin{split}
\gan\Big(m_{n-\ell-2j+2-\mathcalboondox{p}_{\ell}}\big(a,&m_{\ell+2j+\mathcalboondox{p}_{\ell}}(tf_{1},\dots,tf_{\ell},b, tg_{1}, \dots, tg_{2j-1+\mathcalboondox{p}_{\ell}}),
\\
&tg_{2j+\mathcalboondox{p}_{\ell}}, \dots, tg_{\ell'}, c, th_{1}, \dots, th_{\ell''-1}\big), th_{\ell''}\Big)
\end{split} 
\end{equation}
for $j \in \llbracket 1, \lfloor \ell'/2 \rfloor + \mathcalboondox{i}_{\ell} \mathcalboondox{p}_{\ell'} \rrbracket$, $\gamma_{j,\ell+1}^{\ell,\ell',\ell''}$ as 
\begin{equation}
\label{eq:gammajl+1}
\begin{split}
\gan\Big(m_{n-\ell'-2j+2-\mathcalboondox{p}_{\ell'}}\big(a,tf_{1},\dots,tf_{\ell},&m_{\ell'+2j+\mathcalboondox{p}_{\ell'}}(b, tg_{1}, \dots, tg_{\ell'},c, th_{1}, \dots, th_{2j-2+\mathcalboondox{p}_{\ell'}}),
\\
&th_{2j-1+\mathcalboondox{p}_{\ell'}}, \dots, th_{\ell''-1}\big), th_{\ell''}\Big)
\end{split} 
\end{equation}
for $j \in \llbracket 1, \lfloor \ell''/2 \rfloor + \mathcalboondox{i}_{\ell'} \mathcalboondox{p}_{\ell''}\rrbracket$, and $\gamma_{j,\ell+2}^{\ell,\ell',\ell''}$ as 
\begin{equation}
\label{eq:gammajl+2}
\begin{split}
\gan\Big(m_{n-\ell'-2j+2-\mathcalboondox{p}_{\ell'}}\big(a,tf_{1},\dots,tf_{\ell},b, &m_{\ell'+2j+\mathcalboondox{p}_{\ell'}}(tg_{1}, \dots, tg_{\ell'},c, th_{1}, \dots, th_{2j-1+\mathcalboondox{p}_{\ell'}}),
\\
&th_{2j+\mathcalboondox{p}_{\ell'}}, \dots, th_{\ell''-1}\big), th_{\ell''}\Big)
\end{split} 
\end{equation}
for $j \in \llbracket 1, \lfloor \ell''/2 \rfloor + \mathcalboondox{i}_{\ell'} \mathcalboondox{p}_{\ell''} - \mathcalboondox{p}_{\ell'-\ell''} \rrbracket$. 

For $\ell > 1$ and $r \in \llbracket 2, \ell \rrbracket$ we also define $\gamma_{r}^{\ell,\ell',\ell''}$ as 
\begin{equation}
\label{eq:gammareven}
\begin{split}
\gan\Big(m_{r + \ell''+1}\big(a,tf_{1},\dots,tf_{r-1}, m_{\ell+\ell'+2-r}(&tf_{r}, \dots, tf_{\ell},b, tg_{1}, \dots, tg_{\ell'}),
\\
&c, th_{1}, \dots, th_{\ell''-1}\big), th_{\ell''}\Big)
\end{split} 
\end{equation}
if $\ell+\ell'-r$ is even, and as 
\begin{equation}
\label{eq:gammarodd}
\begin{split}
\gan\Big(m_{r + \ell''}\big(a,tf_{1},\dots,tf_{r-1}, m_{\ell+\ell'+3-r}(&tf_{r}, \dots, tf_{\ell},b, tg_{1}, \dots, tg_{\ell'}, c),
\\
&th_{1}, \dots, th_{\ell''-1}\big), th_{\ell''}\Big)
\end{split} 
\end{equation}
if $\ell+\ell'-r$ is odd. 

Define $\operatorname{SI}(n+1)_{\gan}^{f}(\omega)$ to be 
\begin{equation}
\label{eq:sin+1f}
   \sum_{r = 0}^{1} \sum_{k=\lceil (\ell+2)/2 \rceil}^{\lfloor (\ell + \ell' + 1)/2 \rfloor} (-1)^{r} \gan\Big( \big(m_{n-2k+2} \circ (\mathrm{id}^{\otimes r} \otimes m_{2k} \otimes \mathrm{id}^{\otimes (n-r-2k+1)})\big) \otimes \mathrm{id} \Big)(\omega), 
\end{equation}
$\operatorname{SI}(n+1)_{\gan}^{m}(\omega)$ to be
\begin{small}
\begin{equation}
\label{eq:sin+1m}
   \sum_{r = 2}^{\ell} (-1)^{r} \gan\Big( \big(m_{\ell''+r+1-\mathcalboondox{p}_{\ell-\ell'-r}} \circ (\mathrm{id}^{\otimes r} \otimes m_{\ell+\ell'-r+2+\mathcalboondox{p}_{\ell-\ell'-r}} \otimes \mathrm{id}^{\otimes (\ell''-\mathcalboondox{p}_{\ell-\ell'-r})})\big) \otimes \mathrm{id} \Big)(\omega), 
\end{equation}
\end{small}
\hskip -0.8mm and $\operatorname{SI}(n+1)_{\gan}^{l}(\omega)$ to be 
\begin{equation}
\label{eq:sin+1l}
   \sum_{r = \ell+1}^{\ell+2} \sum_{k=\lceil (\ell'+2)/2 \rceil}^{\lfloor (\ell' + \ell'' + \delta_{r,\ell+1})/2 \rfloor} (-1)^{r} \gan\Big( \big(m_{n-2k+2} \circ (\mathrm{id}^{\otimes r} \otimes m_{2k} \otimes \mathrm{id}^{\otimes (n-r-2k+1)})\big) \otimes \mathrm{id} \Big)(\omega), 
\end{equation}
\textit{i.e.} $\operatorname{SI}(n+1)_{\gan}^{f}(\omega)$ is the sum of the terms in $\operatorname{SI}(n+1)^{\neq 2,n}_{\gan}(\omega)$ corresponding to the case \ref{item:T1}, $\operatorname{SI}(n+1)_{\gan}^{m}(\omega)$ is the sum of the terms in $\operatorname{SI}(n+1)^{\neq 2,n}_{\gan}(\omega)$ corresponding to the case \ref{item:T2}, and $\operatorname{SI}(n+1)_{\gan}^{l}(\omega)$ is the sum of the terms in $\operatorname{SI}(n+1)^{\neq 2,n}_{\gan}(\omega)$ corresponding to the cases \ref{item:T3} and \ref{item:T4}. 
As a consequence, 
\[     \operatorname{SI}(n+1)^{\neq 2,n}_{\gan}(\omega) = \operatorname{SI}(n+1)_{\gan}^{f}(\omega) + \operatorname{SI}(n+1)_{\gan}^{m}(\omega) + \operatorname{SI}(n+1)_{\gan}^{l}(\omega).     \]
Moreover, using definitions \eqref{eq:gammaj0}-\eqref{eq:gammarodd}, $\operatorname{SI}(n+1)^{f}_{\gan}(\omega)$ is given by 
\begin{equation}
\label{eq:sin+1fbis}
\begin{split}
   &\sum_{j=1}^{\lfloor \ell'/2 \rfloor + \mathcalboondox{i}_{\ell} \mathcalboondox{p}_{\ell'} - \mathcalboondox{p}_{\ell-\ell'}} \big( \gamma_{j,0}^{\ell,\ell',\ell''} - \gamma_{j,1}^{\ell,\ell',\ell''}  \big) + \mathcalboondox{p}_{\ell-\ell'} \big( \gamma_{\lfloor \ell'/2 \rfloor+ \mathcalboondox{i}_{\ell} \mathcalboondox{p}_{\ell'},0}^{\ell,\ell',\ell''} - \gamma_{\lfloor \ell'/2 \rfloor+ \mathcalboondox{i}_{\ell} \mathcalboondox{p}_{\ell'},1}^{\ell,\ell',\ell''} \big),
\end{split}
\end{equation}
whereas $\operatorname{SI}(n+1)^{m}_{\gan}(\omega)$ is given by 
\begin{equation}
\label{eq:sin+1mbis}
\begin{split}
   &\mathcalboondox{p}_{\ell-\ell'} \gamma_{2}^{\ell,\ell',\ell''}+ (-1)^{\mathcalboondox{p}_{\ell-\ell'}} \sum_{j=1}^{\lfloor \ell /2 \rfloor - 1 + \mathcalboondox{p}_{\ell} \mathcalboondox{p}_{\ell'}} \big(\gamma_{2j+\mathcalboondox{p}_{\ell-\ell'}}^{\ell,\ell',\ell''} - \gamma_{2j +\mathcalboondox{p}_{\ell-\ell'}+1}^{\ell,\ell',\ell''}\big) + (-1)^{\ell} \mathcalboondox{i}_{\ell'} \gamma_{\ell}^{\ell,\ell',\ell''},
\end{split}
\end{equation}
and $\operatorname{SI}(n+1)^{l}_{\gan}(\omega)$ is given by 
\begin{equation}
\label{eq:sin+1lbis}
\begin{split}
   &-(-1)^{\ell}\bigg(\sum_{j=1}^{\lfloor \ell''/2 \rfloor + \mathcalboondox{i}_{\ell'} \mathcalboondox{p}_{\ell''} - \mathcalboondox{p}_{\ell'-\ell''}} \big( \gamma_{j,\ell+1}^{\ell,\ell',\ell''} - \gamma_{j,\ell+2}^{\ell,\ell',\ell''}  \big) + \mathcalboondox{p}_{\ell'-\ell''} \gamma_{\lfloor \ell''/2 \rfloor+ \mathcalboondox{i}_{\ell'} \mathcalboondox{p}_{\ell''},\ell+1}^{\ell,\ell',\ell''}\bigg).
\end{split}
\end{equation}

\begin{fact}
\label{fact:sif}
Assume the terminology we introduced in Subsubsection \ref{subsubsec:generalsetting} and after Lemma \ref{lemma:sioddl1}. 
Set $\bar{\delta}_{0}^{\ell,\ell',\ell''} = C_{1,\ell''} C_{\ell,\ell'} \kappa(\Omega)$, where $\kappa(\Omega)$ was defined in \eqref{eq:kappa}.   
Then, $\operatorname{SI}(n+1)_{\gan}^{f}(\omega)$ is given by 
\begin{equation}
\label{eq:gamma0-1bisbisbis}
\begin{split}
        &\mathcalboondox{p}_{\ell-\ell'} \bar{\delta}_{0}^{\ell,\ell',\ell''} \tilde{\Sigma}_{f}^{\ell,\ell',\ell''}(\Omega) h_{1}(1_{A})
        \\ 
        &+ \sum_{j=1}^{\lfloor \ell'/2 \rfloor + \mathcalboondox{i}_{\ell} \mathcalboondox{p}_{\ell'}} \bar{\beta}_{j,f}^{\ell,\ell',\ell''} \Sigma_{j,f}^{\ell,\ell',\ell''}(\Omega) \big(g_{\ell'}(c) h_{1}(1_{A}) - g_{\ell'}(1_{A}) h_{1}(c)\big),     
        \end{split}     
\end{equation}
where 
\[     \bar{\beta}_{j,f}^{\ell,\ell',\ell''} = C_{\ell'-2j+2-\mathcalboondox{p}_{\ell},\ell''} C_{\ell,2j-1+\mathcalboondox{p}_{\ell}} \kappa(\Omega),       \]   
$\Sigma_{j,f}^{\ell,\ell',\ell''}(\Omega)$ is given by 
\begin{small}
\begin{equation}
 \label{eq:sigmaf}
       \begin{cases}
f_{1}(b) g_{1}(1_{A}) h_{\ell''}(a) + f_{1}(a) g_{1}(b) h_{\ell''}(1_{A})           &\text{}
\\
- f_{1}(ba) g_{1}(1_{A}) h_{\ell''}(1_{A}) - f_{1}(1_{A}) g_{1}(b) h_{\ell''}(a),  &\text{if $\ell = 1$;}
\\
&\text{}
\\
- \big(f_{1}(a) h_{\ell''}(1_{A}) - f_{1}(1_{A}) h_{\ell''}(a)\big)\big( f_{\ell}(b) g_{1}(1_{A}) - f_{\ell}(1_{A}) g_{1}(b) \big), &\text{if $\ell > 1$ odd, or}
\\
&\text{$\ell > 1$ even and $j>1$;}
\\
g_{1}(1_{A}) \big(f_{1}(a) f_{\ell}(1_{A}) h_{\ell''}(b) + f_{1}(1_{A}) f_{\ell}(b) h_{\ell''}(a) &\text{}
\\
- f_{1}(a) f_{\ell}(b) h_{\ell''}(1_{A}) - f_{1}(1_{A}) f_{\ell}(1_{A}) h_{\ell''}(ab)\big),  &\text{if $\ell$ even and $j = 1$;}
 \end{cases}
\end{equation}
\end{small}
\hskip -0.8mm and $\tilde{\Sigma}_{f}^{\ell,\ell',\ell''}(\Omega)$ is given by 
\begin{small}
\begin{equation}
\label{eq:tildesigmaf1}
\begin{cases}
 \big(f_{1}(ba) g_{1}(1_{A}) g_{\ell'}(c)  + f_{1}(ca) g_{1}(b) g_{\ell'}(1_{A})          &\text{}
\\
- f_{1}(a) g_{1}(b) g_{\ell'}(c) - f_{1}(bca) g_{1}(1_{A}) g_{\ell'}(1_{A})\big) h_{\ell''}(1_{A})  &\text{}
\\
+ \big(f_{1}(bc) g_{1}(1_{A}) g_{\ell'}(1_{A})  + f_{1}(1_{A}) g_{1}(b) g_{\ell'}(c)         &\text{}
\\
- f_{1}(b) g_{1}(1_{A}) g_{\ell'}(c) - f_{1}(c) g_{1}(b) g_{\ell'}(1_{A})\big) h_{\ell''}(a),  &\text{if $\ell = 1$,}
\\
&\text{}
\\
\Big[ \big(f_{1}(a) g_{\ell'}(c) - f_{1}(c a) g_{\ell'}(1_{A})\big) h_{\ell''}(1_{A}) 
- \big(f_{1}(1_{A}) g_{\ell'}(c) - f_{1}(c) g_{\ell'}(1_{A})\big) h_{\ell''}(a) \Big]       &\text{}
\\
\big(f_{\ell}(b) g_{1}(1_{A}) - f_{\ell}(1_{A}) g_{1}(b) \big) ,  &\text{if $\ell > 1$.}
\end{cases}
\end{equation}
\end{small}
\end{fact}
\begin{proof}
Using definition \eqref{eq:mun} for $m_{n-\ell-2j+2-\mathcalboondox{p}_{\ell}}$ in \eqref{eq:gammaj0}-\eqref{eq:gammaj1}, the cyclicity of the higher operations and the fact they are $1_{A}$-normalized, we see that $\gamma_{j,0}^{\ell,\ell',\ell''} - \gamma_{j,1}^{\ell,\ell',\ell''}$ equals 
\begin{small}
\begin{equation}
\label{eq:gammaj1-2}
     \beta_{j,f}^{\ell,\ell',\ell''} \gan\big(m_{\ell+2j+\mathcalboondox{p}_{\ell}}(a,tf_{1}, \dots, tf_{\ell},b,tg_{1},\dots, tg_{2j-2+\mathcalboondox{p}_{\ell}}),t h_{\ell''} \big) \big(g_{\ell'}(c) h_{1}(1_{A}) - g_{\ell'}(1_{A})h_{1}(c)\big), 
\end{equation}
\end{small}
\hskip -0.8mm for all $j \in \{ 1, \dots, \lfloor \ell'/2 \rfloor + \mathcalboondox{i}_{\ell} \mathcalboondox{p}_{\ell'} - \mathcalboondox{p}_{\ell-\ell'} \}$, where 
\[     \beta_{j,f}^{\ell,\ell',\ell''} = 
C_{\ell'-2j+2-\mathcalboondox{p}_{\ell},\ell''} \prod_{i=2j-1+\mathcalboondox{p}_{\ell}}^{\ell'-1} g_{i}(1_{A}) \prod_{i=2}^{\ell''-1} h_{i}(1_{A}).     \]

Using definition \eqref{eq:mun} for $m_{\ell+2j+\mathcalboondox{p}_{\ell}}$, we see that \eqref{eq:gammaj1-2} is given by 
\begin{equation}
\label{eq:gammaj1-2bis}
     \bar{\beta}_{j,f}^{\ell,\ell',\ell''} \Sigma_{j,f}^{\ell,\ell',\ell''}(\Omega) \big(g_{\ell'}(c) h_{1}(1_{A}) - g_{\ell'}(1_{A}) h_{1}(c)\big), 
\end{equation}
for all $j \in \{ 1, \dots, \lfloor \ell'/2 \rfloor + \mathcalboondox{i}_{\ell} \mathcalboondox{p}_{\ell'} - \mathcalboondox{p}_{\ell-\ell'} \}$. 

If $\mathcalboondox{p}_{\ell-\ell'} = 1$, using definition \eqref{eq:mun} for $m_{n-\ell-2j+2-\mathcalboondox{p}_{\ell}}$, we see that $\gamma_{\lfloor \ell'/2 \rfloor + \mathcalboondox{i}_{\ell} \mathcalboondox{p}_{\ell'},0}^{\ell,\ell',\ell''}$ is written as a sum of four nonzero terms, whereas  $\gamma_{\lfloor \ell'/2 \rfloor + \mathcalboondox{i}_{\ell} \mathcalboondox{p}_{\ell'},1}^{\ell,\ell',\ell''}$ is written as a sum of three nonzero terms. 
Moreover, one of the four terms of $\gamma_{\lfloor \ell'/2 \rfloor + \mathcalboondox{i}_{\ell} \mathcalboondox{p}_{\ell'},0}^{\ell,\ell',\ell''}$ cancels one of the three terms of $\gamma_{\lfloor \ell'/2 \rfloor + \mathcalboondox{i}_{\ell} \mathcalboondox{p}_{\ell'},1}^{\ell,\ell',\ell''}$, and 
two of the remaining three nonzero terms of $\gamma_{\lfloor \ell'/2 \rfloor + \mathcalboondox{i}_{\ell} \mathcalboondox{p}_{\ell'},0}^{\ell,\ell',\ell''}$ give exactly \eqref{eq:gammaj1-2} for $j = \lfloor \ell'/2 \rfloor + \mathcalboondox{i}_{\ell} \mathcalboondox{p}_{\ell'}$, so \eqref{eq:gammaj1-2bis} for that value of $j$. 
The remaining three nonzero terms of $\gamma_{\lfloor \ell'/2 \rfloor + \mathcalboondox{i}_{\ell} \mathcalboondox{p}_{\ell'},0}^{\ell,\ell',\ell''} - \gamma_{\lfloor \ell'/2 \rfloor + \mathcalboondox{i}_{\ell} \mathcalboondox{p}_{\ell'},1}^{\ell,\ell',\ell''}$ give 
\begin{equation}
\label{eq:gammajbisbis}
\begin{split}
     \delta_{0} &\Big(-\gan\big(c.m_{\ell+\ell'+1}(a,tf_{1},\dots,tf_{\ell},b,tg_{1},\dots,tg_{\ell'-1}), t g_{\ell'}\big) h_{\ell''}(1_{A}) 
     \\
     &+ \gan\big(m_{\ell+\ell'+1}(ca,tf_{1},\dots,tf_{\ell},b,tg_{1},\dots,tg_{\ell'-1}),t g_{\ell'} \big) h_{\ell''}(1_{A}) 
     \\
     &- \gan\big(m_{\ell+\ell'+1}(c,tf_{1},\dots,tf_{\ell},b,tg_{1},\dots,tg_{\ell'-1}),t g_{\ell'} \big) h_{\ell''}(a) \Big),     
\end{split}     
\end{equation}
where 
\[     \delta_{0} = C_{1,\ell''} \prod_{i=1}^{\ell''-1} h_{i}(1_{A}).     \]

Finally, using definition \eqref{eq:mun} for $m_{\ell+\ell'+1}$ in \eqref{eq:gammajbisbis}, we see that the latter expression is $\bar{\delta}_{0}^{\ell,\ell',\ell''}$ times $\tilde{\Sigma}_{f}^{\ell,\ell',\ell''}(\Omega) h_{1}(1_{A})$. 
The claimed expression of $\operatorname{SI}(n+1)_{\gan}^{f}(\omega)$ follows from \eqref{eq:sin+1fbis} and the previous computations. 
\end{proof}

\begin{fact}
\label{fact:sil}
Assume the terminology we introduced in Subsubsection \ref{subsubsec:generalsetting} and after Lemma \ref{lemma:sioddl1}. 
Recall that $\bar{\delta}_{0}^{\ell',\ell'',\ell} = C_{1,\ell} C_{\ell',\ell''} \kappa(\Omega)$, where $\kappa(\Omega)$ was defined in \eqref{eq:kappa}. 
Then, $\operatorname{SI}(n+1)^{l}_{\gan}(\omega)$ is precisely 
\begin{equation}
\label{eq:gammal+1-l+2bisbisbis}
         \sum_{j=1}^{\lfloor \ell''/2 \rfloor} \bar{\beta}_{j,l}^{1,\ell',\ell''} \Big( \Sigma_{j,l}^{1,\ell',\ell''}(\Omega) h_{\ell''}(a) 
        - \Sigma_{j,l}^{1,\ell',\ell''}(\Omega'') h_{\ell''}(1_{A})\Big)         
\end{equation}
if $\ell = 1$, and
\begin{equation}
\label{eq:gammal+1-l+2bisbisbis2}
\begin{split}
        (-1)^{\ell+1} \bigg(&\mathcalboondox{p}_{\ell'-\ell''} \bar{\delta}_{0}^{\ell',\ell'',\ell} f_{\ell}(1_{A}) \tilde{\Sigma}_{l}^{\ell,\ell',\ell''}(\Omega) 
        \\ 
        &+ \sum_{j=1}^{\lfloor \ell''/2 \rfloor + \mathcalboondox{i}_{\ell'} \mathcalboondox{p}_{\ell''}} \bar{\beta}_{j,l}^{\ell,\ell',\ell''} \Sigma_{j,l}^{\ell,\ell',\ell''}(\Omega) \big(f_{1}(1_{A}) h_{\ell''}(a) - f_{1}(a) h_{\ell''}(1_{A})\big)\bigg)     
        \end{split}     
\end{equation}
if $\ell > 1$, where 
\[     \bar{\beta}_{j,l}^{\ell,\ell',\ell''} = C_{\ell, \ell''-2j+2-\mathcalboondox{p}_{\ell'}} C_{\ell',2j-1+\mathcalboondox{p}_{\ell'}} \kappa(\Omega),       \]  
$\Omega''$ is given by 
\[     (1_{A}, a \cdot tf_{1}, tf_{2}, \dots, tf_{\ell}, b, tg_{1}, \dots, tg_{\ell'}, c, th_{1}, \dots, th_{\ell''}),     \]
$\Sigma_{j,l}^{\ell,\ell',\ell''}(\Omega)$ is given by 
\begin{small}
\begin{equation}
\label{eq:sigmal}
      \begin{cases}
f_{\ell}(b) g_{1}(c) h_{1}(1_{A}) + f_{\ell}(1_{A}) g_{1}(b) h_{1}(c)           &\text{}
\\
- f_{\ell}(1_{A}) g_{1}(c b) h_{1}(1_{A}) - f_{\ell}(b) g_{1}(1_{A}) h_{1}(c),  &\text{if $\ell' = 1$;}
\\
&\text{}
\\
\big( f_{\ell}(b) g_{1}(1_{A}) - f_{\ell}(1_{A}) g_{1}(b) \big)
\big( g_{\ell'}(c) h_{1}(1_{A}) - g_{\ell'}(1_{A}) h_{1}(c) \big),  &\text{if $\ell' > 1$ odd, or}
\\
&\text{$\ell'$ even and $j > 1$;}
\\
\big(f_{\ell}(b) g_{1}(1_{A}) g_{\ell'}(c) + f_{\ell}(c) g_{1}(b) g_{\ell'}(1_{A}) &\text{}
\\
- f_{\ell}(b c) g_{1}(1_{A}) g_{\ell'}(1_{A}) - f_{\ell}(1_{A}) g_{1}(b) g_{\ell'}(c) \big) h_{1}(1_{A}),  &\text{if $\ell'$ even and $j = 1$;}
 \end{cases}
\end{equation}
\end{small}
\hskip -0.8mm and $\tilde{\Sigma}_{l}^{\ell,\ell',\ell''}(\Omega)$ is given by 
 \begin{equation}
 \label{eq:tildesigmal}
 \begin{split}
&\Big[ f_{1}(a) \big(g_{1}(1_{A}) h_{\ell''}(b)  - g_{1}(b) h_{\ell''}(1_{A}) \big) - f_{1}(1_{A}) \big(g_{1}(1_{A}) h_{\ell''}(a b)  - g_{1}(b) h_{\ell''}(a) \big) \Big]         
\\
&\big( g_{\ell'}(c) h_{1}(1_{A}) - g_{\ell'}(1_{A}) h_{1}(c) \big).  
\end{split}
\end{equation}
\end{fact}
\begin{proof}
From definition \eqref{eq:mun} for $m_{n-\ell'-2j+2-\mathcalboondox{p}_{\ell'}}$ in \eqref{eq:gammajl+1}-\eqref{eq:gammajl+2}, the cyclicity of the higher operations and the fact they are $1_{A}$-normalized, we see that $\gamma_{j,\ell+1}^{\ell,\ell',\ell''} - \gamma_{j,\ell+2}^{\ell,\ell',\ell''}$ equals 
\begin{small}
\begin{equation}
\label{eq:gammajl+1-l+2}
     \begin{cases} 
         \Big( f_{1}\big(m_{\ell'+2j+\mathcalboondox{p}_{\ell'}}(b,tg_{1},\dots, tg_{\ell'}, c, th_{1}, \dots, th_{2j-2+ \mathcalboondox{p}_{\ell'}} )\big) h_{\ell''}(a) &\text{}
     \\
     - f_{1}\big(m_{\ell'+2j+\mathcalboondox{p}_{\ell'}}(b,tg_{1},\dots, tg_{\ell'}, c, th_{1}, \dots, th_{2j-2+ \mathcalboondox{p}_{\ell'}} ) a \big) h_{\ell''}(1_{A})\Big) \beta_{j,l}^{1,\ell',\ell''}, &\text{if $\ell = 1$, } 
          \\ &\text{}
     \\
     f_{\ell}\big(m_{\ell'+2j+\mathcalboondox{p}_{\ell'}}(b,tg_{1},\dots, tg_{\ell'}, c, th_{1}, \dots, th_{2j-2+ \mathcalboondox{p}_{\ell'}} )\big) &\text{}
     \\
     \big(f_{1}(1_{A}) h_{\ell''}(a) - f_{1}(a) h_{\ell''}(1_{A})\big) \beta_{j,l}^{\ell,\ell',\ell''}, &\text{if $\ell > 1$, }
     \end{cases}
\end{equation}
\end{small}
\hskip -0.8mm for all $j \in \{ 1, \dots, \lfloor \ell''/2 \rfloor + \mathcalboondox{i}_{\ell'} \mathcalboondox{p}_{\ell''} - \mathcalboondox{p}_{\ell'-\ell''} \}$, where 
\[     \beta_{j,l}^{\ell,\ell',\ell''} = C_{\ell, \ell''-2j+2-\mathcalboondox{p}_{\ell'}} \prod_{i=2}^{\ell-1} f_{i}(1_{A}) \prod_{i=2j-1+\mathcalboondox{p}_{\ell'}}^{\ell''-1} h_{i}(1_{A}).     \] 

Using definition \eqref{eq:mun} for $m_{\ell'+2j+\mathcalboondox{p}_{\ell'}}$, we see that \eqref{eq:gammajl+1-l+2} is given by 
\begin{small}
\begin{equation}
\label{eq:gammajbisl+1-l+2}
     \begin{cases} 
         \Big( \Sigma_{j,l}^{1,\ell',\ell''}(\Omega) h_{\ell''}(a) - \Sigma_{j,l}^{1,\ell',\ell''}(\Omega'') h_{\ell''}(1_{A})\Big) \bar{\beta}_{j,l}^{1,\ell',\ell''}, &\text{if $\ell = 1$,} 
         \\
    &\text{}
     \\
     \Sigma_{j,l}^{\ell,\ell',\ell''}(\Omega) \big(f_{1}(1_{A}) h_{\ell''}(a) - f_{1}(a) h_{\ell''}(1_{A})\big) \bar{\beta}_{j,l}^{\ell,\ell',\ell''},  &\text{if $\ell > 1$,}
     \end{cases}
\end{equation}
\end{small}
\hskip -0.8mm for all $j \in \{ 1, \dots, \lfloor \ell''/2 \rfloor + \mathcalboondox{i}_{\ell'} \mathcalboondox{p}_{\ell''} - \mathcalboondox{p}_{\ell'-\ell''} \}$. 

If $\mathcalboondox{p}_{\ell'-\ell''} = 1$, definition \eqref{eq:mun} for $m_{n-\ell'-2j+2-\mathcalboondox{p}_{\ell'}}$ tells us that $\gamma_{\lfloor \ell''/2 \rfloor + \mathcalboondox{i}_{\ell'} \mathcalboondox{p}_{\ell''},\ell+1}^{\ell,\ell',\ell''}$ is written as a sum of four nonzero terms. 
Note that in this case we can assume that $\ell \geq 1$ is even. 
Indeed, since $\ell + \ell' + \ell'' = n -1$ is odd, $\ell'$ and $\ell''$ must have the same parity if $\ell$ is odd, \textit{i.e.} $\mathcalboondox{p}_{\ell'-\ell''}= 0$ if $\ell$ is odd. 
Moreover, two of the four nonzero terms give exactly \eqref{eq:gammajl+1-l+2} for $j = \lfloor \ell''/2 \rfloor + \mathcalboondox{i}_{\ell'} \mathcalboondox{p}_{\ell''}$, so \eqref{eq:gammajbisl+1-l+2} for that value of $j$. 
The remaining two nonzero terms of $\gamma_{\lfloor \ell''/2 \rfloor + \mathcalboondox{i}_{\ell'} \mathcalboondox{p}_{\ell''},\ell+1}^{\ell,\ell',\ell''}$ give 
\begin{equation}
\label{eq:gammajl+1bisbis}
\begin{split}
     \xi & \Big(f_{1}(a) \gan\big(m_{\ell'+\ell''+1}(b,tg_{1},\dots,tg_{\ell'},c,th_{1},\dots,th_{\ell''-1}), th_{\ell''}\big) 
     \\
     &- f_{1}(1_{A}) \gan\big(a.m_{\ell'+\ell''+1}(b,tg_{1},\dots,tg_{\ell'},c,th_{1},\dots,th_{\ell''-1}), t h_{\ell''}\big) \Big),     
\end{split}     
\end{equation}
where 
\[     \xi = C_{1,\ell} \prod_{i=2}^{\ell} f_{i}(1_{A}).     \]
Using definition \eqref{eq:mun} for $m_{\ell'+\ell''+1}$ in \eqref{eq:gammajl+1bisbis}, we see that the latter expression is $\bar{\delta}_{0}^{\ell',\ell'',\ell}$ times $f_{\ell}(1_{A}) \tilde{\Sigma}_{l}^{\ell,\ell',\ell''}(\Omega)$.
Note that we have used that $\ell', \ell'' \geq \ell \geq 2$. 
The claimed expression of $\operatorname{SI}(n+1)_{\gan}^{l}(\omega)$ follows from \eqref{eq:sin+1lbis} and the previous computations. 
\end{proof}

\begin{fact}
\label{fact:sim}
Assume the terminology we introduced in Subsubsection \ref{subsubsec:generalsetting} and after Lemma \ref{lemma:sioddl1}. 
Suppose further that $\ell > 1$.
Recall that $\bar{\delta}_{0}^{\ell,\ell',\ell''} = C_{1,\ell''} C_{\ell,\ell'} \kappa(\Omega)$, where $\kappa(\Omega)$ was defined in \eqref{eq:kappa}. 
Then, $\operatorname{SI}(n+1)_{\gan}^{m}(\omega)$ is given by 
\begin{equation}
\label{eq:gammarbisbisbis}
\begin{split}
        &- \mathcalboondox{p}_{\ell-\ell'} \bar{\delta}_{0}^{\ell,\ell',\ell''} \tilde{\Sigma}_{f}^{\ell,\ell',\ell''}(\Omega) h_{1}(1_{A})
        \\ 
        &+ (-1)^{\mathcalboondox{p}_{\ell-\ell'}} \sum_{j=\mathcalboondox{i}_{\ell-\ell'}}^{\lfloor \ell/2 \rfloor -1+ \mathcalboondox{p}_{\ell} \mathcalboondox{p}_{\ell'}} \bar{\beta}_{j,m}^{\ell,\ell',\ell''} \Sigma_{m}^{\ell,\ell',\ell''}(\Omega) \big(f_{1}(1_{A}) h_{\ell''}(a)  - f_{1}(a) h_{\ell''}(1_{A})\big)
        \\
        &+(-1)^{\ell} \mathcalboondox{i}_{\ell'} \bar{\beta}_{1,l}^{\ell,\ell',\ell''} \Sigma_{1,l}^{\ell,\ell',\ell''}(\Omega) \big(f_{1}(1_{A}) h_{\ell''}(a)  - f_{1}(a) h_{\ell''}(1_{A})\big),     
        \end{split}     
\end{equation}
where 
\[     \bar{\beta}_{j,m}^{\ell,\ell',\ell''} = C_{2j+\mathcalboondox{p}_{\ell-\ell'},\ell''} C_{\ell',\ell-2j+1-\mathcalboondox{p}_{\ell-\ell'}} \kappa(\Omega),       \]   
$\Sigma_{m}^{\ell,\ell',\ell''}(\Omega)$ is given by 
\begin{equation}
\label{eq:sigmam}
\begin{split}
\big( f_{\ell}(b) g_{1}(1_{A}) - f_{\ell}(1_{A}) g_{1}(b) \big)
\big( g_{\ell'}(c) h_{1}(1_{A}) - g_{\ell'}(1_{A}) h_{1}(c) \big),   
 \end{split}
\end{equation}
and $\tilde{\Sigma}_{f}^{\ell,\ell',\ell''}(\Omega)$, $\bar{\beta}_{1,l}^{\ell,\ell',\ell''}$ and $\Sigma_{1,l}^{\ell,\ell',\ell''}(\Omega)$ were defined in Facts \ref{fact:sif} 
and \ref{fact:sil}, respectively.
\end{fact}
\begin{proof}
Using \eqref{eq:mun} for $m_{r+\ell''+\mathcalboondox{i}_{\ell+\ell'-r}}$ in \eqref{eq:gammareven}-\eqref{eq:gammarodd}, 
the cyclicity of the higher operations and the fact they are $1_{A}$-normalized, we see that $\gamma_{2j+\mathcalboondox{p}_{\ell-\ell'}}^{\ell,\ell',\ell''} - \gamma_{2j+\mathcalboondox{p}_{\ell-\ell'}+1}^{\ell,\ell',\ell''}$ equals 
\begin{equation}
\label{eq:gammar}
\begin{split}
     \beta_{j,m}^{\ell,\ell',\ell''} \gan\big(m_{n-\ell''+1-2j-\mathcalboondox{p}_{\ell-\ell'}}(tf_{2j+1+\mathcalboondox{p}_{\ell-\ell'}}, &\dots, tf_{\ell},b,tg_{1},\dots, tg_{\ell'},c), t h_{1} \big) 
     \\
     &\big(f_{1}(1_{A}) h_{\ell''}(a) - f_{1}(a) h_{\ell''}(1_{A}) \big), 
\end{split}
\end{equation}
for all $j \in \{ 1, \dots, \lfloor \ell/2 \rfloor -1 + \mathcalboondox{p}_{\ell} \mathcalboondox{p}_{\ell'} \}$, where 
\[     \beta_{j,m}^{\ell,\ell',\ell''} = C_{2j+\mathcalboondox{p}_{\ell-\ell'},\ell''} \prod_{i=2}^{2j+\mathcalboondox{p}_{\ell - \ell'}} f_{i}(1_{A}) \prod_{i=2}^{\ell''-1} h_{i}(1_{A}).     \]
From definition \eqref{eq:mun} for $m_{n-\ell''+1-2j-\mathcalboondox{p}_{\ell-\ell'}}$ we see that \eqref{eq:gammar} is given by 
\begin{equation}
\label{eq:gammarbis}
     \bar{\beta}_{j,m}^{\ell,\ell',\ell''} \Sigma_{m}^{\ell,\ell',\ell''}(\Omega) \big(f_{1}(1_{A}) h_{\ell''}(a) - f_{1}(a) h_{\ell''}(1_{A}) \big), 
\end{equation}
for all $j \in \{ 1, \dots, \lfloor \ell/2 \rfloor -1 + \mathcalboondox{p}_{\ell} \mathcalboondox{p}_{\ell'} \}$. 
Note that we have used that $\ell' \geq \ell \geq 2$ in this case. 

If $\mathcalboondox{p}_{\ell - \ell'} =1$, from definition \eqref{eq:mun} for $m_{2+\ell''}$, we see that $\gamma_{2}^{\ell,\ell',\ell''}$ is written as a sum of four nonzero terms. 
Two of the these nonzero terms give exactly minus \eqref{eq:gammar} for $j=0$, so minus \eqref{eq:gammarbis} for that value of $j$, and the remaining two nonzero terms are precisely 
\begin{equation}
\label{eq:tildesigmampre}
\begin{split}
\theta \Big(&\gan\big(m_{\ell+\ell'+1}(tf_{2},\dots,tf_{\ell},b,tg_{1},\dots,tg_{\ell'},c), t f_{1} \big) h_{\ell''}(a) 
\\
&- \gan\big(m_{\ell+\ell'+1}(tf_{2},\dots,tf_{\ell},b,tg_{1},\dots,tg_{\ell'},c).a, t f_{1}\big) h_{\ell''}(1_{A})\Big),  
\end{split}
\end{equation}
where 
\[     \theta = C_{1,\ell''} \prod_{i=1}^{\ell''-1} h_{i}(1_{A}).     \]
Using definition \eqref{eq:mun} for $m_{\ell+\ell'+1}$, we obtain that \eqref{eq:tildesigmampre} gives minus 
$\bar{\delta}_{0}^{\ell,\ell',\ell''}$ times $\tilde{\Sigma}_{f}^{\ell,\ell',\ell''}(\Omega) h_{1}(1_{A})$. 
Note that the first term of \eqref{eq:tildesigmampre} is the opposite of the last term of \eqref{eq:gammajbisbis}. 

On the other hand, if $\ell'$ is even, from definition \eqref{eq:mun} for $m_{\ell+\ell''+1}$, we see that $\gamma_{\ell}^{\ell,\ell',\ell''}$ is 
precisely 
\begin{equation}
\label{eq:gammarl}
\begin{split}
     \bar{\beta}_{1,l}^{\ell,\ell',\ell''} \Sigma_{1,l}^{\ell,\ell',\ell''}(\Omega) \big(f_{1}(1_{A}) h_{\ell''}(a) - f_{1}(a) h_{\ell''}(1_{A}) \big). 
\end{split}
\end{equation}
The claimed expression of $\operatorname{SI}(n+1)_{\gan}^{m}(\omega)$ follows from \eqref{eq:sin+1mbis} and the previous computations. 
\end{proof}

Combining the previous facts we obtain the following result.
\begin{fact}
\label{fact:codd}
Assume the terminology we introduced in Subsubsection \ref{subsubsec:generalsetting} and after Lemma \ref{lemma:sioddl1}. 
Suppose that $\ell, \ell', \ell''$ are odd. 
Then, $\operatorname{SI}(n+1)_{\gan}(\omega) = 0$ (for $\ell, \ell', \ell''$ fixed as above, and any $\omega$) if and only if 
\begin{equation}
\label{eq:Codd}
\begin{split}
     &C_{\ell', \ell''+\ell} + C_{\ell'', \ell'+\ell} + C_{\ell,\ell'+\ell''} 
     \\
     &= \sum_{j=1}^{\lfloor \ell'/2 \rfloor} C_{\ell,2j} C_{\ell'',\ell'- 2 j+1} + \sum_{j=1}^{\lfloor \ell''/2 \rfloor} C_{\ell,\ell''-2j+1} C_{\ell',2 j} + \sum_{j=1}^{\lfloor \ell/2 \rfloor} C_{\ell'',2j} C_{\ell',\ell- 2 j+1}.
\end{split}
\end{equation}
\end{fact}
\begin{proof}
Recall that, by definition, $\operatorname{SI}(n+1)^{2,n}_{\gan}(\omega) = \beta^{\ell,\ell',\ell''} E^{\ell,\ell',\ell''}$, where $\beta^{\ell,\ell',\ell''}$ is given in \eqref{eq:betagen}, $E^{\ell,\ell',\ell''}$ in \eqref{eq:sin+12nbisl=l'=1}-\eqref{eq:sin+12nbisl>1} and $\kappa(\Omega)$ in \eqref{eq:kappa}. 

Assume first that $\ell = \ell' = 1$. 
Facts \ref{fact:sif}-\ref{fact:sim} tell us that $\operatorname{SI}(n+1)_{\gan}^{f}(\omega)$ and $\operatorname{SI}(n+1)_{\gan}^{m}(\omega)$ vanish, as well as  
\[     \operatorname{SI}(n+1)_{\gan}^{l}(\omega) = - \bigg(\sum_{j=1}^{n/2 - 2} C_{1,n-2-2j} C_{1, 2 j} \bigg) \kappa(\Omega) E^{1,1,\ell''},       \]
whereas \eqref{eq:betagen} gives 
\[     \operatorname{SI}(n+1)_{\gan}^{2,n}(\omega) = \big(C_{2,n-3} + 2 C_{1,n-2}\big) \kappa(\Omega) E^{1,1,\ell''}.      \]
This implies that $\operatorname{SI}(n+1)_{\gan}(\omega) = 0$ (for $\ell = \ell' = 1$ and any $\omega$) is equivalent to 
\begin{equation}
\label{eq:C1}
     C_{2,n-3} + 2 C_{1,n-2} = \sum_{j=1}^{n/2 - 2} C_{1,n-2-2j} C_{1, 2 j},
\end{equation}
which gives \eqref{eq:Codd} in this case. 

Assume now that $\ell = 1$ and $\ell' > 1$ is odd. 
Fact \ref{fact:sim} tells us that $\operatorname{SI}(n+1)_{\gan}^{m}(\omega)$ vanishes, whereas Facts \ref{fact:sif}-\ref{fact:sil} imply that 
\[     \operatorname{SI}(n+1)_{\gan}^{f}(\omega) = -\bigg(\sum_{j=1}^{\lfloor \ell'/2 \rfloor} C_{1,2j} C_{\ell'',\ell'- 2 j+1} \bigg)  \kappa(\Omega)E^{1,\ell',\ell''}     \]
and 
\[     \operatorname{SI}(n+1)_{\gan}^{l}(\omega) = -\bigg(\sum_{j=1}^{\lfloor \ell''/2 \rfloor} C_{1,\ell''-2j+1} C_{\ell',2 j} \bigg) \kappa(\Omega) E^{1,\ell',\ell''},      \]
whereas \eqref{eq:betagen} gives 
\[     \operatorname{SI}(n+1)_{\gan}^{2,n}(\omega) = \big(C_{\ell', \ell''+1} + C_{\ell'', \ell'+1} + C_{1,\ell'+\ell''}\big) \kappa(\Omega) E^{1,\ell',\ell''}.      \]
This implies that $\operatorname{SI}(n+1)_{\gan}(\omega) = 0$ (for $\ell = 1$ and $\ell' > 1$ odd, and any $\omega$) is tantamount to 
\begin{equation}
\label{eq:C2}
     C_{\ell', \ell''+1} + C_{\ell'', \ell'+1} + C_{1,\ell'+\ell''} = \sum_{j=1}^{\lfloor \ell'/2 \rfloor} C_{1,2j} C_{\ell'',\ell'- 2 j+1} + \sum_{j=1}^{\lfloor \ell''/2 \rfloor} C_{1,\ell''-2j+1} C_{\ell',2 j},
\end{equation}
which gives \eqref{eq:Codd} in this case. 

Assume finally that $\ell, \ell', \ell'' > 1$ are odd.
A direct computation using Facts \ref{fact:sif}, \ref{fact:sil} and \ref{fact:sim} shows that 
\begin{small}
\begin{equation}
\begin{split}
     \operatorname{SI}(n+1)_{\gan}^{\neq 2, n}&(\omega) = - \kappa(\Omega) E^{\ell,\ell',\ell''}
     \\
     &\bigg(\sum_{j=1}^{\lfloor \ell'/2 \rfloor} C_{\ell,2j} C_{\ell'',\ell'- 2 j+1} + \sum_{j=1}^{\lfloor \ell''/2 \rfloor} C_{\ell,\ell''-2j+1} C_{\ell',2 j} + \sum_{j=1}^{\lfloor \ell/2 \rfloor} C_{\ell',\ell-2j+1} C_{\ell'',2 j} \bigg),     
\end{split}
     \end{equation}
\end{small}
\hskip -0.8mm whereas \eqref{eq:betagen} gives 
\[     \operatorname{SI}(n+1)_{\gan}^{2,n}(\omega) = \big(C_{\ell', \ell''+\ell} + C_{\ell'', \ell'+\ell} + C_{\ell,\ell'+\ell''}\big) \kappa(\Omega) E^{\ell,\ell',\ell''}.      \]
This implies that $\operatorname{SI}(n+1)_{\gan}(\omega) = 0$ (for $\ell, \ell', \ell'' > 1$ odd, and any $\omega$) is equivalent to 
\begin{equation}
\label{eq:C4}
\begin{split}
     &C_{\ell', \ell''+\ell} + C_{\ell'', \ell'+\ell} + C_{\ell,\ell'+\ell''} 
     \\
     &= \sum_{j=1}^{\lfloor \ell'/2 \rfloor} C_{\ell,2j} C_{\ell'',\ell'- 2 j+1} + \sum_{j=1}^{\lfloor \ell''/2 \rfloor} C_{\ell,\ell''-2j+1} C_{\ell',2 j} + \sum_{j=1}^{\lfloor \ell/2 \rfloor} C_{\ell'',2j} C_{\ell',\ell- 2 j+1}.
\end{split}
\end{equation}
\end{proof}

\begin{fact}
\label{fact:ceven}
Assume the terminology we introduced in Subsubsection \ref{subsubsec:generalsetting} and after Lemma \ref{lemma:sioddl1}. 
Suppose that two parameters among $\ell, \ell', \ell''$ are even and the other is odd. 
Let $\ell_{1}$ and $\ell_{2}$ be the two even parameter and $\ell_{3}$ be the odd one.
Then, $\operatorname{SI}(n+1)_{\gan}(\omega) = 0$ (for $\ell, \ell', \ell''$ fixed as above, and any $\omega$) if and only if 
\begin{small}
\begin{equation}
\label{eq:Ceven}
\begin{split}
     &C_{\ell_{1}, \ell_{2}+\ell_{3}} - C_{\ell_{3}, \ell_{1}+\ell_{2}} + C_{\ell_{2},\ell_{1}+\ell_{3}} 
     \\
     &= \sum_{j=1}^{\lfloor \ell_{2}/2 \rfloor} C_{\ell_{1},2j-1} C_{\ell_{3},\ell_{2}- 2 j+2} - \sum_{j=1}^{\lfloor \ell_{3}/2 \rfloor + 1} C_{\ell_{1},\ell_{3}-2j+2} C_{\ell_{2},2 j-1} + \sum_{j=1}^{\lfloor \ell_{1}/2 \rfloor} C_{\ell_{2},2j-1} C_{\ell_{3},\ell_{1}-2 j+2}.
\end{split}
\end{equation}
\end{small}
\end{fact}
\begin{proof}
Recall that, by definition, $\operatorname{SI}(n+1)^{2,n}_{\gan}(\omega) = \beta^{\ell,\ell',\ell''} E^{\ell,\ell',\ell''}$, where $\beta^{\ell,\ell',\ell''}$ is given in \eqref{eq:betagen}, $E^{\ell,\ell',\ell''}$ in \eqref{eq:sin+12nbisl=l'=1}-\eqref{eq:sin+12nbisl>1} and $\kappa(\Omega)$ in \eqref{eq:kappa}. 

Suppose first that $\ell = 1$ and $\ell' > 1$ is even, so $\ell''$ is also even. 
Fact \ref{fact:sim} tells us that $\operatorname{SI}(n+1)_{\gan}^{m}(\omega) = 0$. 
Moreover, one can easily check that the term in the first line of \eqref{eq:gamma0-1bisbisbis} 
coincides with minus the term indexed by $j=1$ in the sum in \eqref{eq:gammal+1-l+2bisbisbis}. 
Using Facts \ref{fact:sif} and \ref{fact:sil} we see that 
\begin{small}
\[     \operatorname{SI}(n+1)_{\gan}^{\neq 2, n}(\omega) = -\bigg(\sum_{j=1}^{\lfloor \ell'/2 \rfloor} C_{1,2j} C_{\ell'',\ell'- 2 j+1} + \sum_{j=2}^{\lfloor \ell''/2 \rfloor} C_{1,\ell''-2j+2} C_{\ell',2 j-1} \bigg) \kappa(\Omega) E^{1,\ell',\ell''},     \]
\end{small}
\hskip -0.8mm whereas \eqref{eq:betagen} gives 
\[     \operatorname{SI}(n+1)_{\gan}^{2,n}(\omega) = \big(C_{\ell', \ell''+1} + C_{\ell'', \ell'+1} - C_{1,\ell'+\ell''}\big) \kappa(\Omega) E^{1,\ell',\ell''}.      \]
This implies that $\operatorname{SI}(n+1)_{\gan}(\omega) = 0$ (for $\ell = 1$ and $\ell' > 1$ even, and any $\omega$) is tantamount to
\begin{equation}
\label{eq:C3}
     C_{\ell', \ell''+1} + C_{\ell'', \ell'+1} - C_{1,\ell'+\ell''} = \sum_{j=1}^{\lfloor \ell'/2 \rfloor} C_{1,2j} C_{\ell'',\ell'- 2 j+1} + \sum_{j=2}^{\lfloor \ell''/2 \rfloor} C_{1,\ell''-2j+2} C_{\ell',2 j-1},
\end{equation}
which coincides with \eqref{eq:Ceven} for $\ell_{1} = \ell'$, $\ell_{2} = \ell''$ and $\ell_{3}=1$. 
Indeed, reindexing the first sum in the right member of \eqref{eq:C3} by sending $\ell'-2j+1$ to $2j-1$, we obtain the third sum in the right member of \eqref{eq:Ceven}. 
Moreover, by adding and subtracting to the right member of \eqref{eq:C3} the summand indexed by $j = 1$ of the second sum of the right member of \eqref{eq:C3}, we obtain precisely the first and second sums in the right member of \eqref{eq:Ceven}.

Assume now that $\ell, \ell' > 1$ are even. 
It is clear that the term indexed by $j=1$ in the sum in the second line of \eqref{eq:gammal+1-l+2bisbisbis2}  
is minus the term of the third line of \eqref{eq:gammarbisbisbis}. 
Moreover, the sum of the term indexed by $j=1$ in the sum of the second line of \eqref{eq:gamma0-1bisbisbis}, 
and the term in the first line of \eqref{eq:gammal+1-l+2bisbisbis2} gives exactly
\[     - \kappa(\Omega) C_{1,\ell} C_{\ell',\ell''} E^{\ell,\ell',\ell''}.      \] 
Then, a direct computation using Facts \ref{fact:sif}-\ref{fact:sim} shows that 
\begin{small}
\begin{equation}
\begin{split}
     \operatorname{SI}&(n+1)_{\gan}^{\neq 2, n}(\omega) = - \kappa(\Omega) E^{\ell,\ell',\ell''}
     \\
     &\bigg(\sum_{j=1}^{\lfloor \ell'/2 \rfloor} C_{\ell,2j-1} C_{\ell'',\ell'- 2 j+2} - \sum_{j=2}^{\lfloor \ell''/2 \rfloor + 1} C_{\ell,\ell''-2j+2} C_{\ell',2 j-1} + \sum_{j=1}^{\lfloor \ell/2 \rfloor - 1} C_{\ell',\ell-2j+1} C_{\ell'',2 j} \bigg),     
\end{split}
     \end{equation}
\end{small}
\hskip -0.8mm whereas \eqref{eq:betagen} gives 
\[     \operatorname{SI}(n+1)_{\gan}^{2,n}(\omega) = \big(C_{\ell', \ell''+\ell} - C_{\ell'', \ell'+\ell} + C_{\ell,\ell'+\ell''}\big) \kappa(\Omega) E^{\ell,\ell',\ell''}.      \]
This implies that $\operatorname{SI}(n+1)_{\gan}(\omega) = 0$ (for $\ell, \ell' > 1$ even, and any $\omega$) is equivalent to 
\begin{equation}
\label{eq:C5}
\begin{split}
     &C_{\ell', \ell''+\ell} - C_{\ell'', \ell'+\ell} + C_{\ell,\ell'+\ell''} 
     \\
     &= \sum_{j=1}^{\lfloor \ell'/2 \rfloor} C_{\ell,2j-1} C_{\ell'',\ell'- 2 j+2} - \sum_{j=2}^{\lfloor \ell''/2 \rfloor + 1} C_{\ell,\ell''-2j+2} C_{\ell',2 j-1} + \sum_{j=1}^{\lfloor \ell/2 \rfloor - 1} C_{\ell',\ell-2j+1} C_{\ell'',2 j},
\end{split}
\end{equation}
which coincides with \eqref{eq:Ceven} for $\ell_{1} = \ell$, $\ell_{2} = \ell'$ and $\ell_{3}=\ell''$. 
Indeed, note that the first sum of the right member of \eqref{eq:C5} coincides with the first sum of the right member of \eqref{eq:Ceven}. 
Now, add and subtract to the right member of \eqref{eq:C5} the summand indexed by $j = \lfloor \ell/2 \rfloor$ of the third sum of the right member of \eqref{eq:C5}. 
Absorbing the negative added term in the second sum of the right member of \eqref{eq:C5}, the new second sum of the right member of \eqref{eq:C5} gives exactly the second sum of the right member of \eqref{eq:Ceven}. 
Finally, absorbing the positive added term in the third sum of the right member of \eqref{eq:C5} and reindexing it by sending $\ell-2j+1$ to $2j-1$, we get precisely the third sum of the right member of \eqref{eq:Ceven}. 

Suppose now that $\ell > 1$ is even and $\ell' > 1$ is odd. 
It is clear that the term in the first line of \eqref{eq:gamma0-1bisbisbis}
is minus the term in the first line of \eqref{eq:gammarbisbisbis}. 
Moreover, by the same reason as in the previous paragraph, the sum of the term indexed by $j=1$ in the sum of the second line of \eqref{eq:gamma0-1bisbisbis}   
and the term in the first line of \eqref{eq:gammal+1-l+2bisbisbis2} gives exactly
\[     - \kappa(\Omega) C_{1,\ell} C_{\ell',\ell''} E^{\ell,\ell',\ell''}.      \] 
Then, a direct computation using the previous facts shows that 
\begin{small}
\begin{equation}
\begin{split}
     \operatorname{SI}&(n+1)_{\gan}^{\neq 2, n}(\omega) = \kappa(\Omega) E^{\ell,\ell',\ell''}
     \\
     &\bigg(-\sum_{j=1}^{\lfloor \ell'/2 \rfloor + 1} C_{\ell,2j-1} C_{\ell'',\ell'- 2 j+2} + \sum_{j=1}^{\lfloor \ell''/2 \rfloor} C_{\ell,\ell''-2j+1} C_{\ell',2 j} + \sum_{j=0}^{\lfloor \ell/2 \rfloor - 1} C_{\ell',\ell-2j} C_{\ell'',2 j+1} \bigg),     
\end{split}
     \end{equation}
\end{small}
\hskip -0.8mm whereas \eqref{eq:betagen} gives 
\[     \operatorname{SI}(n+1)_{\gan}^{2,n}(\omega) = \big(C_{\ell', \ell''+\ell} - C_{\ell'', \ell'+\ell} - C_{\ell,\ell'+\ell''}\big) \kappa(\Omega) E^{\ell,\ell',\ell''}.      \]
This implies that $\operatorname{SI}(n+1)_{\gan}(\omega) = 0$ (for $\ell > 1$ even, $\ell' > 1$ odd, and any $\omega$) is tantamount to 
\begin{equation}
\label{eq:C6}
\begin{split}
     &C_{\ell,\ell'+\ell''} + C_{\ell'', \ell'+\ell} - C_{\ell', \ell''+\ell}
     \\
     &= \sum_{j=1}^{\lfloor \ell''/2 \rfloor} C_{\ell,\ell''-2j+1} C_{\ell',2 j} -\sum_{j=1}^{\lfloor \ell'/2 \rfloor + 1} C_{\ell,2j-1} C_{\ell'',\ell'- 2 j+2} + \sum_{j=0}^{\lfloor \ell/2 \rfloor - 1} C_{\ell',\ell-2j} C_{\ell'',2 j+1}
     \\ 
       &= \sum_{j=1}^{\lfloor \ell''/2 \rfloor} C_{\ell,2j-1} C_{\ell',\ell''-2 j+2} -\sum_{j=1}^{\lfloor \ell'/2 \rfloor + 1} C_{\ell,2j-1} C_{\ell'',\ell'- 2 j+2} + \sum_{j=1}^{\lfloor \ell/2 \rfloor} C_{\ell',\ell-2j+2} C_{\ell'',2 j-1},   
\end{split}
\end{equation}
where we have reindexed the third sum of the last member by replacing $j$ by $j-1$, and the first sum of the last member by replacing $2j$ by $\ell''-2j+2$.
This coincides with \eqref{eq:Ceven} for $\ell_{1} = \ell''$, $\ell_{2} = \ell$ and $\ell_{3}=\ell'$. 

Finally, assume that $\ell > 1$ is odd and $\ell' > 1$ is even. 
As in the previous case, the term in the first line of \eqref{eq:gamma0-1bisbisbis} is minus the term in the first line of \eqref{eq:gammarbisbisbis}. 
Moreover, the term indexed by $j=1$ in the sum of the second line of \eqref{eq:gammal+1-l+2bisbisbis2}
is minus the term in the third line of \eqref{eq:gammarbisbisbis}.  
As a consequence, 
\begin{small}
\begin{equation}
\begin{split}
     \operatorname{SI}&(n+1)_{\gan}^{\neq 2, n}(\omega) = - \kappa(\Omega) E^{\ell,\ell',\ell''}
     \\
     &\bigg(\sum_{j=1}^{\lfloor \ell'/2 \rfloor} C_{\ell,2j} C_{\ell'',\ell'- 2 j+1} + \sum_{j=2}^{\lfloor \ell''/2 \rfloor} C_{\ell,\ell''-2j+2} C_{\ell',2 j - 1} - \sum_{j=0}^{\lfloor \ell/2 \rfloor - 1} C_{\ell',\ell-2j} C_{\ell'',2 j+1} \bigg),     
\end{split}
     \end{equation}
\end{small}
\hskip -0.8mm whereas \eqref{eq:betagen} gives 
\[     \operatorname{SI}(n+1)_{\gan}^{2,n}(\omega) = \big(C_{\ell', \ell''+\ell} + C_{\ell'', \ell'+\ell} - C_{\ell,\ell'+\ell''}\big) \kappa(\Omega) E^{\ell,\ell',\ell''}.      \]
This implies that $\operatorname{SI}(n+1)_{\gan}(\omega) = 0$ (for $\ell > 1$ odd, $\ell' > 1$ even, and any $\omega$) is equivalent to
\begin{equation}
\label{eq:C7}
\begin{split}
     &C_{\ell', \ell''+\ell} + C_{\ell'', \ell'+\ell} - C_{\ell,\ell'+\ell''}
     \\
     &= \sum_{j=1}^{\lfloor \ell'/2 \rfloor} C_{\ell,2j} C_{\ell'',\ell'- 2 j+1} - \sum_{j=0}^{\lfloor \ell/2 \rfloor - 1} C_{\ell',\ell-2j} C_{\ell'',2 j+1} + \sum_{j=2}^{\lfloor \ell''/2 \rfloor} C_{\ell,\ell''-2j+2} C_{\ell',2 j - 1}
     \\
     &= \sum_{j=1}^{\lfloor \ell'/2 \rfloor} C_{\ell,\ell'-2j+2} C_{\ell'',2 j-1} - \sum_{j=1}^{\lfloor \ell/2 \rfloor} C_{\ell',\ell-2j+2} C_{\ell'',2 j-1} + \sum_{j=2}^{\lfloor \ell''/2 \rfloor} C_{\ell,\ell''-2j+2} C_{\ell',2 j - 1},
\end{split}
\end{equation}
where, in the second equality, we have reindexed the first sum of the second member by changing $2j$ by $\ell'-2j+2$, 
and also the second sum of the second member by changing $j$ by $j - 1$. 
It is clear that the last expression in \eqref{eq:C7} coincides with \eqref{eq:Ceven} for $\ell_{1} = \ell'$, $\ell_{2} = \ell''$ and $\ell_{3} = \ell$. 
Indeed, this follows from adding and subtracting the summand indexed by $j=1$ in the third sum of the last member of \eqref{eq:C7}, since the negative added term coincides exactly with the summand indexed by $j= \lfloor \ell/2 \rfloor + 1$ in the second sum of the last member of \eqref{eq:C7}.
\end{proof}

We can provide the following unified expression for the equations \eqref{eq:Codd} and \eqref{eq:Ceven}. 
It follows directly from Facts \ref{fact:codd} and \ref{fact:ceven}. 
\begin{fact}
\label{fact:cgen}
Assume the terminology we introduced in Subsubsection \ref{subsubsec:generalsetting} and after Lemma \ref{lemma:sioddl1}. 
Then, $\operatorname{SI}(n+1)_{\gan}(\omega) = 0$ (for $\ell, \ell', \ell'' \in \NN$ such that $\ell+\ell'+\ell''=n-1$ with $n \geq 6$ even, and any $\omega$ as fixed before) if and only if 
\begin{equation}
\label{eq:Cgen}
\begin{split}
     (-1)^{\ell+1} C_{\ell, \ell'+\ell''} &+ (-1)^{\ell''+1} C_{\ell'', \ell'+\ell} + (-1)^{\ell'+1} C_{\ell',\ell+\ell''} 
     \\
     &= (-1)^{\ell'+1} \sum_{j=1}^{\lfloor \ell'/2 \rfloor + \mathcalboondox{i}_{\ell} \mathcalboondox{i}_{\ell''}} C_{\ell,2j-\mathcalboondox{i}_{\ell}} C_{\ell'',\ell'- 2 j+1+\mathcalboondox{i}_{\ell}} 
    \\
    &+ (-1)^{\ell''+1} \sum_{j=1}^{\lfloor \ell''/2 \rfloor + \mathcalboondox{i}_{\ell} \mathcalboondox{i}_{\ell'}} C_{\ell',2 j-\mathcalboondox{i}_{\ell'}} C_{\ell,\ell''-2j+1+\mathcalboondox{i}_{\ell'}}  
     \\
     &+ (-1)^{\ell+1} \sum_{j=1}^{\lfloor \ell/2 \rfloor + \mathcalboondox{i}_{\ell'} \mathcalboondox{i}_{\ell''}} C_{\ell'',2j-\mathcalboondox{i}_{\ell''}} C_{\ell',\ell-2 j+1+\mathcalboondox{i}_{\ell''}}.
\end{split}
\end{equation}
\end{fact}

\begin{remark}
\label{remark:perm}
Note that \eqref{eq:Cgen} coincides with \eqref{eq:Codd} if all $\ell, \ell', \ell''$ are odd, but it is minus \eqref{eq:Ceven} 
if there is only one parameter among $\ell, \ell', \ell''$ which is odd. 
Note that equations \eqref{eq:Cgen} can \textit{a priori} be defined for any triple $(\ell,\ell', \ell'') \in \NN^{3}$ such that 
$\ell + \ell' + \ell'' = n - 1$, and $n \geq 6$ is even. 
We will denote by $\mathcalboondox{Eq}(\ell,\ell',\ell'')$ the identity given by \eqref{eq:Cgen}, indexed by $(\ell,\ell',\ell'')$. 
The expression \eqref{eq:Cgen} easily tells us that $\mathcalboondox{Eq}(\ell_{1},\ell_{2},\ell_{3})$ coincides with $\mathcalboondox{Eq}(\mathcalboondox{m}_{1},\mathcalboondox{m}_{2},\mathcalboondox{m}_{3})$ if $\{ \ell_{1},\ell_{2},\ell_{3} \} = \{ \mathcalboondox{m}_{1},\mathcalboondox{m}_{2},\mathcalboondox{m}_{3} \}$. 
\end{remark}

After these preparations we are now able to prove Lemma \ref{lemma:sioddl1}. 

\begin{proof}[Proof of Lemma \ref{lemma:sioddl1}]
By Fact \ref{fact:cgen}, the expression of $\operatorname{SI}(n+1)_{\gan}(\omega)$ in Lemma \ref{lemma:sioddl1} (where $\ell, \ell', \ell'' \geq 1$) vanishes if and only if the constants $C_{i,j}$ satisfy the identities \eqref{eq:Cgen}. 
It remains to prove that the latter are verified. 
To simplify our notation, we will write $n = 2 (k +1)$, with $k \geq 2$. 
Recall that we will denote by $\mathcalboondox{Eq}(\ell,\ell',\ell'')$ the identity \eqref{eq:Cgen} indexed by $(\ell,\ell',\ell'')$.
To prove that $\mathcalboondox{Eq}(\ell,\ell',\ell'')$ holds we will show that it can obtained as a linear combination of 
the identities appearing in \cite{BCM17}, Thm. 2.4. 

Let $a, b, c \in \NN_{0}$ such that $a + b+ c = 2k - 1$. 
We recall the identity on the Bernoulli numbers appearing in \cite{BCM17}, Thm. 2.4, that we will denote by $\operatorname{Eq}(a,b,c)$, which is of the form
\begin{equation}
\label{eq:bcvm}
     -\mu_{0} \mathcal{B}_{2k} = \frac{\mu_{k}}{2} \mathcal{B}_{k}^{2} + \sum_{j=1}^{\lfloor (k-1)/2 \rfloor} \mu_{2j} \mathcal{B}_{2 j} \mathcal{B}_{2k - 2 j},    
\end{equation}
where $\mu_{2j} = \mu_{2j}(a,b,c)$ is given by 
\begin{scriptsize}
\begin{equation}
\label{eq:bcvmc}
\begin{split}
     \begin{pmatrix} 2k \\ 2j \end{pmatrix} &\bigg[ (-1)^{c} \begin{pmatrix} 2k- 2 j \\ c \end{pmatrix} \hskip -0.8em \sum_{d=\operatorname{max}(0,2j-b)}^{\operatorname{min}(a,2j)} \hskip -0.8em (-1)^{d} \begin{pmatrix} 2j \\ d \end{pmatrix} 
     + (-1)^{c} \begin{pmatrix} 2 j \\ c \end{pmatrix} \hskip -0.8em \sum_{d=\operatorname{max}(0,2k-2j-b)}^{\operatorname{min}(a,2k-2j)} \hskip -0.8em (-1)^{d} \begin{pmatrix} 2k-2j \\ d \end{pmatrix}
     \\
     &- (-1)^{a} \begin{pmatrix} 2k - 2 j \\ a \end{pmatrix} \hskip -0.8em \sum_{d=\operatorname{max}(0,2j-b)}^{\operatorname{min}(c,2j)} \hskip -0.8em (-1)^{d} \begin{pmatrix} 2j \\ d \end{pmatrix} 
     - (-1)^{a} \begin{pmatrix} 2 j \\ a \end{pmatrix} \hskip -0.8em \sum_{d=\operatorname{max}(0,2k-2j-b)}^{\operatorname{min}(c,2k-2j)} \hskip -0.8em (-1)^{d} \begin{pmatrix} 2k-2j \\ d \end{pmatrix} \bigg], 
     \end{split}
\end{equation}
\end{scriptsize}
\hskip -0.8mm for all $j \in \llbracket 0, \lfloor k/2 \rfloor \rrbracket$.
Note that $n$ in \cite{BCM17}, Thm. 2.4, coincides with $2 k$ in our notation, and the index $k$ in \cite{BCM17}, Thm. 2.4, is our $2j$. 
Notice also that the first term of the right member of \eqref{eq:bcvm} vanishes if $k \geq 2$ is odd since $\mathcal{B}_{k}= 0$ in that case. 

We will first reduce the expression \eqref{eq:bcvmc}, to be able to handle it. 
We first denote $\mu_{2j}^{\circ}(a,b,c)$ the quotient of $\mu_{2j}(a,b,c)$ by the binomial coefficient $C(2k,2j)$. 
Note that \eqref{eq:bcvmc} tells us directly that 
\begin{equation}
\label{eq:bcvmc0}
\begin{split}
     \mu_{0}^{\circ}(a,b,c) = (-1)^{c} \begin{pmatrix} 2k \\ c \end{pmatrix} - (-1)^{a} \begin{pmatrix} 2k \\ a \end{pmatrix}. 
     \end{split}
\end{equation}

On the other hand, using \eqref{eq:binpasgen} as well as the comments below it, together with \eqref{eq:binids}, we see that $\mu_{2j}^{\circ} = \mu_{2j}^{\circ} (a,b,c)$ coincides with 
\begin{equation}
\label{eq:bcvmc2}
\begin{split}
     &(-1)^{a + c} \bigg\{\bigg[ \begin{pmatrix} 2k- 2 j \\ c \end{pmatrix} \begin{pmatrix} 2 j -1 \\ a \end{pmatrix} - \begin{pmatrix} 2k- 2 j - 1 \\ c \end{pmatrix} \begin{pmatrix} 2 j \\ a \end{pmatrix}\bigg] 
     \\
     &\phantom{(-1)^{a + c} \bigg\{}+ \bigg[ \begin{pmatrix} 2 j \\ c \end{pmatrix} \begin{pmatrix} 2k-2 j -1 \\ a \end{pmatrix} - \begin{pmatrix} 2j - 1 \\ c \end{pmatrix} \begin{pmatrix} 2k - 2 j \\ a \end{pmatrix}\bigg]\bigg\}
     \\
     &+(-1)^{b + c} \bigg\{ \begin{pmatrix} 2k- 2 j \\ c \end{pmatrix} \begin{pmatrix} 2 j -1 \\ b \end{pmatrix} + \begin{pmatrix} 2 j \\ c \end{pmatrix} \begin{pmatrix} 2k-2 j -1 \\ b \end{pmatrix}\bigg\} 
     \\
     &- (-1)^{a+b} \bigg\{ \begin{pmatrix} 2k - 2j \\ a \end{pmatrix} \begin{pmatrix} 2j - 1 \\ b \end{pmatrix} + \begin{pmatrix} 2 j \\ a \end{pmatrix} \begin{pmatrix} 2k-2 j -1 \\ b \end{pmatrix} \bigg\}
     \end{split}
\end{equation}
for all $j \in \llbracket 1, \lfloor k/2 \rfloor \rrbracket$. 
Furthermore, applying \eqref{eq:binid} to the previous expression, we obtain that $\mu_{2j}^{\circ} = \mu_{2j}^{\circ}(a,b,c)$ is given by
\begin{small}
\begin{equation}
\label{eq:bcvmc3}
\begin{split}
     &(-1)^{a + c} \bigg\{ \begin{pmatrix} 2k- 2 j - 1 \\ c - 1 \end{pmatrix} \bigg[ \frac{2 k - 2 j}{c} \begin{pmatrix} 2 j -1 \\ a \end{pmatrix}  - \frac{2 k - 2 j - c}{c} \begin{pmatrix} 2 j \\ a \end{pmatrix}\bigg]
     \\
     &\phantom{ (-1)^{a + c} \bigg\{ }+ \begin{pmatrix} 2 j - 1 \\ c - 1 \end{pmatrix} \bigg[ \frac{2j}{c} \begin{pmatrix} 2k-2 j -1 \\ a \end{pmatrix} - \frac{2j-c}{c} \begin{pmatrix} 2k-2 j \\ a \end{pmatrix} \bigg]\bigg\}
     \\
     &+(-1)^{b + c} \bigg\{ \begin{pmatrix} 2k- 2 j-1 \\ c - 1 \end{pmatrix} \begin{pmatrix} 2 j -1 \\ b \end{pmatrix} \frac{2k-2j}{c} + \begin{pmatrix} 2 j - 1 \\ c - 1 \end{pmatrix} \begin{pmatrix} 2k-2 j -1 \\ b \end{pmatrix} \frac{2j}{c} \bigg\} 
     \\
     &- (-1)^{a+b} \bigg\{ \begin{pmatrix} 2k - 2j - 1 \\ b \end{pmatrix} \begin{pmatrix} 2j \\ a \end{pmatrix} + \begin{pmatrix} 2 j - 1 \\ b \end{pmatrix} \begin{pmatrix} 2k-2 j \\ a \end{pmatrix} \bigg\},
     \end{split}
\end{equation}
\end{small}
\hskip -0.8mm for all $j \in \llbracket 1, \lfloor k/2 \rfloor \rrbracket$, whenever $c > 0$. 

For later use, we let the reader check that 
\begin{equation}
\label{eq:maincomp}
\begin{split}
 &c \mu_{0}^{\circ}(a,b,c) + (b+1) \mu_{0}^{\circ}(a,c-1,b+1)
 \\
 &= 2k \bigg[ (-1)^{c} \begin{pmatrix} 2k-1 \\ c - 1 \end{pmatrix} - (-1)^{b} \begin{pmatrix} 2k-1 \\ b \end{pmatrix} 
 - (-1)^{a} \begin{pmatrix} 2k-1 \\ a \end{pmatrix} \bigg],  
 \end{split}
\end{equation}
as well as
\begin{equation}
\label{eq:maincomp2}
\begin{split}
 &c \mu_{2j}^{\circ}(a,b,c) + (b+1) \mu_{2j}^{\circ}(a,c-1,b+1)
 \\
 &= 2k \bigg[ (-1)^{a+c} \bigg\{ \begin{pmatrix} 2k- 2 j-1 \\ c - 1 \end{pmatrix} \begin{pmatrix} 2 j -1 \\ a \end{pmatrix} + \begin{pmatrix} 2 j - 1 \\ c - 1 \end{pmatrix} \begin{pmatrix} 2k-2 j -1 \\ a \end{pmatrix} \bigg\}
 \\
 &+ (-1)^{b+c} \bigg\{ \begin{pmatrix} 2k- 2 j-1 \\ c - 1 \end{pmatrix} \begin{pmatrix} 2 j -1 \\ b \end{pmatrix}  + \begin{pmatrix} 2 j - 1 \\ c - 1 \end{pmatrix} \begin{pmatrix} 2k-2 j -1 \\ b \end{pmatrix} \bigg\}
 \\
 &-(-1)^{a+b} \bigg\{ \begin{pmatrix} 2k- 2 j-1 \\ b \end{pmatrix} \begin{pmatrix} 2 j -1 \\ a \end{pmatrix} + \begin{pmatrix} 2 j - 1 \\ b \end{pmatrix} \begin{pmatrix} 2k-2 j -1 \\ a \end{pmatrix} \bigg\} \bigg]
 \end{split}
\end{equation}
for all $j \in \llbracket 1, \lfloor k/2 \rfloor \rrbracket$, if $c > 0$. 
Identity \eqref{eq:maincomp} follows directly from \eqref{eq:bcvmc0} together with \eqref{eq:binid}, whereas \eqref{eq:maincomp2} follows from \eqref{eq:bcvmc3} combined with \eqref{eq:binid}. 

We will now prove identities $\mathcalboondox{Eq}(\ell,\ell',\ell'')$ given by \eqref{eq:Cgen}, for $\ell, \ell', \ell'' \geq 1$. 
It will be more convenient to change the name of the parameters slightly, to simplify some expressions. 
Let $(\ell_{1},\ell_{2},\ell_{3}) \in \NN^{3}$ such that $\ell_{1} + \ell_{2} + \ell_{3} = 2k + 1$, and $k \geq 2$ is an integer. 
Assume $\ell_{1},\ell_{2},\ell_{3} \geq 1$. 
We consider now a generic expression of one of the three sums in the right member of \eqref{eq:Cgen}. 
More precisely, define $\mathscr{E}(\ell_{1},\ell_{2},\ell_{3})$ as
\begin{small}
\begin{equation}
\label{eq:l1l2l3} 
\begin{split}
\sum_{j=1}^{\lfloor \ell_{2}/2 \rfloor + \mathcalboondox{i}_{\ell_{1}} \mathcalboondox{i}_{\ell_{3}}} &C_{\ell_{1},2j-\mathcalboondox{i}_{\ell_{1}}} C_{\ell_{3},\ell_{2}- 2 j+1+\mathcalboondox{i}_{\ell_{1}}} 
\\
&= (-1)^{k} \sum_{j=1}^{\lfloor \ell_{2}/2 \rfloor + \mathcalboondox{i}_{\ell_{1}} \mathcalboondox{i}_{\ell_{3}}} \begin{pmatrix} \ell_{1}+2j-2 - \mathcalboondox{i}_{\ell_{1}} \\ \ell_{1}\ - 1 \end{pmatrix} 
\begin{pmatrix} \ell_{2}+\ell_{3}-2j-1 + \mathcalboondox{i}_{\ell_{1}} \\ \ell_{3}\ - 1 \end{pmatrix} 
\\
&\phantom{= (-1)^{k} \sum_{j=1}^{\lfloor \ell_{2}/2 \rfloor + \mathcalboondox{i}_{\ell_{1}} \mathcalboondox{i}_{\ell_{3}}} xx}
\frac{\mathcal{B}_{\ell_{1}+2j-1-\mathcalboondox{i}_{\ell_{1}}}}{(\ell_{1}+2 j -1-\mathcalboondox{i}_{\ell_{1}})!} 
 \frac{\mathcal{B}_{\ell_{2}+\ell_{3}-2j+\mathcalboondox{i}_{\ell_{1}}}}{(\ell_{2}+\ell_{3}-2 j+\mathcalboondox{i}_{\ell_{1}})!}
\\
&= (-1)^{k} \sum_{i=(\ell_{1}+\mathcalboondox{p}_{\ell_{1}})/2}^{k-(\ell_{3}+\mathcalboondox{p}_{\ell_{3}})/2} \begin{pmatrix} 2i-1 \\ \ell_{1}\ - 1 \end{pmatrix} 
\begin{pmatrix} 2k-2i-1 \\ \ell_{3}\ - 1 \end{pmatrix} 
\frac{\mathcal{B}_{2i}}{(2 i)!} \frac{\mathcal{B}_{2k-2i}}{(2k-2 i)!},
\end{split}
\end{equation}
\end{small}
\hskip -0.8mm where we have used \eqref{eq:newcij} in the first equality, and we have reindexed the last sum by means of $2 i = \ell_{1} + 2j -1 - \mathcalboondox{i}_{\ell_{1}}$. 
Note that 
\[     k - (\ell_{3}+\mathcalboondox{p}_{\ell_{3}})/2 = (\ell_{1}+\ell_{2}-1-\mathcalboondox{p}_{\ell_{3}})/2 \geq (\ell_{1}+\mathcalboondox{p}_{\ell_{1}})/2 \text{ if and only if } \ell_{2} \geq 1+\mathcalboondox{p}_{\ell_{1}}+\mathcalboondox{p}_{\ell_{3}},     \] 
which is always verified except for $\ell_{2}=1$ and $\ell_{1}, \ell_{3}$ odd. 
In the latter case, \eqref{eq:l1l2l3} trivially vanishes. 
Taking into account that a binomial coefficient vanishes if the nonnegative numerator is strictly smaller than the denominator, 
\eqref{eq:l1l2l3} implies that 
\begin{equation}
\label{eq:l1l2l3ex} 
 \mathscr{E}(\ell_{1},\ell_{2},\ell_{3}) = (-1)^{k} \sum_{i=1}^{k-1} \begin{pmatrix} 2i-1 \\ \ell_{1}\ - 1 \end{pmatrix} 
\begin{pmatrix} 2k-2i-1 \\ \ell_{3}\ - 1 \end{pmatrix} 
\frac{\mathcal{B}_{2i}}{(2 i)!} \frac{\mathcal{B}_{2k-2i}}{(2k-2 i)!}.
\end{equation}
Note that each summand in the right member of \eqref{eq:l1l2l3ex} vanishes if $\ell_{2}=1$ and $\ell_{1}, \ell_{3}$ odd, 
since at least one of the denominators of the two binomial coefficients is strictly greater than the corresponding nonnegative numerator. 
In consequence, \eqref{eq:l1l2l3ex} holds for all $\ell_{1},\ell_{2},\ell_{3} \geq 1$. 

We now decompose the sum in \eqref{eq:l1l2l3ex} for $i \in \llbracket 1, \lfloor \frac{k-1}{2} \rfloor \rrbracket$, $i = k/2$ if $k$ is even, and $i \in \llbracket \lceil \frac{k+1}{2} \rceil , k - 1 \rrbracket$. 
If we further reindex the sum corresponding to $i \in \llbracket \lceil \frac{k+1}{2} \rceil , k - 1 \rrbracket$ by replacing $i$ 
by $k-i$, we get that $\mathscr{E}(\ell_{1},\ell_{2},\ell_{3})$ is
\begin{equation}
\label{eq:l1l2l3bis1final} 
\begin{split}
&(-1)^{k} \sum_{i=1}^{\lfloor \frac{k-1}{2} \rfloor} \begin{pmatrix} 2i-1 \\ \ell_{1}\ - 1 \end{pmatrix} 
\begin{pmatrix} 2k-2i-1 \\ \ell_{3}\ - 1 \end{pmatrix} 
\frac{\mathcal{B}_{2i}}{(2 i)!} \frac{\mathcal{B}_{2k-2i}}{(2k-2 i)!}
+ \begin{pmatrix} k-1 \\ \ell_{1}\ - 1 \end{pmatrix} 
\begin{pmatrix} k-1 \\ \ell_{3}\ - 1 \end{pmatrix} 
\frac{\mathcal{B}_{k}^{2}}{(k!)^{2}}
\\
&+(-1)^{k} \sum_{i=1}^{\lfloor \frac{k-1}{2} \rfloor} \begin{pmatrix} 2i-1 \\ \ell_{3}\ - 1 \end{pmatrix} 
\begin{pmatrix} 2k-2i-1 \\ \ell_{1}\ - 1 \end{pmatrix} 
\frac{\mathcal{B}_{2i}}{(2 i)!} \frac{\mathcal{B}_{2k-2i}}{(2k-2 i)!},
\end{split}
\end{equation}
where the second term in the first line vanishes if $k \geq 2$ is odd, since $\mathcal{B}_{k}= 0$ in that case. 

We then conclude that \eqref{eq:Cgen} for $\ell_{1}, \ell_{2}, \ell_{3} \geq 1$ is $(-1)^{k}$ times 
\begin{equation}
\label{eq:Cgenex} 
\begin{split}
     &\sum_{p=1}^{3} (-1)^{\ell_{p}} \begin{pmatrix} 2 k - 1 \\ \ell_{p} - 1 \end{pmatrix} \frac{\mathcal{B}_{2k}}{(2 k)!}
     \\
     &= \sum_{p=1}^{3} (-1)^{\ell_{p}+1} \bigg[ \sum_{i=1}^{\lfloor \frac{k-1}{2} \rfloor} \bigg(\begin{pmatrix} 2i-1 \\ \ell_{q} - 1 \end{pmatrix} 
\begin{pmatrix} 2k-2i-1 \\ \ell_{r}\ - 1 \end{pmatrix} 
\\
&\phantom{= \sum_{p=1}^{3} (-1)^{\ell_{p}+1} \bigg[ \sum_{i=1}^{\lfloor \frac{k-1}{2} \rfloor} \bigg(}+ \begin{pmatrix} 2i-1 \\ \ell_{r} - 1 \end{pmatrix} 
\begin{pmatrix} 2k-2i-1 \\ \ell_{q}\ - 1 \end{pmatrix} \bigg)
\frac{\mathcal{B}_{2i}}{(2 i)!} \frac{\mathcal{B}_{2k-2i}}{(2k-2 i)!} 
\\
&\phantom{= \sum_{p=1}^{3} (-1)^{\ell_{p}+1} \bigg[ xx }+ \begin{pmatrix} k-1 \\ \ell_{q}\ - 1 \end{pmatrix} 
\begin{pmatrix} k-1 \\ \ell_{r}\ - 1 \end{pmatrix} 
\frac{\mathcal{B}_{k}^{2}}{(k!)^{2}}\Bigg],
\end{split}
\end{equation}
where the indices $q, r \in \{ 1, 2, 3\}$ are taken so that $\{ 1,2,3\} = \{ p,q,r\}$ and $q < r$, and we note that the last term vanishes if $k \geq 2$ is odd since $\mathcal{B}_{k}= 0$ in that case. 

We now claim that $\mathcalboondox{Eq}(\ell,\ell',\ell'')$ is exactly given as 
\begin{equation}
\label{eq:ide}
     - \frac{(-1)^{k}}{(2k)!} \bigg( \frac{\ell}{2 k} \operatorname{Eq}(\ell''-1,\ell'-1,\ell) + \frac{\ell'}{2 k} \operatorname{Eq}(\ell''-1,\ell-1,\ell') \bigg),
\end{equation}
for all $\ell, \ell', \ell'' \geq 1$, where we recall that $2 k + 1 = \ell + \ell' + \ell''$ for an integer $k \geq 2$. 
This follows directly from \eqref{eq:maincomp}, \eqref{eq:maincomp2} and \eqref{eq:Cgenex}. 
Since the equations on \eqref{eq:ide} are verified, due to \cite{BCM17}, Thm. 2.4, equation $\mathcalboondox{Eq}(\ell,\ell',\ell'')$ 
is verified as well. 
The lemma is thus proved.
\end{proof}



\begin{small}
\vskip 0.5cm
\noindent \scshape{David Fern\'andez: Fakult\"at f\"ur Mathematik, Universit\"at Bielefeld, 
\\
33501 Bielefeld, Germany.}
\\
\normalfont\textit{E-mail address}: \href{mailto:David.Fernandez@math.uni-bielefeld.de}{\texttt{david.fernandez@math.uni-bielefeld.de}}.

\medskip 

\noindent \scshape{Estanislao Herscovich: Institut Fourier, Universit\'e Grenoble Alpes, 38610 Gi\`eres, France.}
\\
\normalfont\textit{E-mail address}: \href{mailto:Estanislao.Herscovich@univ-grenoble-alpes.fr}{\texttt{Estanislao.Herscovich@univ-grenoble-alpes.fr}}. 
\end{small}
\end{document}